\documentclass[11pt]{article}
\usepackage[utf8]{inputenc}
\usepackage{amsmath}
\usepackage{placeins}
\usepackage{authblk}
\usepackage{amssymb}
\usepackage[english]{babel}
\usepackage{hyperref}
\usepackage{apacite}
\usepackage{ragged2e}
\usepackage{enumerate}
\usepackage{enumitem}
\usepackage{authblk}
\usepackage{graphicx}
\usepackage{float}
\usepackage{geometry}
\usepackage{pifont}
\usepackage{subcaption}
\usepackage{booktabs}
\usepackage{setspace}
\usepackage{xcolor}
\usepackage{multirow}
\usepackage{array}
\hypersetup{colorlinks=true,linkcolor=blue,anchorcolor=blue,citecolor=blue}

\newtheorem{remark}{Remark}
\newtheorem{theorem}{Theorem}
\newtheorem{lemma}{Lemma}
\newtheorem{proposition}{Proposition}

\newtheorem{assumption}{Assumption}
\newtheorem{corollary}{Corollary}

\newcommand{\e}{\mathbf{e}}

\newcommand{\M}{\mathbf{M}}
\newcommand{\B}{\mathbf{B}}
\newcommand{\A}{\mathbf{A}}

\renewcommand{\P}{\mathbb{P}}

\newcommand{\E}{\mathbb{E}}

\begin{document}
\title{Safe and Sharp Honest Inference for Nonparametric Smoothing\\ via Empirical Bernstein Calibration}
\author[1]{Zihao Yuan}

\author[2]{Sven Klaassen}

\author[3]{Holger Dette}

\affil[1]{University of Hamburg, Germany}

\affil[2]{Kiel University, Germany}

\affil[3]{Ruhr University Bochum, Germany}
\date{}
\maketitle
\begin{abstract}
Constructing honest confidence intervals for nonparametric smoothers is a fundamental task in nonparametric econometrics, and it often depends on reliable bias control and accurate calibration based on asymptotic normality, including standard-normal and folded-normal calibration. Substantial progress has been made to correct or control the smoothing bias, including robust bias correction and bias-aware inference. We first show that, even after smoothing bias has been well corrected or controlled, asymptotic-normality-based calibration may still be a binding source of finite-sample undercoverage. Thus, the resulting intervals may struggle to achieve the minimax shrinkage rate and uniformly small undercoverage error simultaneously. Instead of using distributional approximation, we calibrate the radius directly by combining an empirical Bernstein bound, a data-driven variance proxy, Lepski-type bandwidth selection, and a bias-aware fixed-length-radius criterion. The formal theory covers nonparametric regression and density estimation, with regression results ranging from local-polynomial to sieve estimators. The resulting empirical Bernstein confidence intervals are safe and sharp. Uniformly over functions with $S$-th order local smoothness, both one-sided and two-sided intervals attain nominal coverage up to $o(n^{-2S/(2S+1)})$, or exponential remainders under bounded or sub-Gaussian conditions, while their widths shrink at the minimax rate $n^{-S/(2S+1)}$ (or up to a $\sqrt{\log n}$-level factor). The calibration principle is modular and can also be combined with other existing bias-control strategies, like robust bias correction. Thus, the contribution of this paper is not a bias-control device but a new angle of calibration. Compared with asymptotic-normality-based calibration, empirical Bernstein calibration safely and conveniently converts the specified smoothness into both coverage accuracy and interval-length efficiency. Simulations support the theory.
\end{abstract}


\newpage
\section{Introduction}
\label{sec 1}

Calibration of an honest confidence interval means choosing, for a fixed level  \(\alpha\in(0,1)\), an \(\alpha\)-dependent critical value that yields an interval with coverage probability at least \(1-\alpha\). Constructing such honest confidence intervals for target functions using nonparametric smoothers, such as kernel, local-polynomial, and more general weighted-average estimators, is a central problem in nonparametric econometrics and causal inference (e.g.,  \citeA{ullah1999nonparametric}, \citeA{Tsybakov2009}, \citeA{calonico2022coverage}, and \citeA{noack2024bias}). When the target function is \(S\)-times differentiable at the evaluation point, a natural objective is to exploit this smoothness fully, in the sense that the length of the interval shrinks at the minimax optimal rate \(n^{-{S}/{(2S+1)}}\) (see  \citeA{low1997nonparametric}), where \(n\) is the sample size. The key question is then whether this efficiency can be achieved together with reliable coverage guarantees. We refer to this goal as {\it safely exhausting the assumed smoothness}.

\paragraph{Our objectives}
 Given a smooth function $\theta \colon [a,b] \to \mathbb{R}$, a level \(\alpha\in(0,1)\)  and an evaluation point $x_0 \in [a,b]$, we aim to develop a calibration method for confidence intervals that simultaneously satisfies the following two requirements.

\begin{enumerate}[label=\textbf{G\arabic*}, ref=G\arabic*]
    \item \label{G1} \textbf{Uniform safety.}
 Let \({\cal F}_{S,n}(x_0)\) be a class of $S$-order locally differentiable functions specified through an envelope control of the \(S\)-order Taylor remainder. 
The proposed one- and two-sided honest confidence intervals, say $\mathrm{CI}$, satisfy
\begin{align}
\label{eq G1}
    \inf_{\theta\in\mathcal F_{S,n}(x_0)}
P_\theta\{\theta(x_0)\in \mathrm{CI}\}
\ge 1-\alpha-\mathrm{CE}_n 
\end{align}
where \((\mathrm{CE}_n)_{n \geq 1} \) is a sequence satisfying \(|\mathrm{CE}_n|  \cdot n^{\frac{2S}{2S+1}}\to 0\) as $n\to\infty$.

    \item \label{G2} \textbf{Minimax-rate sharpness.}
    The lengths of the confidence intervals attain the minimax-optimal shrinking rate \footnote{Here and throughout, the minimax rate refers to the sharp polynomial
benchmark when only $S$-th order local differentiability is specified. No specific assumption for the local modulus of continuity of the $S$-th order derivative is used.}. Specifically, there exists a constant \(c>0\) such that
    \begin{align}
    \label{eq G2}
        n^{\frac{S}{2S+1}}
        \times
        \text{length of CI}
        \xrightarrow[n\to\infty]{\mathbb{P}}c.
    \end{align}
Here, for a one-sided confidence interval of the form, $[\hat \theta (x_0) - r_n, \infty ) $ or $( -\infty, \hat \theta (x_0) + r_n ] $,  its length is defined by $r_n$, where $\hat \theta (x_0)$ denotes an estimate of $\theta (x_0)$.
\end{enumerate}
As will be discussed in Section \ref{sec 2.1} and Proposition \ref{prop RE cov error},  when we use standard-normal critical values to calibrate the confidence interval, achieving \ref{G1} and \ref{G2} faces a fundamental tension. The same normalization that yields an asymptotic distribution also inflates the smoothing bias. As a result, a bias that is negligible for estimation may become non-negligible for inference, and thus may damage the credibility of statistical inference. Meanwhile, by using a higher-order normal approximation analysis, as, for example, provided for kernel smoothers in \citeA{hall2013bootstrap}, we can show that the difficulty of achieving \ref{G1} is not solely due to bias. According to the discussion in Chapter 2 of \citeA{hall2013bootstrap}, even when the bias of a kernel smoother has been subtracted exactly, a symmetric two-sided confidence interval based on standard normal calibration (SNC) and bias correction still fails to satisfy \ref{G1}. This problem is even more pronounced for one-sided confidence intervals since one-sided inference lacks the symmetry that often yields a cancellation of first-order error terms. Another noteworthy point is that, in many statistical applications, the intervals are used as intermediate devices for testing, screening, and decision making. In such settings, undercoverage is not only a coverage distortion of the reported interval, but also directly determines the size (or Type-1 error) of a test or decision. Requirement \ref{G1} therefore captures a substantive reliability demand, while \ref{G2} ensures that this reliability is achieved without sacrificing the minimax order of interval length.

In the sense of the minimax shrinking rate of honest nonparametric confidence intervals, the requirement \ref{G2} corresponds to minimax-rate sharpness
under $S$-smoothness (see \citeA{donoho1994statistical}, \citeA{low1997nonparametric}). \citeA{susanne} proposed a method for constructing confidence intervals attaining \ref{G2} which can be interpreted as augmenting a normal-approximation interval by an estimated bias bound. The {\it bias-aware inference} introduced in \citeA{armstrong2018optimal,armstrong2020simple} achieves  \ref{G2} as well by elegantly using the property of the folded normal distribution (or Chi-square distribution). Moreover, these authors introduced a fixed-length confidence interval (FLCI) whose efficiency is asymptotically optimal with respect to both constant and shrinking rates.  However, both of these two approaches leave the rate-explicit undercoverage control in requirement \ref{G1} open.

\paragraph{Our contributions}
 We propose a calibration approach for one-sided and two-sided confidence intervals based on a novel combination of the tail control by Bernstein inequalities and the radius optimization from fixed-length confidence intervals (FLCIs) of \citeA{armstrong2018optimal,armstrong2020simple}. The resulting confidence intervals are thus addressed as {\it empirical Bernstein confidence intervals}  (EBCIs). More specifically, our contributions are the following:
\begin{enumerate}[label=\textbf{C\arabic*}, ref=C\arabic*]
    \item\label{C1}We show that \ref{G1} and \ref{G2} are difficult to achieve simultaneously for confidence intervals that combine bias-control strategies with asymptotic-normality-based calibration. More specifically, for intervals based on standard-normal critical-value calibration (SNC) after bias correction, we identify in Proposition \ref{prop RE cov error}  two important bottlenecks: the substantial undercoverage caused by remaining normalized bias and the length--coverage trade-off. For intervals using folded-normal critical-value calibration after worst-case bias control, we show that the calibration step can inherit a first-order Edgeworth-type coverage term. This effect remains visible even under oracle knowledge of the absolute bias and normalizing factor. In particular, this may cause a significant effect for nonparametric regressions with highly skewed disturbances, which could widely exist in finance and risk management. (Proposition~\ref{prop FLCI normal approximation}).

    \item \label{C2} 
    We develop an empirical Bernstein calibration that bypasses the difficulties described in \ref{C1}. The resulting EBCIs are “safe” and “sharp” in the sense that they achieve \ref{G1} and \ref{G2} (or \ref{G2} up to a logarithmic factor) for a broad class of both local and global linear smoothers. These include local polynomial regression and kernel density estimation under local smoothness conditions, as well as widely used linear sieve estimators of regression functions, such as B-splines and wavelets. In the regression setting, the proposed methods allow for heteroskedasticity and require only moments of the disturbances of order slightly larger than four. For more details, we refer to Theorems \ref{th refinement}-\ref{th eta free density} and their discussions.

\item \label{C3}
We demonstrate that, except for rate, the leading constants of EBCIs are first-order aligned with those of FLCIs in the small-$\alpha$ regime (Proposition \ref{prop EBCI FLCI sharpness}). This indicates that the premium for EBCIs to be “safe” is light, which highlights that EBCIs and FLCIs are complementary bias-aware procedures. EBCI is attractive when a transparent finite-sample coverage guarantee is paramount. Whereas FLCI is preferable when the normal approximation is very accurate. Moreover, we show that empirical Bernstein calibration is also a modular tool and, except for worst-case bias control, can also be combined with robust bias corrections (see Proposition \ref{prop RBC EBCI} and its discussion).


\end{enumerate}

\paragraph{Literature Review}
 We review key literature on two main approaches to calibrating confidence intervals for local and global smoothers in nonparametric econometrics using normal approximations: standard-normal critical-value calibration (SNC) after bias correction, and folded-normal critical-value calibration (FNC) after worst-bias control. The main difference between these two approaches lies in how they account for the inferential bias, typically defined as the {\it estimation bias divided by the standard-error scale. This difference leads to distinct limiting distributions and calibration principles.}

For bias-correction-based SNC, a common approach is to investigate and correct the leading bias. Robust bias correction, as introduced by \citeA{calonico2018effect} and \citeA{calonico2022coverage}, and Richardson-extrapolation-type estimators, as considered for example by  \citeA{schucany1977improvement} and \citeA{zhou2010simultaneous}, are two prominent strategies.  Robust bias correction estimates and subtracts leading bias, while using a studentization that accounts for the first-order randomness introduced by bias estimation.  Richardson-extrapolation-type estimators instead form linear combinations of estimators computed at different bandwidths, so that the leading bias term cancels algebraically.  In both cases, the purpose is to replace the original normalized bias with a higher-order residual bias. Note that this replacement does not remove the role of the bias in coverage accuracy. It only changes its order. After the leading term has been removed, the remaining normalized bias may still determine the coverage loss. The coverage-error expansions developed by \citeA{calonico2018effect} provide a detailed and powerful analysis of this phenomenon for SNC after robust bias correction.

FNC after worst bias control takes a different route. It focuses directly on the honesty of the confidence interval (see, for example, \citeA{li1989honest}) by requiring the coverage guarantee to hold uniformly over a prescribed class of functions. Hence, in contrast to calibration based on SNC after bias correction, this approach only requires the limiting coverage probability of the confidence interval to be no less than $1-\alpha$, for some fixed $\alpha\in (0,1)$. Building on the fixed-length CI construction of \citeA{donoho1994statistical}, \citeA{armstrong2018optimal,armstrong2020simple} develop this approach for kernel and local-polynomial estimators in nonparametric regression. To describe this calibration principle, let \(t_{\mathcal F}(h)\) denote an upper bound on the worst-case normalized bias over a function class \(\mathcal F\) and $h$ denote the bandwidth. Rather than requiring \(t_{\mathcal F}(h)\to 0\), they construct a two-sided CI by multiplying the standard error by the \(1-\alpha\) quantile of the folded-normal distribution \(\lvert {\cal N}(t_{\mathcal F}(h),1)\rvert\). This folded-normal calibration explicitly accommodates a non-negligible worst-case bias and delivers asymptotic honesty uniformly over \(\mathcal F\). They further derive bandwidth choices and efficiency comparisons for FLCIs, showing that a CI centered at an RMSE-optimal local-polynomial estimator can be nearly as efficient as one using the length-optimal bandwidth. In the common local-linear case with rate \(n^{-2/5}\), the resulting \(0.95\) FLCI uses a critical value of approximately \(2.18\), rather than the standard-normal value \(1.96\).

\section{From Standard-Normal to Empirical-Bernstein Calibration}
\label{sec 2}

\subsection{A Small Noise is NOT Always Beneficial}
\label{sec 2.1}
Pointwise confidence intervals for nonparametric regression functions underpin many fundamental causal-inference problems, including regression-discontinuity effects at the cutoff (e.g., \citeA{calonico2014robust}) and conditional average treatment effect inference at covariate values of substantive interest (see, for example, \citeA{rubin1974estimating}, \citeA{abrevaya2015estimating}, \citeA{athey2016recursive}, \citeA{wager2018estimation}, and \citeA{fan2022estimation}). At the same time, a standard objective of experimental design is to improve precision by limiting extraneous variation through experimental control, blocking, and covariate adjustment (see, for example, \citeA{cochran1957analysis}, \citeA{bruhn2009pursuit}, and \citeA{lin2013agnostic}). As emphasized by \citeA{falk2009lab}, this objective is particularly salient in laboratory and behavioral-economics experiments, where controlled variation is often viewed as a central advantage of the experimental environment. This conventional precision logic raises a fundamental question:
\begin{align*}
\textit{Is smaller regression noise always beneficial for reliable finite-sample pointwise inference?}
\end{align*}
We will illustrate by a numerical example  that this intuition can fail when a standard-normal critical value calibrates the confidence interval after bias correction. More specifically, we consider the model 
\begin{align} \label{example1}
    Y=&0.1+\sum_{j=1}^S jx^j+(S+1)\max(x,0)^{S+0.1} +e,
\end{align}
where the predictor $X$ is uniformly distributed on the interval $[-1,1]$ and the error is governed by a 
${\cal N}(0,\sigma^2)$-distribution. We calculate the Richardson-extrapolation-type local linear smoother $\hat m_{\text{REE}}(0)$  (\citeA{schucany1977improvement}) from a sample $\{(X_i,Y_i)\}_{i=1}^n$ in model \eqref{example1}, where we specify the smoothness condition as $m\in C^2([-1,1])$ and use the bandwidth $h_{\text{REE}}=n^{-{1/5}}$. The REE-type $95\%$-confidence interval is then defined by  $\text{CI}_{\text{REE}}=[\hat m_{\text{REE}}(0)\pm \rm{se}(h_{\text{REE}})z_{1-0.025}],$ where $\rm{se}(h)$ is the oracle standard deviation of the estimator and $z_{1-\alpha}$ is the $(1-\alpha)$-quantile of the standard normal distribution.

In Figure \ref{fig:small-noise-ree} we display the empirical coverage probabilities of the $95\%$-confidence intervals (obtained by $50,000$ simulation runs) for two models corresponding to  $S\in \{2,4\}$, different sample sizes \(n\in\{100,300,500,700,1000\}\) and standard deviations \(\sigma\in\{0.003, 0.005,0.010\}\),
where the standard deviations of $\hat m_{\rm{REE}}$ is assumed as known. Thus, neither distributional approximation error nor the variance estimation error contributes to the reported coverage distortions.

 Figure~\ref{fig:small-noise-ree} isolates two distinct mechanisms through which standard-normal calibration can fail even after leading-bias correction. In panel (a), the regression function in model \eqref{example1} has smoothness $S=4$, which is well above the order-$2$ smoothness underlying the local-linear leading-bias correction. Together with exact conditional Gaussianity and oracle conditional standard errors, this construction rules out non-Gaussian approximation error, standard-error estimation error, insufficient smoothness, and failure to remove the leading smoothing-bias term of order $h^2$  as explanations for the observed undercoverage (here $h$ denotes the bandwidth). Nevertheless, the coverage deteriorates sharply as $\sigma$ decreases. In panel~(b) we keep the same Gaussian and oracle-information setting, but consider a regression function with  $S=2$ in model \eqref{example1}. This means that the available smoothness is exhausted by the leading-bias correction. The coverage distortion then becomes even more severe.
In the following Section \ref{sec 2.2}, we will give a theoretical explanation of these observations.

\begin{figure}[H]
    \centering
\includegraphics[width=\textwidth,keepaspectratio]{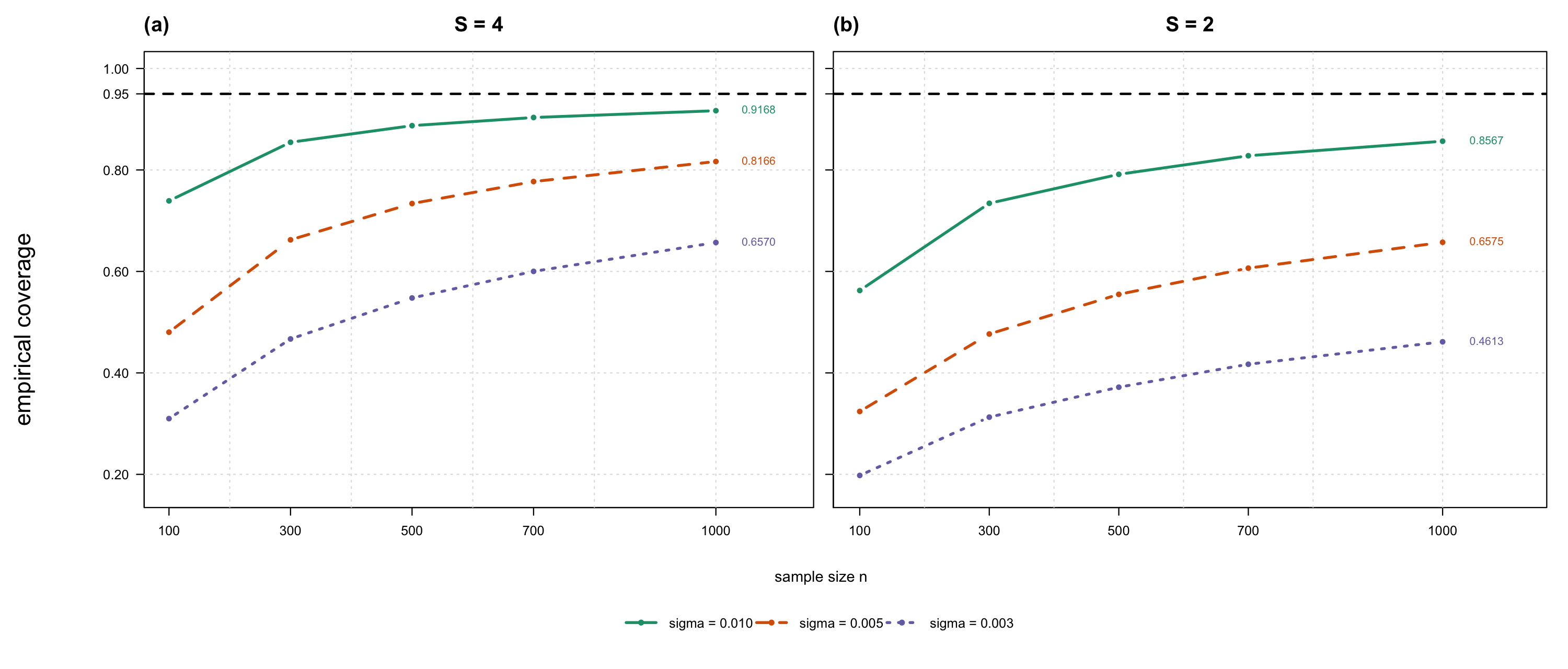}
    \caption{\it Empirical coverage probabilities of $95\%$ Richardson-extrapolation confidence intervals under Gaussian disturbances and known conditional variance.  Panel~(a) considers $S=4$, whereas panel~(b) considers $S=2$ in model \eqref{example1}.
    }
    \label{fig:small-noise-ree}
\end{figure}

\subsection{Theoretical Limitations of SNC after Bias Correction}
\label{sec 2.2}
The problem of coverage distortion observed in Section \ref{sec 2.1} is not specific to Richardson extrapolation or local-linear smoothing. It reflects a generic limitation of standard-normal calibration after bias correction. The following Proposition \ref{prop RE cov error} formalizes the mechanism behind the numerical evidence demonstrated in Section \ref{sec 2.1}. More specifically, it highlights a structural difficulty of confidence intervals for the value 
$\theta (x_0)$
of a function $\theta: [a,b] \to \mathbb{R}$ relying on bias correction of a nonparametric estimator and critical values from standard-normal approximation.

\begin{proposition}[Limitations of SNC after Bias Correction]
\label{prop RE cov error}
Let \(\widetilde\theta_h(x_0)\) denote any de-biased estimator with tuning parameter $h$,  define  $R_n(h,\theta):=\E[\widetilde\theta_h(x_0)]-\theta(x_0)$ and  consider the decomposition
$$
\frac{\widetilde\theta_h(x_0)-\theta(x_0)}{\sigma_n(h)}
=
\frac{\widetilde\theta_h(x_0)-\E[\widetilde\theta(x_0)]}{\sigma_n(h)}+\frac{\E[\widetilde\theta_h(x_0)]-\theta(x_0)}{\sigma_n(h)}
=:Z_n+\frac{R_n(h,\theta)}{\sigma_n(h)},
$$
where  $\sigma_n(h)$ is a normalizer such that \(Z_n\sim {\cal N}(0,1)\)  and \(\frac{R_0(h,\theta)}{\sigma_n(h)}\to0\) as $n \to \infty$, for every fixed \(\alpha\in(0,1)\), there exists a constant \(N_\alpha \in \mathbb N\) such that for all \(n\ge N_\alpha\),
\begin{align}
\label{eq residual bias lower bound}
(1-\alpha)
-
\mathbb P_\theta
\left(
\theta(x_0)\in
\left[
\widetilde\theta_h(x_0)
\pm
z_{1-\alpha/2}\sigma_n(h)
\right]
\right)
\geq
\frac{1}{4}
z_{1-\alpha/2}
\phi(z_{1-\alpha/2})
\left(
\frac{R_n(h,\theta)}{\sigma_n(h)}
\right)^2 ,
\end{align}
where $z_{1-\alpha/2}$ and $\phi$ denote the $(1-\alpha/2)$-quantile and density of the standard normal distribution. 
\end{proposition}
 We emphasize that Proposition \ref{prop RE cov error} is not specific to kernel smoothers. Rather, \(h\)  may represent a generic smoothing or regularization parameter, including, for example, the bandwidth of a kernel estimate, the inverse of the number of basis functions of a sieve estimate, or the ratio between the number of nearest neighbors and the sample size of a nearest neighbor estimate.

Proposition~\ref{prop RE cov error} reveals two bottlenecks for SNC-based confidence intervals after bias correction.  First, it identifies a rate bottleneck for coverage accuracy. The lower bound in \eqref{eq residual bias lower bound} shows that any normalized residual bias leads immediately to undercoverage, since $z_{1-\frac{\alpha}{2}}\phi(z_{1-\frac{\alpha}{2}})>0$.  Even after exact leading-bias cancellation, the coverage loss of a symmetric two-sided interval is generically lower bounded by the square of the remaining normalized bias. Hence, if \(\frac{R_0(h,\theta)}{\sigma_n(h)}\) vanishes slowly, the coverage error must also vanish slowly. More specifically, suppose that the interval is required to attain the minimax length rate under \(S\)-smoothness. Then, we obtain \(\sigma_n(h)\asymp (nh)^{-1/2}\asymp n^{-S/(2S+1)}\), which corresponds to \(h\asymp n^{-1/(2S+1)}\) and hence \(\sigma_n(h)\asymp h^S\). For functions for which the residual bias after exact \(S\)-order leading-bias correction is of order \(R_n(h,\theta)\asymp h^{S+\delta}\), the remaining normalized bias is of order \(h^\delta\). Proposition \ref{prop RE cov error} then implies a two-sided coverage-error lower bound of order \(h^{2\delta}\). Thus, when the true local smoothness is only slightly stronger than the smoothness exploited by the estimator, the residual inferential bias can substantially deteriorate the coverage accuracy.  This explains the larger coverage error in panel (b) compared to panel (a) in Figure \ref{fig:small-noise-ree}.

Second, Proposition~\ref{prop RE cov error} reveals an intrinsic length--coverage trade-off for confidence intervals based on standard-normal critical-value calibration and bias correction. The length of a symmetric normal interval is proportional to $\sigma_n(h)$, whereas the lower bound in Proposition~\ref{prop RE cov error} is proportional to $\left(\frac{R_n(h,\theta)}{\sigma_n(h)}\right)^2.$
Thus, holding the residual bias fixed, reducing $\sigma_n(h)$ shortens the interval but increases its coverage error. As discussed in Section~\ref{sec 2.1}, in kernel estimation of a regression function, $ \sigma_n(h)=\frac{\sigma(x_0)}{\sqrt{nh}}$, where $\sigma(x_0)$ is the conditional standard deviation of the regression error at $x_0$. Therefore, smaller regression noise does not necessarily improve SNC-based inference unless $|R_n(h,\theta)|=0$ holds non-asymptotically, which is generally impossible. In particular, for any fixed $n_0\geq N_\alpha$ and $h_0>0$ such that $ \left|R_{n_0}(h_0,\theta)\right|>0$, a smaller $\sigma(x_0)$ admitting a shorter interval often spontaneously yields a larger coverage-error lower bound $\left(\sqrt{n_0h_0}R_{n_0}(h_0,\theta)/\sigma(x_0) \right)^2$.  This explains the increasing coverage error for a decreasing variance of the noise in both panels of Figure \ref{fig:small-noise-ree}.

Overall, these two bottlenecks reflect the necessity of considering bias-aware inference.

\subsection{Bias-Aware Inference: From FLCI to EBCI}
\label{sec 2.3}
The preceding analysis shows that, after leading-bias correction, valid inference cannot rely on the remaining normalized bias being negligible. A natural alternative is to incorporate an upper bound on the worst-case bias directly into the calibration of the confidence interval. We first review the construction of fixed-length confidence intervals (FLCIs) introduced in \citeA{armstrong2018optimal,armstrong2020simple} and show how it motivates the empirical Bernstein confidence intervals (EBCIs) developed in this paper.

Let $B_{\mathcal F}(h)=\sup_{\theta\in \mathcal{F}}|\E[\hat \theta_h(x_0)]-\theta(x_0)|$ denote the worst-case bias of an estimator $\hat \theta_h(x_0)$ over a function class $\mathcal F$, and let $\sigma_n(h)$ denote its standard deviation. The corresponding worst-case normalized bias is thus $t_{\mathcal F}(h):=\frac{B_{\mathcal F}(h)}{\sigma_n(h)}$. For a given bandwidth $h$, a raw bias-aware interval takes the form
\begin{align*}
    \mathrm{CI}_{\mathrm{BA}}(\alpha,t_{\mathcal{F}}(h))=\left[\hat \theta_h(x_0)\pm\sigma_n(h)\operatorname{cv}_{1-\alpha}\!\left(t_{\mathcal{F}}(h)\right)\right],
\end{align*}
where $\operatorname{cv}_{1-\alpha}(t)$ denotes the $(1-\alpha)$-quantile of the folded normal distribution 
$\lvert {\cal N} (t,1)\rvert$ with expectation $t$. Thus, rather than requiring the normalized bias to be asymptotically negligible, the interval radius is calibrated to the largest bias compatible with $\mathcal F$. The FLCI is a length-optimal special case of this raw bias-aware construction. In the local-smoothing setting with some $S$-order smoothness, one typically has  $B_{\mathcal F}(h)=\eta h^S$ and $\sigma_n(h)=\sqrt{C_V/nh}$ for some  $\eta,C_V>0$. Then, since the function $h\mapsto t_{\mathcal F}(h)$ is a one-to-one mapping with $t_{\mathcal F}(h)=\eta\sqrt{n/C_V}h^{S+\frac12}$, the minimization of the raw bias-aware radius over $h$ can therefore be rewritten as a minimization over the normalized worst-case bias $t$. Then, for any given $\alpha\in (0,1)$, two-sided fixed length confidence interval with coverage $1-\alpha$, denoted as $\text{FLCI}(\alpha)$, is defined by  
\begin{align}
\label{eq FLCI}
    &\mathrm{FLCI}(\alpha):= \mathrm{CI}_{\mathrm{BA}}(\alpha,t^*) = \left[\hat \theta_h(x_0)\pm\sigma_n(h^*)\operatorname{cv}_{1-\alpha}\!(t^*)\right],
    \end{align}
    where $h^*=(\frac{t^*\sqrt{C_V}}{\eta\sqrt{n}})^{\frac{1}{2S+1}}$, $t^*=\inf_{t>0}t^{-\frac{1}{2S+1}}\operatorname{cv}_{1-\alpha}(t)$, and the radius of FLCI, denoted as $R_{\rm{FLCI}}$, is
    \begin{align}
  R_{\rm{FLCI}}(\alpha)= & \sigma_n(h^*)\operatorname{cv}_{1-\alpha}(t^*) =
    \eta^{\frac{1}{2S+1}}
    \left(
        \frac{C_V}{n}
    \right)^{\frac{S}{2S+1}}
    \inf_{t>0}
    \left\{
        t^{-\frac{1}{2S+1}}
        \operatorname{cv}_{1-\alpha}(t)
    \right\}.\notag
\end{align}

We now motivate our empirical Bernstein approach for constructing confidence intervals.
Note that for small $\alpha$, the  critical value of the folded-normal distribution satisfies $\operatorname{cv}_{1-\alpha}(t)=t+z_{1-\alpha}+o\!\left(z_{1-\alpha}\right)$ uniformly with respect to  $t\geq 0$. Hence,  with some  algebra and the fact that $z_{1-\alpha}^2\sim 2\log(1/\alpha)$ as $\alpha\downarrow0$, the FLCI half-length yields  the approximation
\begin{align}
\label{eq oracle radius}
    (2S+1)
    \eta^{\frac{1}{2S+1}}
    \left(
        \frac{
            2\log(1/\alpha)C_V
        }{
            4S^2 n
        }
    \right)^{\frac{S}{2S+1}}
\end{align}
for $R_{\rm{FLCI}}(\alpha)$.
For the worst estimation bias $B_{\mathcal{F}}(h)=\eta h^S$, we immediately obtain
    \begin{align*}
          (2S+1)
    \eta^{\frac{1}{2S+1}}
    \left(
        \frac{
            2\log(1/\alpha)C_V
        }{
            4S^2 n
        }
    \right)^{\frac{S}{2S+1}} =\inf_{h>0}\left\{\sqrt{\frac{2\log(1/\alpha)C_V}{nh}}+B_{\mathcal{F}}(h)\right\} =:R_{\rm{EBCI}}(\alpha)
    \end{align*}
    for each given $\alpha\in (0,1)$. Suppose $ \hat \theta_h(x_0)-\E[\hat \theta_h(x_0)]$ takes the form $\sum_{i=1}^n\vartheta_i(h)$, where $\vartheta_i(h)$'s are real-valued random variables. If the partial sum admits a Bernstein-type high-probability bound, this implies an oracle confidence interval for $\theta(x_0)$ with radius $R_{\rm{EBCI}}(\alpha)$:
    \begin{align}
    \label{eq EBCI}
    \mathrm{EBCI}(\alpha):= \left[\hat \theta_h(x_0)\pm R_{\mathrm{EBCI}}(\alpha)\right].
    \end{align}
 However, constructing a feasible version of the oracle confidence interval with radius $R_{\rm EBCI}(\alpha)$ is a highly non-trivial task. It requires Bernstein-type high-probability bounds of the form \eqref{eq oracle radius} for a data-dependent choice of bandwidth and a data-driven proxy for the variance and we discuss these issues in the following sections. In broad terms, the resulting procedure follows the spirit of empirical Bernstein inequalities; see, for example, \citeA{maurer2009empirical}, \citeA{wang2024sharp}, and \citeA{waudby2024estimating}. We therefore call the resulting feasible interval an empirical Bernstein confidence interval (EBCI). We further argue that EBCI and FLCI are complementary tools for bias-aware inference. Further details are provided in Section~\ref{sec 4.3}.

\section{Empirical Bernstein Calibration for Local Polynomial Regression}
\label{sec 3}

Suppose that $\{(X_i,Y_i)\}_{i=1}^n$ is a set of independent $2$-dimensional vectors such that  
\begin{align}
    \label{eq regression}
    &Y_i=m(X_i)+V^{\frac{1}{2}}(X_i)\varepsilon_i,
\end{align}
where 
$m$ and $V$ denote the regression and variance function, respectively,  and $\E[\varepsilon_i|X_i]=0,\ \E[\varepsilon_i^2|X_i]=1$.  We are interested in a confidence interval for the value of the regression function at a point $x_0$,  which we choose to be $0$ without loss of generality, so that the target parameter is $m(0)$.
  We first focus on constructing one- and two-sided EBCIs for  $m(0)$ based on the local polynomial smoother (see Sections \ref{sec 3.2}--\ref{sec 3.3}). We impose the following assumptions, which are standard in nonparametric smoothing. Throughout this section, for local polynomial regression, \(h>0\) denotes a bandwidth satisfying \(h\to0\) and \(nh\to\infty\).
\begin{assumption} 
    \label{as 1}
    $\{(X_i,Y_i)\}_{i=1}^n$ are mutually independent. We further assume $X_i$'s as independent copies of the random variable $X$ taking values in the interval $[a,b]$ with a Lipschitz continuous density $f_X$ and  Lipschitz constant $L_{f}$ such that  $\inf_{x\in [a,b]}f_X(x)>0$. Without loss of generality, we set $[a,b]$ as $[-1,1]$ or  $[0,1]$ corresponding to the situation where the point of interest $x_0=0$ is an interior or a boundary point.
\end{assumption}

\begin{assumption} 
    \label{as 2}
 By letting  $\psi_n=n^{-\frac{1}{2S+1}}\log n$, we assume that 
  \begin{align}
\label{eq TaylorClass}
m\in \Theta_{0}(M_S,r_S,\psi_n)
:=
\Bigg\{
\theta:\ 
\Big|
\theta(0+\epsilon)-\theta(0)
-\sum_{s=1}^{S}\frac{\theta^{(s)}(0)}{s!}\epsilon^s
\Big|
\le
M_S r_S(|\epsilon|)|\epsilon|^S,
\quad
\forall |\epsilon|\le \psi_n
\Bigg\},
\end{align}
where $\theta^{(s)}$ is the $s$-order ordinary or one-sided derivative. Further, the function \(r_S:[0,\infty)\to[0,\infty)\) in \eqref{eq TaylorClass} is nondecreasing, may depend on evaluation point \(x_0=0\) but not on the function \(\theta\), and satisfies \(r_S(t)\to0\) as \(t\downarrow0\).  The constant \(M_S>0\) may be associated with the  evaluation point \(0\), but does not depend on \(\theta\), $n$ and $h$. We also denote the local Lipschitz constant of $m$ at the point $0$ as $L_m$. Moreover, we assume that the variance function $V$ is global Lipschitz continuous on the interval  $[a,b]$ with Lipschitz constant $L_V$ and $V(0)>0$.
\end{assumption}
 
\begin{assumption} 
\label{as 3}
  There exists a constant  $\varsigma>4+\frac{1}{S}$ such that $E_{\varsigma} :=\max_{1\leq i\leq n}\E[|\varepsilon_i|^{\varsigma}|X_i] <\infty$  almost surely, where  $E_{\varsigma} >0 $  does not depend on  $n$. 
\end{assumption}

Assumptions \ref{as 1}-\ref{as 3} are standard in many widely known textbooks on nonparametric statistics or econometrics (e.g., \citeA{Tsybakov2009} and \citeA{li2007nonparametric}). According to the regularity imposed on the density function of $X$, Assumption \ref{as 1} implies  the existence of local polynomial estimators' inverse matrices and growth conditions of weight functions  with high probability (see Lemma \ref{lemma 3} and \ref{lemma 4} in the appendix). Another noteworthy point is that Assumption \ref{as 2} only requires the existence of $M_S$ and $r_S$, not their knowledge. As for the assumption of $\psi_n$, we actually only require $\lim_{n\to \infty}n^{-\frac{1}{2S+1}}/\psi_n=0$, but we use $\psi_n=n^{-\frac{1}{2S+1}}\log n$ here for the sake of simplicity. A prominent example of \(\Theta_{0}(M_S,r_S,\psi_n)\) is the local H\"older class of order \(S+\delta\), for some \(\delta\in(0,1]\).  In this case, the Taylor remainder is of order \(|\epsilon|^{S+\delta}\), corresponding to \(r_S(t)=t^\delta\). 

Let $K$ denote a kernel supported on the interval $[-1,1]$.
Following \citeA{fan1996local}, we define the local polynomial estimator of order $S$ by 
\begin{align}
    \label{eq local polynomial}
   & \hat m_h(x):=\sum_{i=1}^nW_{ih}(x)Y_{i},,
   \end{align}
   where 
   \begin{align*}
   W_{ih}(x)  &= \e^{\top}_{0}\M_{1h}^{-1}K\Big(\frac{X_i-x}{h}\Big)r\Big (\frac{X_i-x}{h}\Big ),\notag \\ 
    r^\top(u)&=\big(1,u,...,u^{S}\big),
    \ \e^\top_{0}=(1,0,...,0)\in \mathbb{R}^{S+1}, \notag  \\ 
       \M_{kh}&=\sum_{i=1}^{n}K^k\Big(\frac{X_i-x}{h}\Big)r\Big(\frac{X_i-x}{h}\Big)r^{\top}\Big(\frac{X_i-x}{h}\Big),\ \ \  \ k=1,2. \notag
   \end{align*}
 We also assume that  $||K||_{\infty}\leq 1$, $K$ is Lipschitz continuous, and continuously differentiable on $[-1,1]$ except at finitely many points on $[-1,1]$. On every differentiability interval, its first-order derivative exists and is bounded.

The rest of Section \ref{sec 3} is organized as follows. Section \ref{sec 3.2} gives  EBCIs for the local polynomial estimator of target parameter $m(0)$ for both boundary and interior cases. These EBCIs achieve the requirements  \ref{G1} and \ref{G2} simultaneously.  Similar to the FCLIs constructed by \citeA{armstrong2020simple}, these EBCIs depend on a bound, say $\eta$, for the worst-case bias, which has to be chosen by the user. Section \ref{sec 3.3} shows the possibility of constructing an $\eta$-free EBCI that still obtains \ref{G1}  and reaches \ref{G2} up to some logarithmic factors.

\subsection{Fixed-$\eta$ and Minimax EBCI}
\label{sec 3.2}
Based on the representation  \eqref{eq oracle radius} for the optimal length of the oracle  EBCIs,  we develop in this section several feasible and valid EBCIs 
for the value of the regression function at boundary and interior points. A crucial difficulty is that 
a simple plug-in estimator for the variance $C_V$  makes the bandwidth data dependent, while the Bernstein inequalities used to derive the confidence intervals may only be valid for a deterministic bandwidth. To ensure that the coverage probability still holds, a convenient approach is to construct a bandwidth selection procedure that converges to $h^*$ and controls the complexity of the candidate bandwidth set. Motivated by Lepski's method in adaptive estimation (e.g., \citeA{lepski1997optimal}), we consider the following two-step selection procedure to obtain a feasible bandwidth for EBCI.   
\begin{itemize}
    \item [(1)] We define the variance estimator 
    \begin{align}
    \label{eq variance proxy}
        \hat C_V=ng\sum_{i=1}^nW^2_{ig}(0)(Y_i-\hat{m}_{-i}(X_i))^2,
    \end{align}
      where $\hat{m}_{-i}(x)=\sum_{j\neq i}W_{jb}^{-i}(x)Y_j$  is the leave-one-out estimator of the regression function at the point $x$ with bandwidths $b=n^{-\frac{1}{3}}$ and $g=n^{-\frac{1}{2S+1}}$. Then we construct an initial data-dependent bandwidth 
      $$
      \hat h^*(t)=\Big (\frac{2\log(\frac{1}{t})\hat C_V}{4S^2\eta^2n}\Big )^{\frac{1}{2S+1}},
      $$
      where $t$ is defined as ${\alpha}$ and $\frac{\alpha}{2}$ for one-sided and two-sided confidence intervals respectively. 
    \item[(2)] We consider the grid $\mathcal{H}_n=\Big\{j\frac{n^{-\frac{1}{2S+1}}}{(\log n)^4}:1\leq j\leq [(\log n)^5]\Big\}$ and define 
 the EBCI bandwidth  by  
        \begin{align}
    \label{eq finite net}
    &\tilde h^*(t)=\min\Big \{h\in \mathcal{H}_n ~\big | ~h=\arg\min_{g\in\mathcal{H}_n}|g-\hat h^*(t)|\Big \}\ .
\end{align} 
\end{itemize}
Similar to Lepski's method, the candidate set $\mathcal{H}_n$ is always finite, which, except for guaranteeing the consistency of the selection procedure, offers a lot of convenience for the theoretical discussion, particularly for controlling the supremum of the empirical process whose index set is a neighborhood of $h^*(t)$.

Our first result in this section shows that the variance estimator is consistent. 

\begin{proposition}
    \label{prop consistCV}
       If Assumption \ref{as 1}-\ref{as 3} hold, we have  
   \begin{align}
       \label{eq th refinement eq3}
       &\hat C_V\xrightarrow[n\to\infty]{\P} C_V:=\begin{cases}
    \frac{V(0)}{f_X(0)}\int_{-1}^1(K(u)\e_0^\top\Gamma_1^{-1}r(u))^2du,\ \ & \text{ if } 0 \text{ is an interior point}; \\
   \frac{V(0)}{f_X(0)}\int_{0}^1(K(u)\e_0^\top(\Gamma'_1)^{-1}r(u))^2du,\ \  & \text{ if } 0 \text{ is a boundary point},
\end{cases}
       \end{align}
    as $n \to\infty$,   where, for a positive integer $k \geq 1$,
\begin{align}
\label{eq notation}
\begin{split}
    &\Gamma_k=\int_{-1}^1K^k(u)r(u)r^\top(u)du, \\
    &\Gamma_k'=\int_{0}^1K^k(u)r(u)r^\top(u)du, 
    \end{split}
\end{align}
and we assume that the matrices $\Gamma_1$ and $\Gamma'_1$ are non-singular. 
\end{proposition}

The following Theorem \ref{th refinement} and Corollary \ref{corollary 1} deliver one- and two-sided EBCIs for the target parameter $m(0)$, which satisfy the requirements \ref{G1} and \ref{G2}. Note that these results are simultaneously applicable for both interior and boundary cases.

\begin{theorem} [fixed-$\eta$ and minimax optimal EBCI for regression]
\label{th refinement}
    Suppose Assumptions \ref{as 1}-\ref{as 3} hold. For any given $\alpha\in (0,1)$ and a user-defined $\eta>0$, define 
 the radius by 
       \begin{align} 
        \label{eq th refinement eq2}
       &\hat r_{\eta}(t)=(2S+1)(1+\xi_n)\eta^{\frac{1}{2S+1}}\Big(\frac{2\log(1/t)\hat C_V}{4S^2}\Big)^{\frac{S}{2S+1}} n^{-\frac{S}{2S+1}},
    \end{align}
 with  $\xi_n=(\log n)^{-3}$ and $\hat C_V$ is introduced in \eqref{eq variance proxy}. Then we have
    \begin{align}
        \label{eq th refinement eq1}
       & \inf_{m\in\Theta_{0}(M_S,r_S,\psi_n)} \P\big(m(0)\leq \hat m_{\tilde h^*(\alpha)}(0)+ \hat r_{\eta}(\alpha)\big)\geq 1-\alpha- \textbf{CE}_n,  \\
           & \inf_{m\in\Theta_{0}(M_S,r_S,\psi_n)} P\big(m(0)\geq \hat m_{\tilde h^*(\alpha)}(0)- \hat r_{\eta}(\alpha)\big) \geq 1-\alpha- \textbf{CE}_n,\ \label{eq th refinement eq1a}
       \end{align}
   where $\textbf{CE}_n=o( n^{-\frac{2S}{2S+1}}) $. Moreover, when the $\varepsilon_i$'s are sub-Gaussian, it follows that $\textbf{CE}_n=O(e^{-n^{\kappa}})$ for some $\kappa>0$.   
\end{theorem}

\begin{corollary}
\label{corollary 1}
   If the assumptions of Theorem \ref{th refinement} are satisfied, we have 
    \begin{align}
        \label{eq corollary 1 2 eq1}
         \inf_{\theta\in  \Theta_{0}(M_S,r_S,\psi_n)}& \P\Big(m(0)\in \Big[\hat m_{\tilde h^*(\frac{\alpha}{2})}(0)\pm \hat r_{\eta} \Big(\frac{\alpha}{2}\Big) \Big]\Big)\geq 1-\alpha- \textbf{CE}_n^{~\prime}
         ,
    \end{align}
    where $\hat{r}_{\eta}(t)$ is introduced in \eqref{eq th refinement eq2} and    $\textbf{CE}_n^{~\prime}=o( n^{-\frac{2S}{2S+1}}) $. Moreover, when the $\varepsilon_i$'s are sub-Gaussian, it follows that $\textbf{CE}_n^{~\prime}=O(e^{-n^{\kappa}})$ for some $\kappa>0$.   
\end{corollary}

\begin{remark}
    \label{remark 1+}
    {\rm We emphasize that the statements 
 \eqref{eq th refinement eq1}, \eqref{eq th refinement eq1a} and \eqref{eq corollary 1 2 eq1} 
 in Theorem \ref{th refinement} and Corollary \ref{corollary 1} do not follow 
from a direct application of the Bernstein inequality,  because this result requires independent random variables, which can only be achieved with a deterministic bandwidth.  In fact, the proof of the confidence statements in Theorem \ref{th refinement} and Corollary \ref{corollary 1} is based on four ingredients: 1) a fixed-bandwidth empirical-Bernstein calibration; 2) consistency of the variance estimator \(\hat C_V\) in \eqref{eq th refinement eq3}   (proof of Proposition \ref{prop consistCV});  3) finite-grid projection of the plug-in bandwidth; 4)   the envelope control of the Taylor remainder defined by $\Theta_0(M_S,r_S,\psi_n)$}.
\end{remark}
 Theorem \ref{th refinement}  shows that in the regression setting the proposed EBCI is fully feasible and simultaneously retains the requirement of uniform safety in  \ref{G1} and the sharpness target \ref{G2}. More precisely, it establishes the validity of both one-sided intervals, where the under-coverage error has the same order as that for a two-sided interval. Simultaneously, the radius in \eqref{eq th refinement eq2} attains the minimax shrinking rate $n^{-S/(2S+1)}$ under S-order local smoothness. Thus, in the regression setting, the assumed S-order smoothness is translated directly by EBCI into both interval-length efficiency and coverage accuracy. Another important aspect of the proposed EBCI is that it allows for heteroskedasticity and that the statistical guarantees in Theorem \ref{th refinement} require only the existence of moments of order slightly larger than $4$. The limiting constants in \eqref{eq th refinement eq3} automatically distinguish the interior and boundary cases through the corresponding equivalent-kernel variance expressions.  Thus, the same EBCI construction covers both settings without changing the rate statements.

Next we comment on the effect of the tuning parameter \(\eta\). 
Since the unknown Taylor remainder is of order \(o(h^S)\), any fixed \(\eta>0\) eventually provides a valid \(h^S\)-level envelope and therefore the same optimal rate conclusion by Theorem \ref{th refinement}.  However, different choices of \(\eta\) affect the finite-sample constant and the implied bias--stochastic-error balance. Note that this user-specified property is not a peculiarity of EBCI. It is actually a standard feature of honest bias-aware inference.  For example, the FLCI construction discussed in Section \ref{sec 3.3} requires specifying either a smoothness constant or a worst-case bias bound.  Thus, \(\eta\) serves as a bias-budget sensitivity parameter rather than an estimand.  Smaller values of \(\eta\) produce shorter intervals but require the asymptotic Taylor remainder to be correspondingly small at the relevant bandwidth, while larger values yield more conservative bias protection.  In this sense, the choice of \(\eta\) is analogous to the choice of a smoothness constant in BA-inference.  

\subsection{$\eta$-free and nearly-minimax EBCIs}
\label{sec 3.3}
As pointed out in the previous section, the choice of \(\eta\) can strongly affect finite-sample efficiency because it changes both the leading constant of the EBCI radius and the oracle bandwidth \(h^*\).  In this sense, although any $\eta>0$ provides a confidence interval satisfying the (asymptotic) requirements \ref{G1} and \ref{G2}, it becomes an additional sensitivity parameter for data analysis by EBCI.  It also creates an interpretability issue under the present function class  \(\Theta_{0}(M_S,r_S,\psi_n)\) defined in \eqref{eq TaylorClass}. Since the rate of \(r_S(h)\) is allowed to be unknown, any fixed \(\eta>0\) eventually yields a valid worst-bias control, which implies that \(\eta\) cannot be directly interpreted as the size of the function class \(\Theta_{0}(M_S,r_S,\psi_n)\).

This section is thus dedicated to delivering one-sided and two-sided simple EBCIs whose components are all independent of the choice of $\eta$. Moreover, the following results  show that the only cost of being “$\eta$-free” is that the shrinking rates of $\eta$-free EBCIs  obtain 
the optimal rate up to a sub-polynomial factor with arbitrary small exponent,
that is
\begin{align} \label{eq nearly minimax} (\text{length of CI})/n^{-\frac{S}{2S+1}+\tau}\xrightarrow{n\to\infty} 0\ \ \text{holds for arbitrary fixed}\ \tau>0 \end{align}
For one-sided confidence intervals of the form  $[\hat m(0) - \hat r_n, \infty ) $ or   $( -\infty, \hat m(0) + \hat r_n ]$ we measure their length again by $\hat r_n$.

\begin{theorem} [$\eta$-free and nearly minimax  EBCI]
\label{th eta free regression}
Suppose Assumptions~\ref{as 1}--\ref{as 3} hold.  Let \(d_n\) be a deterministic sequence satisfying $d_n\to\infty$, $\frac{d_n}{\sqrt {\log n}}\to\infty$, \(\alpha\in(0,1)\) and define the radii 
    $$
   \hat r (t)= 
    d_n
    \Big (
        \frac{2\log(1/t)\widehat C_{V,0}}{n}
    \Big)^{\frac{S}{2S+1}}~, 
    $$
    where  $\widehat C_{V,0}=nh_{\mathrm{nv}}\sum_{i=1}^nW^2_{ih_{\mathrm{nv}}}(0)(Y_i-\hat m_{h_{\mathrm{nv}}}(0))^2$ with  
   $  h_{\mathrm{nv}}
    =
    n^{-\frac{1}{2S+1}}.$ Then we have 
\begin{align*}
     \inf_{m\in \Theta_{0}(M_S,r_S,\psi_n)}
     \mathbb P
    \left(
        m(0)\le
        \widehat m_{h_{\mathrm{nv}}}(0)+\widehat r(\alpha)
    \right)
    & \ge
    1-\alpha- \textbf{CE}_n,
    \\
       \inf_{m\in \Theta_{0}(M_S,r_S,\psi_n)}
     \mathbb P
    \left(
        m(0)\ge
        \widehat m_{h_{\mathrm{nv}}}(0)-\widehat r(\alpha)
    \right)
    & \ge
    1-\alpha- \textbf{CE}_n,
    \\
    \inf_{m\in \Theta_{0}(M_S,r_S,\psi_n)}
    \mathbb P
    \left(
        m(0)\in
        \left[
            \widehat m_{h_{\mathrm{nv}}}(0)
            \pm
            \widehat r(\alpha/2)
        \right]
    \right)
    & \ge
    1-\alpha- \textbf{CE}_n^{~\prime},
\end{align*}
where $\max \big \{  \textbf{CE}_n ,   \textbf{CE}_n^{~\prime} \big \}=o\!(n^{-\frac{2S}{2S+1}})$. Moreover,  if the $\varepsilon_i$'s  in model \eqref{eq regression} are sub-Gaussian, it follows that $\max \big \{  \textbf{CE}_n ,   \textbf{CE}_n^{~\prime} \big \}=O(e^{-n^{\kappa}})$ for some $\kappa>0$.
\end{theorem}
By taking, for example, \(d_n=\log\log n\sqrt{\log n}\), the radii in Theorem \ref{th eta free regression} are independent of the tuning parameter \(\eta\) and satisfy the minimax property in \eqref{eq nearly minimax} up to a
sub-polynomial factor with arbitrary small exponent. Thus, the $\eta$-free EBCI can be viewed as a light-undersmoothing-type confidence interval: it gives up the exact minimax constant in exchange for a simple, stable, and interpretable radius whose calibration is determined only by the level \(\alpha\). This observation also highlights a useful distinction between $\eta$-free EBCI and SNC-based undersmoothing. If one aims to construct a confidence interval satisfying \ref{G1} by using an SNC-based undersmoothing strategy, Proposition \ref{prop RE cov error} implies that a polynomial-level efficiency loss may be unavoidable in regimes where the normalized remaining bias has a polynomial shrinking rate. By contrast, the $\eta$-free EBCI attains \ref{G1} while paying only the slowly diverging factor \(d_n\) with a sub-polynomial rate with arbitrary small exponent. This favorable tradeoff arises because empirical Bernstein calibration controls tail uncertainty on the original estimation scale, rather than through a normalized bias.

  Generally speaking, the fixed-$\eta$ EBCI should be viewed as the sharp bias-budget benchmark, whereas the $\eta$-free EBCI is the implementation-oriented version that keeps honesty without requiring this additional sensitivity input. The preceding theorem is stated for local-polynomial regression, but the role of the slowly diverging factor \(d_n\) is more general. It turns the fixed-bias-budget construction into a transportable calibration template, as discussed next.

\section{Empirical Bernstein Calibration for Sieve Regression}
\label{sec 3.4}
In this section, we focus on constructing a $\eta$-free and nearly-minimax EBCI for linear sieve estimators introduced in \citeA{chen2015optimal} for model \eqref{eq regression} with $V^{\frac{1}{2}}(x)\equiv V^{\frac{1}{2}}>0$.

Suppose $\{b_{K1},...,b_{KK}\}$ is a sieve basis of dimension $K$. A linear sieve estimator of $m(0)$ is 
\begin{align}
    \label{eq linear sieve}
    &\hat m_{K}(0)=\sum_{i=1}^nW_{iK}(0)Y_i:=\sum_{i=1}^n\Big(\frac{1}{n}b^\top_K(0)\widehat Q_K^{-1}b_K(X_i)\Big)Y_i,\\
    b_K(x)=&(b_{K1}(x),...,b_{KK}(x))^\top,\ \B_K^\top=(b_1(X_1),...,b_K(X_n)),\ \widehat Q_K= \frac{1}{n}\B_K^\top \B_K.\notag
\end{align}
Before we exhibit the $\eta$-free EBCI for the $\hat m_K(0)$, we first introduce the following assumptions that are necessary for the bias properties of $\hat m_K(0)$, which are also the core assumptions of \citeauthor{chen2015optimal}'s work.
\begin{assumption} [basis regularity]
    \label{as 8}
    There exist constants \(C_0,C_\infty,C_1,C_2<\infty\) independent of \(n\) and \(K\), such that
\begin{align*}
    \sup_{x\in[-1,1]}\#\{k:b_{Kk}(x)\neq 0\}\le C_0,&
\qquad
\max_{1\le k\le K}\|b_{Kk}\|_\infty\le C_\infty K^{1/2},\\
\max_{1\le k\le K}\E|b_{Kk}(X)|\le C_1K^{-1/2},&
\qquad
\max_{1\le k\le K}\E[b_{Kk}^2(X)]\le C_2,
\end{align*}
where $\# A$ denotes the cardinality of a set $A$.
\end{assumption}

\begin{assumption}[non-overlap properties]
\label{as 9}
There exists a fixed integer \(M_0<\infty\), independent of \(n\) and \(K\),
such that
\[
b_{Kk}(x)b_{K\ell}(x)\equiv 0
\qquad
\text{whenever } |k-\ell|>M_0 .
\]
Consequently, the $(k,l)$-th components of matrices $Q_K:=E[b_K(X)b^\top_K(X)]$ and $\widehat Q_K$
are equal to $0$ once \(|k-\ell|>M_0\).
\end{assumption}

\begin{assumption}[eigenvalues of Gram matrix]
\label{as 10}
Based on the $Q_k$ defined in Assumption \ref{as 9}, we assume that there exist constants \(0<q_-<q_+<\infty\), independent of \(K\), such that
\[
q_- \le \lambda_{\min}(Q_K)
\le \lambda_{\max}(Q_K)
\le q_+ ,
\]
where $\lambda_{\max}$ and $\lambda_{\min}$ denote the largest and smallest eigenvalue.
\end{assumption}

\begin{assumption}[Sieve approximation]
\label{as 11}
Given a fixed $S\ge 1$, let $\mathcal F_S$ be a fixed collection of functions
such that $m\in\mathcal F_S$. By letting $\mathcal B_K:=\mathrm{span}\{b_{K1},\ldots,b_{KK}\}$, we assume that
\[
\sup_{f\in\mathcal F_S}
\inf_{g\in\mathcal B_K}
\|g-f\|_\infty
\le C_S K^{-S},
\qquad \forall K\ge 1,
\]
for some constant $C_S>0$ independent of $K$.
\end{assumption}
Assumptions \ref{as 8}--\ref{as 10} are concluded based on the rescaled fixed-degree B-spline basis and compactly supported fixed-resolution wavelet projection basis discussed in \citeA{chen2015optimal}. More specifically, for the non-overlap property of B-splines,  Section 6 of \citeauthor{chen2015optimal}'s work constructs a fixed-order spline basis with uniformly bounded mesh ratio. After the usual rescaling, this yields \(K=m+r\) splines of fixed order \(r\), so each basis function overlaps only with a fixed number of neighbouring splines. Hence, after ordering the basis functions according to their knot locations, there exists a constant \(M_{\mathrm{spl}}\)
independent of $n$ and \(K\) such that Assumption \ref{as 9} is satisfied. For wavelet sieves, the condition is even more explicit. \citeauthor{chen2015optimal} use wavelets introduced in \citeA{cohen1993wavelets} and they notice that the fixed-resolution scaling basis can equivalently represent the wavelet projection space
\(\{\phi_{J,k}:0\le k\le 2^J-1\}\). They state that, by compact support,
at most \(2N-1\) of these basis functions overlap on a set of positive
Lebesgue measure, where \(N\) is the support length. Therefore, after ordering the scaling functions by location, there exists a constant \(M_{\mathrm{wav}}\) depending only on the support length such that Assumption \ref{as 9} is satisfied. Assumption \ref{as 11}  is a standard assumption  to control the bias of series regression (e.g., \citeA{newey1997convergence}), and $\mathcal{F}_S$ mentioned there could be a collection of functions with $S$-order smoothness, like  a Hölder class.  
\begin{theorem} [$\eta$-free and nearly-minimax EBCI for Sieve]
    \label{th linear sieve}
    Suppose that,  Assumptions \ref{as 1}, \ref{as 3} and \ref{as 8}-\ref{as 11} are satisfied, and that  $\hat V$ is an estimator of $V$ such that $\P(|\hat V-V|\leq  n^{-\gamma})\geq 1-o(n^{-\frac{2S}{2S+1}})$ holds for some universal $\gamma>0$. Then, by letting $K\asymp n^{\frac{1}{2S+1}}$, the following inequality holds for every fixed $\alpha\in (0,1)$ and user-defined $d_n$ such that $d_n/\sqrt{\log n}\to \infty$,
    \begin{align}
        \inf_{m\in \mathcal{F}_S}\P\Big(m(0)\in \Big[\hat m_K(0)\pm d_n\Big(\frac{2\log(\frac{2}{\alpha})\hat V\hat C_W}{n}\Big)^{\frac{S}{2S+1}}\Big]\Big)\geq 1-\alpha-\textbf{CE}_n,
    \end{align}
    where $\textbf{CE}_n=o(n^{-\frac{2S}{2S+1}})$, $\hat C_W=\frac{n}{K}\sum_{i=1}^nW_{iK}^2(0)$ and $\mathcal{F}_S$ is introduced in Assumption \ref{as 11}.
\end{theorem}
\textrm{\begin{remark}
{Many variance estimators satisfy condition  $\P(|\hat V-V|\leq  n^{-\gamma})\geq 1-o(n^{-\frac{2S}{2S+1}})$. Two prominent examples of such $\hat V$ are residual-based and difference-based estimators (e.g., \citeA{hall1992effect} and \citeA{shen2020optimal}). The variance proxy used in Theorem \ref{th eta free regression} also meets this requirement.}
\end{remark}}

 For generic sieve estimators, as discussed in \citeA{cattaneo2020}, it is often nontrivial to obtain the analytic form of the pointwise leading bias, since it is governed by the approximation properties of the sieve space and by the stability of the associated projection operator. When an explicit leading bias expansion is not available, undersmoothing is therefore a widely used route (e.g. \citeA{dong2021weighted} and \citeA{bach2025estimation}). Proposition \ref{prop RE cov error}, however, shows why this strategy may be
costly when the final interval is calibrated by a standard-normal critical value. By contrast, Theorem \ref{th linear sieve} demonstrates that  $\eta$-free EBCI for sieve regression  obtains \ref{G1} and nearly \ref{G2} since the bias is controlled on the original estimation scale and absorbed into the
empirical-Bernstein radius. The only price is the slowly diverging factor \(d_n\), which is required to satisfy \(d_n/\sqrt{\log n}\to\infty\). 
In this sense, for the standard-normal-calibrated undersmoothing route discussed above, the $\eta$-free EBCI for sieve estimators replaces the usual polynomial-level efficiency cost suggested 
by Proposition~\ref{prop RE cov error} with a logarithmic-level cost, while providing rate-explicit 
protection against undercoverage.

The comparison with RBC-type intervals for series regression is not direct, because
the two approaches start from different information sets. The pointwise inference of RBC-type partitioning-based series estimators introduced in \citeA{cattaneo2020} is built on an analytic characterization of the leading approximation bias. By contrast, our Assumption \ref{as 11} only imposes the uniform approximation rate \(K^{-S}\), and no explicit leading-bias expansion or bias estimator is required. A second distinction is the requirement of smoothness. Series RBC removes leading bias by using a higher-order approximation structure.
In our formulation, once the class \(\mathcal F_S\) and the order \(S\) are fixed, the assumption does not provide an additional smoothness margin. The point of Theorem \ref{th linear sieve} is therefore different: empirical-Bernstein calibration nearly achieves \ref{G2} allowed by Assumption \ref{as 11}, while maintaining the coverage target \ref{G1}, rather than estimating and subtracting a leading sieve bias under additional smoothness.

\section{Empirical Calibration for Kernel Density Estimation}
\label{sec density}
In this section, we construct EBCIs for the value of a density at a given point, covering both interior and boundary cases.  The only difference between the two cases is the choice of kernel support and the corresponding variance proxy.

More specifically, let $f_X$ denote the Lebesgue density of a real-valued random variable $X$  and assume it is uniformly bounded away from $0$. We are interested in a confidence interval for $f_X(x_0)$, where we let $x_0=0$ without loss of generality. We assume that $0$ is either an interior point of the support of $f_X$ or a boundary point. We consider kernel density estimation and denote by $I$  the support of the kernel function $K_I$, where we choose \(I=[-1,1]\) for the interior case and \(I=[0,1]\) or $I=[-1,0]$ for the left or right boundary cases.  We assume that  $K_I$  is bounded  and satisfies
\[
\int_I K_I(u)\,du=1,
\qquad
\int_I u^kK_I(u)\,du=0,
\quad k=1,\ldots,S.
\]
For example, when $S=2$, for the interior case, \(K_I\) is the usual symmetric density function, for the boundary case, \(K_I\) is a one-sided boundary kernel (e.g., Chapter 1 of \citeA{li2007nonparametric}). 

Suppose  that $f_X\in\Theta_{0}(M_S,r_S,\psi_n)$, where $\Theta_{0}(M_S,r_S,\psi_n)$ is introduced in \eqref{eq TaylorClass} and $\psi_n=n^{-\frac{1}{2S+1}}\log n$. Given an i.i.d. sample \(\{X_i\}_{i=1}^n\), we define the kernel density estimator
\begin{align}
    \label{eq KDE}
    \widehat f_{\tilde h_0(t)}(0)
    =
    \frac{1}{n\tilde h_0(t)}
    \sum_{i=1}^n
    K_I\!\left(\frac{X_i}{\tilde h_0(t)}\right),
\end{align}
where $t$ is defined as $\alpha$ and $\frac{\alpha}{2}$ for one-sided and two-sided confidence intervals respectively. The bandwidth  \(\tilde h_0(t)\) is chosen by a similar selection procedure as  in Section~\ref{sec 2}. More specifically, we define
\[
    \tilde h_0(t)
    =
    \arg\min_{h\in\mathcal H_n}
    |h-\hat h_0(t)|,
\]
where the grid  \(\mathcal H_n\) is defined in \eqref{eq finite net} and
\begin{align}
\label{eq Vng}
      \hat h_0(t)
    =
    \left(
    \frac{
        2\log(\frac{1}{t})V_{ng}
    }{
        4S^2\eta^2 n
    }
    \right)^{\frac{1}{2S+1}},\quad  V_{ng}
    =
    \frac{1}{ng}
    \sum_{i=1}^{n/2}\left(K_{I}\Big(\frac{X_{2i}}{g}\Big)-K_{I}\Big(\frac{X_{2i-1}}{g}\Big).
    \right)^2
\end{align} 

\begin{theorem}[fixed-$\eta$ and minimax optimal  EBCIs for the density]
\label{th density}
For any   \(\alpha\in(0,1)\) and \(\eta>0\) define 
\begin{align*}
     &\bar r_\eta(t)
    =
    (2S+1)(1+\xi_n)\eta^{\frac{1}{2S+1}}
    \left(
        \frac{2\log(1/t)V_{ng}}{4S^2}
    \right)^{\frac{S}{2S+1}}
    n^{-\frac{S}{2S+1}},
\end{align*}
where 
$V_{ng}$ is defined in  \eqref{eq Vng},  \(\xi_n = (\log n)^{-3}\)   and \(g=n^{-\frac{1}{2S+1}}\) serves as the pilot bandwidth. Under the assumptions stated in the  paragraph at the beginning of this section, we have  
\begin{align}
\label{eq th density eq1}
 \inf_{f\in\Theta_{0}(M_S,r_S,\psi_n)} 
\mathbb P\!\left(
\widehat f_{\tilde h_0(\alpha)}(0)
\ge
f_X(0)-\bar r_{\eta}(\alpha)
\right)
&\geq
1-\alpha-   \textbf{CE}_n ,
\\
\label{eq th density eq1a}
\inf_{f\in\Theta_{0}(M_S,r_S,\psi_n)}
\mathbb P\!\left(
f_X(0)
 \ge
\widehat f_{\tilde h_0(\alpha)}(0)-\bar r_\eta(\alpha)
\right)
& \geq
1-\alpha-\textbf{CE}_n,
\\
\label{eq th density eq2}
\inf_{f\in\Theta_{0}(M_S,r_S,\psi_n)}
\mathbb P\!\left(
f_X(0)
\in
\left[
\widehat f_{\tilde h_0(\frac{\alpha}{2})}(0)
\pm
\bar r_\eta\!\left(\frac{\alpha}{2}\right)
\right]
\right)
& \ge
1-\alpha- \textbf{CE}_n^{~\prime} ,
\end{align}
where $ \textbf{CE}_n =  O(e^{-n^{\kappa_1}})$ and  
 $ \textbf{CE}_n^{~\prime} =  O(e^{-n^{\kappa_2}})$  for  constants \(\kappa_1,\kappa_2>0\).
 Moreover, we have
 \begin{align*}
     V_{ng}
    \xrightarrow[n\to\infty]{\mathbb P}
    f_X(0)
    \int_IK_I^2(u)\,du .
 \end{align*}
\end{theorem}
Theorem \ref{th density} illustrates the cleanest form of the EBCI mechanism.  In density estimation, the localized kernel summands are naturally bounded, so empirical-Bernstein calibration can be applied without imposing additional moment conditions.  Consequently, both one-sided and two-sided EBCIs enjoy exponentially small coverage-error remainders, while the radius still shrinks at the minimax optimal rate $n^{-\frac{S}{2S+1}}$. Thus, for interior-point density estimation, EBCI achieves the two goals in their strongest form: minimax-sharp length and exponentially accurate coverage.  The variance proxy \(V_{ng}\) makes the radius feasible by estimating the leading variance constant $ f_X(0)\int_{-1}^{1}K^2_I(u)du$, resulting in a confidence interval that is both data-driven and free from normal-approximation calibration.

Similar to Section \ref{sec 3.3}, the following result  delivers $\eta$-free and nearly-minimax EBCIs for $f_X(0)$. We obtain

\begin{theorem}[$\eta$-free and nearly-minimax optimal EBCIs for the  density]
    \label{th eta free density}
     Let \(d_n\) be a deterministic sequence satisfying $d_n\to\infty$, $\frac{d_n}{\log n}\to\infty$ and define
        \begin{align*}
      \bar r(t)=\Big(\frac{2\log(1/t)V_{ng}}{n}\Big)^{\frac{S}{2S+1}}d_n,
    \end{align*} 
      where  $V_{ng}$ is defined in \eqref{eq Vng}. 
       Under the assumptions stated in the  paragraph at the beginning of this section, we have  
    \begin{align*}
        \inf_{f\in\Theta_{0}(M_S,r_S,\psi_n)}
\mathbb P\!\left(
\widehat f_{h_{nv}}(0)
\ge
f_X(0)-\bar r(\alpha)
\right) 
& \geq
1-\alpha-\textbf{CE}_n,
\\
       \inf_{f\in\Theta_{0}(M_S,r_S,\psi_n)}
\mathbb P\!\left(
f_X(0)
\ge
\widehat f_{ h_{nv}}(0)-\bar r(\alpha)
\right)
& \geq
1-\alpha-\textbf{CE}_n,
\\
\inf_{f\in\Theta_{0}(M_S,r_S,\psi_n)}
\mathbb P\!\left(
f_X(0)
\in
\left[
\widehat f_{h_{nv}}(0)
\pm
\bar r\!\left(\frac{\alpha}{2}\right)
\right]
\right)
& \ge
1-\alpha-\textbf{CE}_n^{~\prime},
    \end{align*}
    where $h_{nv}$ is introduced in Theorem \ref{th eta free regression}, 
    $\textbf{CE}_n =  O(e^{-n^{\kappa_1}})$ and  
 $ \textbf{CE}_n^{~\prime} =  O(e^{-n^{\kappa_2}})$  for \(\kappa_1,\kappa_2>0\).
\end{theorem}
By taking $d_n=\log n(\log\log n)$, $\bar r(\alpha)$ satisfies the nearly minimax criteria introduced in \eqref{eq nearly minimax}.

\section{Bias Control and Calibration}
\label{sec 4}
 
 Constructing pointwise confidence interval for a nonparametric smoother often involves two layers, calibration of a critical value and bias control. 
Worst-case bias control and RBC should be viewed as bias-control devices, while critical values based on asymptotic normality and tail control via Bernstein inequalities are calibration principles. This section clarifies how EBCI relates to these two layers of the inference problem.

\subsection{Smoothness Allocation under SNC-based RBC and EBCI}
\label{sec 4.1}
The discussion in this subsection should not be read as a criticism of RBC as a debiasing principle. Instead, we focus on how a given smoothness condition is converted under different procedures of constructing confidence intervals. Provided that some higher-order local smoothness is specified correctly, the features of SNC-based RBC and EBCI can be concluded as follow,
\begin{enumerate}[label=\textbf{F\arabic*}, ref=F\arabic*]
    \item \label{F1} Similar to Proposition \ref{prop RE cov error}, SNC-based RBC still faces a coverage--length trade-off. On the one hand, it uses higher-order smoothness to reduce the residual normalized bias required for normal calibration. On the other hand, SNC-based RBC must select a bandwidth that minimizes the leading terms in a uniform Edgeworth expansion of coverage error. This bandwidth is often of smaller order than the MSE-optimal bandwidth, thereby inducing a polynomial rate loss in interval length relative to the pointwise minimax benchmark under the same correctly specified smoothness class.
    \item \label{F2} To construct an EBCI, it suffices to control the deterministic bias on the original estimation scale. The residual bias enters the interval radius directly and is not required to be negligible after standardization by the stochastic scale. EBCI thus avoids this normalized-bias channel and converts the same smoothness more directly into interval-length efficiency and coverage control.
\end{enumerate}
However, the two ideas are not logically opposed because RBC and EBCI address different objectives. Indeed, as discussed in Section \ref{sec 4.2} below, an RBC-type EBCI can be constructed by applying empirical-Bernstein calibration to the RBC equivalent-kernel estimator. The comparison presented below is therefore about SNC-based RBC as an implementation of the RBC idea, not about the RBC debiasing philosophy itself. 

There are two natural ways to compare the theoretical properties of EBCI and SNC-based RBC. The first is class-driven: Given a fixed smoothness budget \(S\) and the goal of maintaining the CE-optimal RBC guarantee, one evaluates the possible local-polynomial order \(p\).  From this perspective, SNC-based RBC dedicates part of the available smoothness to its higher-order coverage expansion. The second is order-driven and consistent with \citeA{calonico2018effect}: Given a fixed order \(p\) for a local-polynomial order \(p\), the required smoothness S for the theoretical guarantee of CE-optimal bandwidth is then determined.  We focus on the second perspective in the following discussion. 
 Recall that, according to \citeA{calonico2018effect}, once the order \(p\) of a local-polynomial estimator and smoothness budget $S$ are fixed, the CE-optimal SNC-based RBC theory then requires additional smoothness: at an interior point, roughly \(S=p+3\) in the odd case and \(S=p+2\) in the even case; at a boundary point, roughly \(S=p+2\).  Then, the right part of Table~\ref{tab:rbc_ebci_rates} shows the length and CE  of the EBCI if the same smoothness budget required by the SNC-based RBC guarantee is taken as a primitive condition and used directly in the construction of the  EBCI based on the local polynomial estimator.

\begin{table}[H]
\centering
\begingroup
\scriptsize
\setlength{\tabcolsep}{5pt}
\renewcommand{\arraystretch}{1.3}
\begin{tabular}{ccc ccc ccc}
\toprule
\multirow{2}{*}{Point} & \multirow{2}{*}{Estimation} & \multirow{2}{*}{Order $p$}
  & \multicolumn{3}{c}{SNC-based RBC}
  & \multicolumn{3}{c}{EBCI} \\
\cmidrule(lr){4-6}\cmidrule(l){7-9}
 & & & Length & CE & Moment & Length & CE & Moment \\
\midrule
\multirow{4}{*}{Interior}
 & \multirow{2}{*}{Regression} & odd
   & $n^{-\frac{p+3}{2(p+4)}}$
   & $n^{-\frac{p+3}{p+4}}$
   & \multirow{2}{*}{$>8$}
   & $n^{-\frac{p+3}{2p+7}}$
   & $n^{-\frac{2(p+3)}{2p+7}}$ or exp.
   & $>4+\tfrac{1}{p+3}$
 \\
 & & even
   & $n^{-\frac{p+2}{2(p+3)}}$
   & $n^{-\frac{p+2}{p+3}}$
   &
   & $n^{-\frac{p+2}{2p+5}}$
   & $n^{-\frac{2(p+2)}{2p+5}}$ or exp.
   & $>4+\tfrac{1}{p+2}$
 \\
\cmidrule(l){2-9}
 & \multirow{2}{*}{Density} & odd
   & $n^{-\frac{p+3}{2(p+4)}}$
   & $n^{-\frac{p+3}{p+4}}$
   & \multirow{2}{*}{$-$}
   & $n^{-\frac{p+3}{2p+7}}$
   & \multirow{2}{*}{exp.}
   & \multirow{2}{*}{$-$}
 \\
 & & even
   & $n^{-\frac{p+2}{2(p+3)}}$
   & $n^{-\frac{p+2}{p+3}}$
   &
   & $n^{-\frac{p+2}{2p+5}}$
   &
   &
 \\
\midrule
\multirow{2}{*}{Boundary}
 & Regression & $-$
   & $n^{-\frac{p+2}{2(p+3)}}$
   & $n^{-\frac{p+2}{p+3}}$
   & $>8$
   & $n^{-\frac{p+2}{2p+5}}$
   & $n^{-\frac{2(p+2)}{2p+5}}$ or exp.
   & $>4+\tfrac{1}{p+2}$
 \\
 & Density & $-$
   & $n^{-\frac{p+2}{2(p+3)}}$
   & $n^{-\frac{p+2}{p+3}}$
   & $-$
   & $n^{-\frac{p+2}{2p+5}}$
   & exp.
   & $-$
 \\
\bottomrule
\end{tabular}
\endgroup
\caption{\it Comparison of interval length, coverage error and moment conditions for SNC-based RBC for a fixed order $p$ of the local-polynomial estimator and EBCI under the same and sufficient smoothness. This table is an exhibition of smoothness allocation rather than a same-estimator comparison.}
\label{tab:rbc_ebci_rates}
\end{table}
\noindent The EBCI entries in Table~\ref{tab:rbc_ebci_rates} illustrate the allocation of the same smoothness budget directly to the EBCI local-polynomial center and radius. The table is therefore a smoothness-allocation comparison, rather than a same-estimator comparison. For instance, in the interior case, CE-optimal SNC-based RBC with 
a local-polynomial estimator of odd order 
requires the smoothness budget \(S=p+3\).  Its length and coverage-error rates are
\[
    n^{-\frac{p+3}{2(p+4)}}
    \qquad\text{and}\qquad
    n^{-\frac{p+3}{p+4}},
\]
respectively.  If the same smoothness budget \(S=p+3\) is used directly in EBCI, the corresponding rates for the length and CE become slightly  better, namely 
\[
    n^{-\frac{p+3}{2p+7}}
    \qquad\text{and}\qquad
    n^{-\frac{2(p+3)}{2p+7}}.
\]
The same patterns appear for a local polynomial estimator of even order and for the construction of confidence intervals for the value of the regression function at the boundary.  Moreover, for density estimation and regression models with bounded responses, the Bernstein inequality further yields exponential coverage-error remainders.

However, this comparison should be interpreted with care.  It does not imply that RBC is a weak bias-correction method.  In fact, what we want to highlight is the opposite. RBC is sufficiently strong that, when additional smoothness is present, its CE-optimal bandwidth theory can exploit that smoothness to improve normal-approximation coverage refinements.  Table~\ref{tab:rbc_ebci_rates} instead shows that the standard-normal critical-value calibration may constrain how much of this smoothness is converted into interval length and coverage-error rates.  In this sense, EBCI points to a complementary possibility. A strong bias-control idea may become even more effective when paired with a different stochastic calibration method. Of course, there is also a limitation on the EBCI side.  The gains displayed in Table~\ref{tab:rbc_ebci_rates} rely on correctly specified smoothness.  If the additional smoothness \(S=p+3\) or \(S=p+2\) is not actually present, similar to many other class-driven methods, such as  FLCI introduced in \citeA{armstrong2018optimal,armstrong2020simple}, an EBCI constructed under that smoothness level is no longer honest for the claimed class.  By contrast, RBC-type procedures may retain validity through more conservative or undersmoothing-like behavior.  Thus, the comparison highlights a tradeoff: EBCI more directly converts correctly specified smoothness into efficiency and coverage accuracy, while RBC retains robustness when higher-order smoothness is uncertain.

In the following section, we discuss the possibility of connecting RBC with empirical Bernstein calibration, which results in an RBC-type EBCI for local polynomial regression.

\subsection{Toward RBC-Type EBCI}
\label{sec 4.2}
We now focus on constructing an RBC-type EBCI for the parameter $m(0)$ based on a $p$-th order local polynomial smoother, where $m$ is the mean function defined in model \eqref{eq regression} with homoscedastic and bounded errors. According to \citeA{calonico2018effect}, the RBC-type local linear smoother can be written as 
\begin{align}
\label{eq rbc estimator}
    &\hat m^{\mathrm{rbc}}_{h,b}(0)=\sum_{i=1}^nw_{i,h,b}^{\mathrm{rbc}}(0)Y_i,
\end{align}
where $h$ and $b$ are bandwidths and the  specific definition of  the weights $w_{i,h,b}^{\mathrm{rbc}}$ can be found in  \citeA{calonico2018effect}. According to this reference, $\rho: = h/b > 0$ is a tuning parameter that controls the amount of pilot bandwidth $b$, and its default value is $1$.  Simple algebra shows that the  RBC-type weights 
$w_{i,h,b}^{\mathrm{rbc}}(0)$ satisfy 
the   reproduction properties 
\begin{align}
\label{eq reproduction of rbc}
    \sum_{i=1}^nw_{i,h,b}^{\mathrm{rbc}}(0)X_i^{k}= \delta_{k0} , \ \ \  k=0,1,...,p+1.
\end{align}
The following proposition points out that this conceptual compatibility is possible and promising.
\begin{proposition}
    \label{prop RBC EBCI}
   Let Assumptions \ref{as 1} and \ref{as 3} be satisfied and  assume that Assumption \ref{as 2} holds with $S> p+1$, $|\varepsilon_i|\leq B$ and $V(x)\equiv\sigma^2>0$, where $\sigma,B>0$ are known.
    Further, let  $h=b=n^{-\frac{1}{2p+3}}$, $g_n=l_n=(\log n/n)^{{1}/{(2p+3)}}$ and define   
     \begin{align*}
     R_{\mathrm{EBCI}}^{\mathrm{rbc}}
    (\alpha,h,b)=\sigma\sqrt{2\log(\frac{1}{\alpha})\sum_{i=1}^n(w_{i,h,b}^{\mathrm{rbc}}(0))^2}.
    \end{align*}
Then,  we have 
    \begin{align*}
   \P\left(m(0)\in      \left[\hat m^{\mathrm{rbc}}_{g_n,l_n}(0)\pm   R_{\mathrm{EBCI}}^{\mathrm{rbc}}
    (\alpha,h,b)\right]\right)\geq 1-\alpha-\textbf{CE}_n,
    \end{align*}
where $\textbf{CE}_n=2(\P(A_n)+\P(B_n)+\P(C_n)),$
\begin{align*}
    A_n:&=\bigg \{
\frac{\max_{i\le n}|w_{i,g_n,l_n}^{\mathrm{rbc}}(0)|}
{\sqrt{\sum_{i=1}^n(w_{i,g_n,l_n}^{\mathrm{rbc}}(0))^2}}
>
\frac{3\sqrt{2}}
{2B\sqrt{\log(2/\alpha)}}
\bigg \},\ B_n:=\bigg \{\frac{|\mathrm{Bias}_{\mathrm{rbc}}|}{\sqrt{\sum_{i=1}^n(w_{i,g_n,l_n}^{\mathrm{rbc}}(0))^2}}> \sqrt{\frac{\sigma^2}{4}\log(\frac{2}{\alpha})}\bigg \},\\
    C_n:&=\bigg \{\frac{\sum_{i=1}^n(w_{i,g_n,l_n}^{\mathrm{rbc}}(0))^2}{\sum_{i=1}^n(w_{i,h,b}^{\mathrm{rbc}}(0))^2}> \frac{\log(1/\alpha)}{4\log(2/\alpha)}\bigg \},\qquad \mathrm{Bias}_{\mathrm{rbc}}:=\sum_{i=1}^nw_{i,g_n,l_n}^{\mathrm{rbc}}(0)(m(X_i)-m(0)).
\end{align*}  
\end{proposition}
\begin{remark}
   \rm{ Proposition \ref{prop RBC EBCI} is a diagnostic proposition with a constructive compatibility result. It provides a feasible RBC-type EBCI for the target parameter $m(0)$ whose coverage guarantee is non-asymptotic. 
   Meanwhile, for any given $h=\rho b>0$ with some fixed $\rho>0$, repeating the technique used in proving Lemmas \ref{lemma 4} and \ref{lemma 5} shows that  $$|\mathrm{Bias}_{\mathrm{rbc}}|\lesssim (\log n/n)^{\frac{(p+2)}{2(p+1)+1}}\qquad \text{and} \qquad \max\{\max_{i\leq n}|w_{i,h,b}^{\mathrm{rbc}}|,\sum_{i=1}^n(w_{i,h,b}^{\mathrm{rbc}}(0))^2\}\asymp \frac{1}{nh}$$
hold with high probability, which implies that $A_n$, $B_n$ and $C_n$ are events with (exponentially) small probability. 
Thus, the shrinking rate of $ R_{\mathrm{EBCI}}^{\mathrm{rbc}}
    (\alpha,h,b)$ is $n^{-\frac{p+1}{2p+3}}$, which coincides with the optimal MSE rate of $p$-th order RBC-type estimator.}
\end{remark}

The preceding proposition provides a constructive counterpart to the discussion in \ref{F1} and \ref{F2}. It shows that RBC itself does not impose the higher smoothness requirement and bandwidth distortion associated with CE-optimal SNC-based RBC. Rather, they arise from the
standard-normal calibration approach, which must further reduce the residual normalized bias to obtain its Edgeworth coverage refinement. By contrast, after the RBC-type debiasing, empirical-Bernstein calibration only requires the remaining bias to be absorbable on the original estimation scale. In the present fixed-design example, this already holds for
every $S> p+1$, yielding eventual non-asymptotic coverage with an MSE-rate empirical-Bernstein radius. Thus, RBC-type EBCI remains available in the intermediate smoothness regime  $p+1< S<p+3$, where the additional
smoothness required for a CE-optimal SNC-based RBC confidence interval at the boundary does not hold.

Proposition~\ref{prop RBC EBCI} therefore complements the boundary comparison in Table~\ref{tab:rbc_ebci_rates}. The table shows how a correctly specified $p+2$-order smoothness budget is converted under CE-optimal SNC-based RBC and direct EBCI. The proposition shows, in addition, that RBC debiasing can be paired with empirical-Bernstein calibration before that stronger CE-optimal smoothness requirement is met.

\subsection{FLCI and EBCI: Complementary Tools of Bias-Aware Inference}
\label{sec 4.3}

In this section, we discuss the relationship and distinction between EBCIs and FLCIs. Proposition \ref{prop EBCI FLCI sharpness} shows that, under a common bias-budget normalization, the oracle EBCI and FLCI radii are first-order equivalent in the small-$\alpha$ regime. However, the leading constants of EBCIs for two-sided intervals are often slightly more conservative than those of FLCIs,
On the other hand, we demonstrate in Proposition \ref{prop FLCI normal approximation} that this conservativeness caused by the leading constants in EBCIs can be interpreted as a coverage premium, particularly for the distributional robustness.

\begin{proposition}[Leading Constant in Small-$\alpha$ Regime] \label{prop EBCI FLCI sharpness}
{Let $S\geq 1$, and  $\hat \theta_h(x_0)$ be  an estimator of the value 
$\theta(x_0)$ with bandwidth $h$, where $\theta \in \Theta_{x_0}(M_S,r_S,\psi_n)$,  and $\Theta_{x_0}(M_S,r_S,\psi_n)$ is introduced in \eqref{eq TaylorClass}. We assume that the deterministic bias envelope and stochastic scale satisfy 
\begin{align}
\sup_{\theta\in \Theta_{x_0}}|\E[\hat \theta_h (x_0)]-\theta(x_0)| = \eta h^S,\ \ \sigma^2_n(h):=\text{Var}(\hat{\theta}_h(x_0)) = \sqrt{\frac{C_V}{nh}}, 
\end{align} 
for some fixed $\eta,C_V>0$. Let $R_{{\rm FLCI},k}(\alpha;\eta)$ and $R_{{\rm EBCI},k}(\alpha;\eta)$ denote the oracle radii of the $k$-sided FLCI and EBCI, respectively, where $k=1,2$, and the two-sided EBCI allocates probability $\alpha/2$ to each tail. Then, as $\alpha\downarrow0$, we have}
\begin{align} \frac{ R_{{\rm EBCI},1}(\alpha;\eta) }{ R_{{\rm FLCI},1}(\alpha;\eta) } =1+o(1), \qquad \frac{ R_{{\rm EBCI},2}(\alpha;\eta) }{ R_{{\rm FLCI},2}(\alpha;\eta) } =\left(\frac{\log(2/\alpha)}{\log(1/\alpha)}\right)^{\frac{S}{2S+1}}(1+o(1)). \label{eq EBCI FLCI first order equivalence} 
\end{align} 
\end{proposition}
Proposition~\ref{prop EBCI FLCI sharpness} quantifies the price of empirical-Bernstein calibration. Under a common bias budget $\eta$, the one-sided oracle EBCI has the same first-order radius as its FLCI counterpart. For two-sided inference, EBCI is longer at a fixed confidence level because it allocates probability across two deviation bounds. This premium is explicit and vanishes in relative terms as $\alpha\downarrow0$. Thus, for any fixed and stringently small $\alpha$, empirical-Bernstein calibration only imposes mild efficiency loss relative to bias-aware FLCI calibration.

A noteworthy point is that the premium is very meaningful rather than merely technical. The incoming Proposition~\ref{prop FLCI normal approximation} and the simulation results in Section \ref{sec 5.3} isolate a finite-sample distributional-calibration channel: even when the absolute bias and normalizing factor are supplied oracle-wise, the stochastic component can still generate a first-order coverage error for folded-normal calibration. Meanwhile, in applications where skewed disturbances are widely seen and confidence intervals are used to make decisions, including risk management and “gatekeeping” tests, the relevant objective is not the shortest possible interval, but a credible decision rule. From this perspective, the slightly larger radius of EBCI is a necessary cost of protecting against the non-robustness of inferential conclusions induced by distributional calibration error.

Before we exhibit Proposition \ref{prop FLCI normal approximation}, we introduce the following mild assumption.

\begin{assumption}
    \label{as 7}
 Let $\theta(x_0)$ denote the value of a function $\theta\in \Theta$, where $\Theta$ is some given function class $\mathrm{(\text{e.g.,   $ \Theta_{x_0}(M_S,r_S,\psi_n)$})}$. Let  $\hat \theta_h(x_0)$ be an estimator of $\theta(x_0)$ with tuning parameter $h$.
    \begin{enumerate}[label=\textbf{A\arabic*}, ref=A\arabic*]
        \item\label{A1} There exists a known deterministic factor $\sigma_n(h)$ such that the distribution of the standardized statistic admits the representation
        \begin{align*}
            \P \Big (\frac{\hat \theta_h(x_0)-\E[\hat \theta_h(x_0)]}{\sigma_n(h) }\leq z \Big )=\Phi(z)+\frac{\kappa_{3,n}}{6}q(z)+r_n(z), 
        \end{align*}
        where, $q(z)=(1-z^2)\phi(z),$ $\phi$ denotes the density of the  standard normal distribution and, for any given $z$, 
        $(r_n(z))_{n \geq 1} $ is a sequence satisfying $\lim_{n\to\infty}|\frac{r_n(z)}{\kappa_{3,n}}|=0$. 
        \item \label{A2}  Let $b_{0}(\theta):= \E[\hat{\theta}_h(x_0)]-\theta(x_0)$ be the estimation bias and define $\sup_{\theta\in \Theta}|b_0(\theta)/\sigma_n(h)|=:t$. We assume that, for each $s\in\{-1,1\}$, there exists $\theta_s\in \Theta$ such that both $b_0(\theta_s)/\sigma_n(h)=st$. 
    \end{enumerate}
\end{assumption}
Conditions  \ref{A1} and \ref{A2} in  Assumption~\ref{as 7} can be understood as strengthened non-asymptotic versions of the Gaussian approximation and centering conditions used in \citeA{armstrong2020simple}. They are imposed solely to isolate the effect of normal-approximation error on coverage: they remove bias-bound, variance-estimation, and studentization errors, so that any remaining undercoverage is attributable solely to folded-normal calibration. Moreover, for \ref{A1}, $\kappa_{3,n}$ is usually the leading standardized third cumulant of the self-normalized stochastic component $Z_n:= \big ({\hat \theta_h(x_0)-\E[\hat \theta_h(x_0)]} \big )/{\sigma_n(h)}$. Thus, it measures the first-order departure of the sampling distribution of $Z_n$ from the standard normal distribution that is induced by skewness.  
In particular, suppose i.i.d data are collected from a regression model and the estimator admits the weighted representation $\hat\theta_h(x_0)-\E[\hat\theta_h(x_0)\mid X_1,\ldots,X_n] = \sum_{i=1}^n w_i\varepsilon_i$. Under some mild moment conditions of disturbance, together Lemmas \ref{lemma 1} and \ref{lemma 2}, we can show 
\begin{align}
\label{eq kappa3}
    \kappa_{3,n} = \frac{ \sum_{i=1}^n w_i^3 }{ \left( \sum_{i=1}^n w_i^2 \right)^{3/2} } \operatorname{skew}(\varepsilon_1) = \frac{ \sum_{i=1}^n w_i^3 }{ \left( \sum_{i=1}^n w_i^2 \right)^{3/2} } \frac{\E[\varepsilon_1^3]}{(\rm{Var}(\varepsilon_1))^{3/2}}.
\end{align}
Then, the next proposition rigorously quantifies the potential risk of undercoverage of FLCI caused by normal approximation error.

\begin{proposition}
    \label{prop FLCI normal approximation}
Under Assumption \ref{as 7}, we have
\begin{align}
\label{eq prop FLCI normal approximatio eq1}
  \inf_{\theta\in \Theta}  \P(\theta\in \mathrm{FLCI}(\alpha))\leq 1-\alpha-\frac{|\kappa_{3,n}|}{6}|\triangle_{\alpha}(t)|+o(|\kappa_{3,n}|),
\end{align}
where $\triangle_{\alpha}(t)=q(\operatorname{cv}_{1-\alpha}(t)-t)-q(\operatorname{cv}_{1-\alpha}(t)+t)$ and mapping $q(\cdot)$ is introduced in \ref{A2}. $\operatorname{cv}_{1-\alpha}(t)$ is the upper $\alpha$-quantile of  the folded normal distribution $|{\cal N} (t,1)|$ and $t$ is defined in Assumption \ref{as 7}. Furthermore, if $\theta_0$ satisfies $\left| \frac{\E[\hat\theta_h]-\theta_0}{\sigma_n(h)} \right| =t$ and $\mathrm{sgn}(\frac{\E[\hat{\theta}_h]-\theta_0}{\sigma_n(h)})\kappa_{3,n}\triangle_{\alpha}(t)<0$, Assumption \ref{A1} implies
\begin{align}
    \label{eq prop FLCI normal approximatio eq2}
     \P(\theta_0\in \mathrm{FLCI}(\alpha))\leq 1-\alpha-\frac{|\kappa_{3,n}|}{6}|\triangle_{\alpha}(t)|+o(|\kappa_{3,n}|).
\end{align}
\end{proposition}
Proposition \ref{prop FLCI normal approximation} shows that distributional-approximation error alone can generate a non-negligible first-order contribution to the undercoverage error of folded-normal calibration. The effect remains visible even when the absolute bias and normalizing factor are supplied oracle-wise. Hence, the proposition should be read as a diagnostic statement about the stochastic calibration engine behind FLCI, rather than as a critique of worst-case bias control. The resulting uniform coverage deficit is at least $\frac{|\kappa_{3,n}|}{6}|\triangle_\alpha(t)|$. By  \eqref{eq kappa3} the quantity  
$\kappa_{3,n}$ is determined by the skewness of $\varepsilon_1$  for weighted regression estimators with i.i.d.\ disturbances. Therefore, the leading coverage error is nonzero whenever the disturbance is skewed
and the cubic weight sum is nonzero. The effect is nondegenerate in the standard setting considered by \citeA{armstrong2020simple}. For $S=2$, the RMSE-optimal bandwidth yields $t^*=\frac{1}{2}$. At $\alpha=0.05$,
$c_\alpha(1/2)=2.181$, we have
$\triangle_\alpha(1/2)\approx-0.11\neq0$. 
The inequality \eqref{eq prop FLCI normal approximatio eq2} in Proposition~\ref{prop FLCI normal approximation}  shows that, when the signed bias of a particular function  $\theta_0$ has the adverse direction, undercoverage occurs at this specific function rather than only at an unspecified least-favorable element of $\Theta$. Thus, the problem cannot be dismissed as an artifact of taking an minimum over a too broad function class. Moreover, this potential risk remains even under oracle knowledge of the exact absolute bias and the normalizing factor. Under this diagnostic oracle setup, the remaining coverage term is attributable to the folded-normal calibration step. Under skewed disturbances, the adverse sign configuration can therefore produce a first-order coverage loss. For some numerical illustration of this phenomenon of FLCI, we refer to Section \ref{sec 5.3}.

Proposition \ref{prop FLCI normal approximation} indicates that the main distinction between EBCI and FLCI lies in stochastic calibration. Thus, FLCI and EBCI are two complementary tools of bias-aware inference. EBCI is attractive when a transparent finite-sample coverage guarantee is paramount. Whereas FLCI is preferable when the normal approximation is very accurate.

\section{Simulations}
\label{sec 5}

The simulation presented in this section should be read as verification of the discussions and results presented in Sections \ref{sec 2}, \ref{sec 3}, and \ref{sec 4}, rather than as a competition.

\subsection{Regression with Sufficient Smoothness}
\label{sec 5.1}
The first set of simulations considers regression designs in which the
smoothness conditions required for coverage-error-optimal RBC are satisfied. The purpose is to examine whether correctly identified smoothness is effectively translated into finite-sample coverage by standard-normal calibration. The comparison implements the smoothness-allocation mechanism discussed in
Section~\ref{sec 4.1}. The two procedures intentionally use local-polynomial centers with different orders because they convert the same smoothness budget through different calibration principles. We compare SNC-based CE-optimal RBC \cite{calonico2022coverage} with the eta-free EBCI of Theorem~\ref{th eta free regression}. EBCI handles stochastic fluctuations on the original estimation scale and incorporates deterministic approximation error directly into its radius. Throughout, it is implemented with the minimal smoothness level $S=p+2$, matching the smoothness requirement used in the RBC comparison. Its practical radius uses $2\log(1/\alpha)$ at interior points and $2\log(2/\alpha)$ at boundary points.

We consider a DGP of the form
\begin{align}
  Y_i = m(X_i) + \varepsilon_i, \qquad
  m(x) = \sum_{j=1}^{3} a_j x^{j} + d \cdot \max(x,0)^{3+\delta} + c,\quad \varepsilon_i \sim N(0,1),
  \label{eq:cusp_dgp}
\end{align}
with  $a_1=1, a_2=2, a_3=4$, intercept
$c = 1$, and cusp amplitude $d = 24$. The predictors satisfy either $X_i \sim U[-1,1]$ or  $X_i \sim  U[0,1]$ corresponding to an interior or boundary point $x_0=0$, respectively.
The one-sided power term $\max(x,0)^{3+\delta}$ introduces a cusp at the
evaluation point $x_0$: the function is $(3+\delta)$-H\"{o}lder smooth at
$x_0$, so as $\delta \to 0$ it approaches the boundary of the smoothness
class assumed by the estimator. 

\begin{figure}[H]
  \centering
\includegraphics[width=0.80\textwidth,height=0.35\textheight,keepaspectratio]{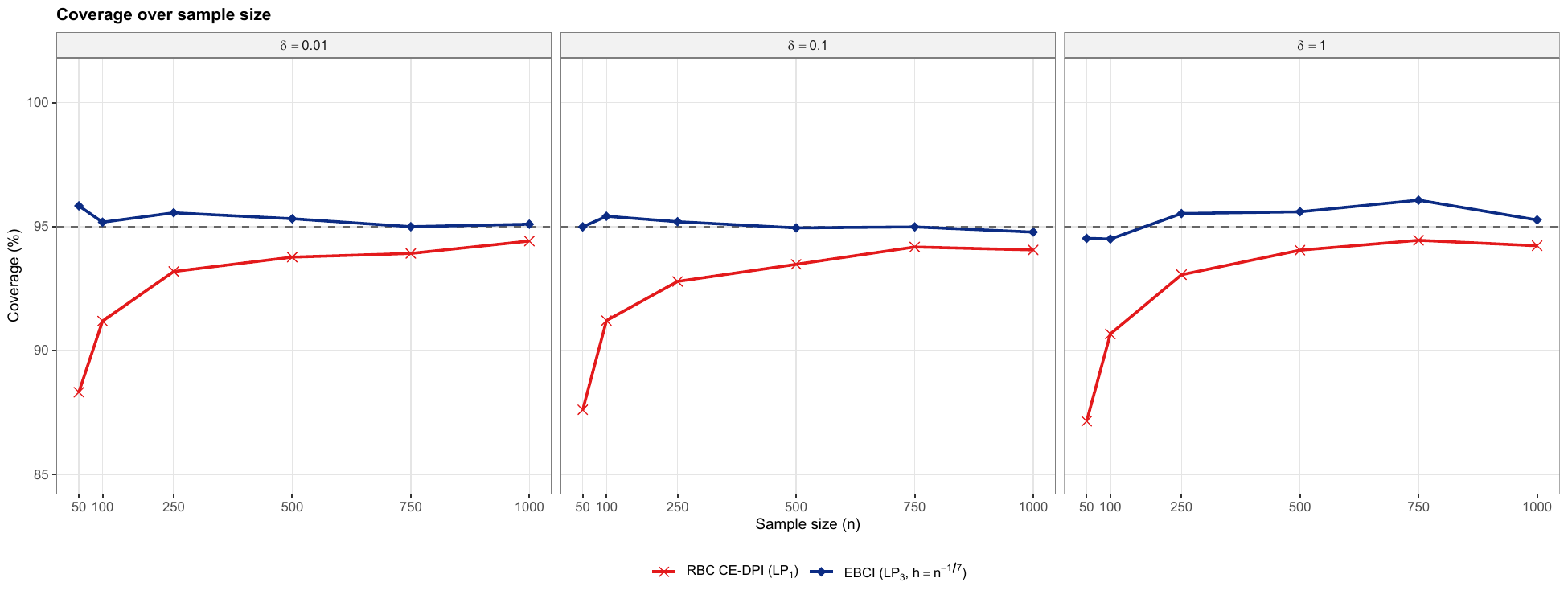}
  \vspace{0.5cm}
\includegraphics[width=0.80\textwidth,height=0.35\textheight,keepaspectratio]{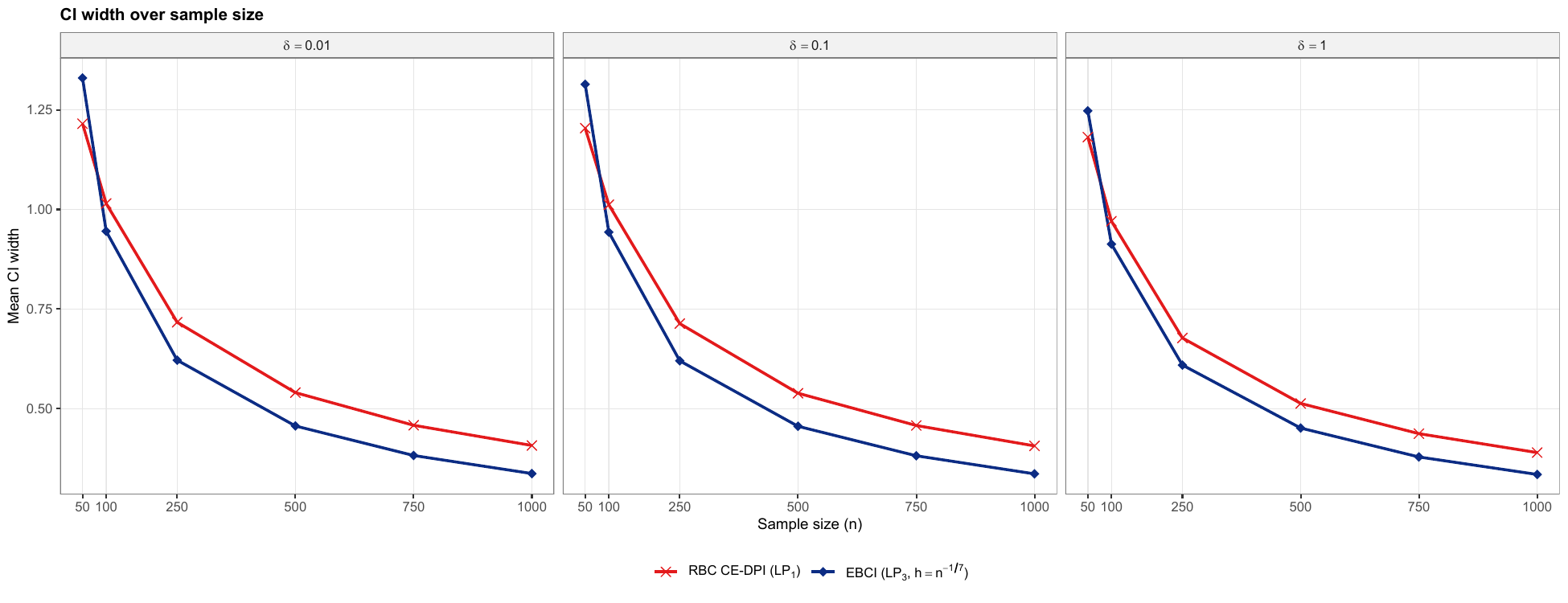}
  \caption{ \it Empirical coverage probabilities (upper panel) and  average confidence interval widths (lower panel) of SNC-based CE-optimal RBC 
  (red) and  eta-free EBCIs (blue). The DGP is the   
polynomial-cusp model \eqref{eq:cusp_dgp} with  $X \sim U[-1,1]$ corresponding to the interior point $x_0 = 0$.}
  \label{fig:cusp_int_norm_combined}
\end{figure}

Throughout, we use $B = 10{,}000$ Monte Carlo replications, the
Epanechnikov kernel, assumed polynomial order of at least $S = 3$, and nominal coverage level $1 - \alpha = 0.95$.
Following Section~6 of \citeA{calonico2022coverage}, we apply a local linear version of RBC. The bandwidth for RBC is chosen as the data-driven version of the inference-optimal bandwidth $\hat{h}_{\text{rbc}}$ based on the \texttt{nprobust}-package \citeA{calonico2019nprobust}  using the optimal choice of $\rho$ based on Table 2 of \citeA{calonico2022coverage}. Since the RBC benchmark uses local linear estimation and its CE-optimal theory requires smoothness $p+2=3$, we implement EBCI with local-polynomial order $S=3$. For further details on the comparison of RBC and EBCI, see Section~\ref{sec 4.1}.

For EBCI, we consider the practical $\eta$-free recommendation $h_{\mathrm{nv}}=n^{-1/7}$ introduced in Theorem~\ref{th eta free regression}. For further details, see Section~\ref{sec 3.3}. Meanwhile, we choose $\hat{r}(\alpha)$ for interior and $\hat{r}(\alpha/2)$ for boundary points, since the diverging factor $d_n$ is user-defined and can absorb a fixed constant. This corresponds to different fixed choices of $\eta$, which are both valid as long as $\eta$ is fixed, but effectively require different sample sizes. As boundary point estimation clearly utilizes fewer observations, we rely on the more conservative choice. We consider sample sizes $n \in \{50, 100, 250, 500, 750, 1000\}$.
Figures~\ref{fig:cusp_int_norm_combined} and \ref{fig:cusp_bdy_norm_combined} display empirical coverage probabilities as well as confidence interval width for interior and boundary points, by varying between $X_i \sim U[-1, 1]$ and $X_i \sim U[0, 1]$,  as functions of sample size.
\begin{figure}[!h]
  \centering
\includegraphics[width=0.85\textwidth,height=0.35\textheight,keepaspectratio]{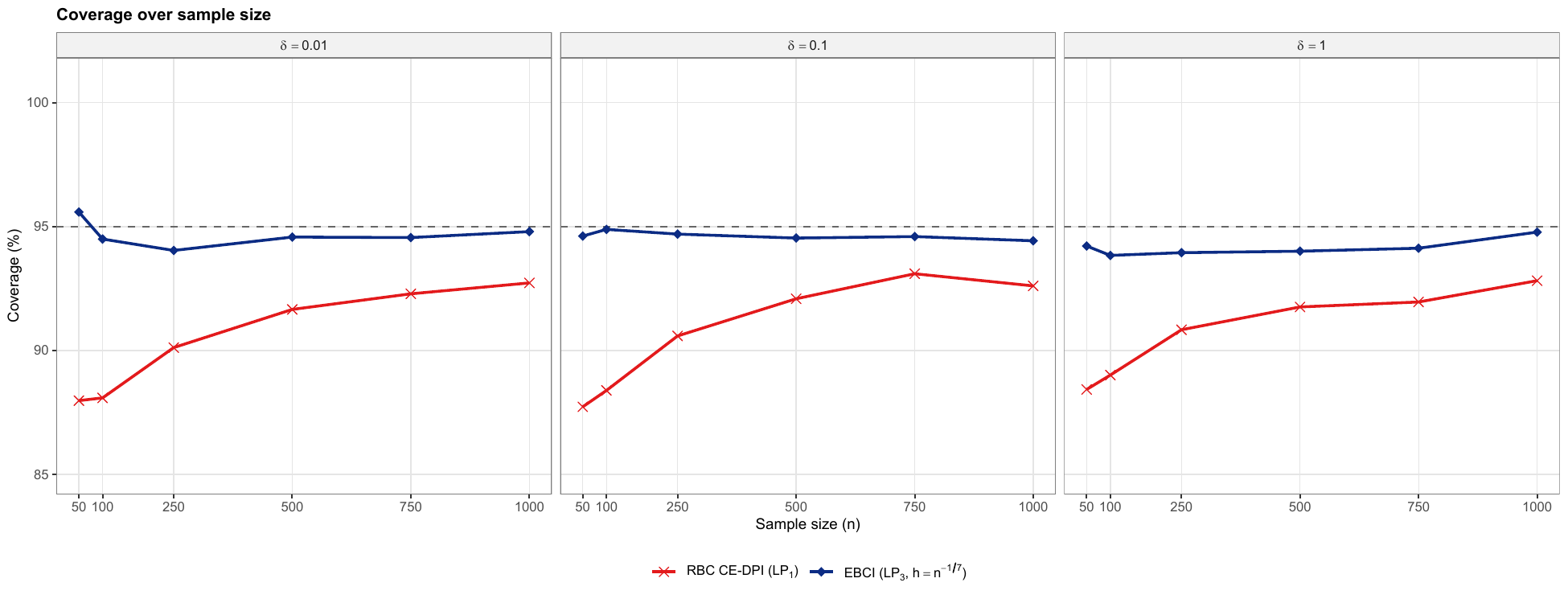} 
  \vspace{0.5cm}
\includegraphics[width=0.85\textwidth,height=0.35\textheight,keepaspectratio]{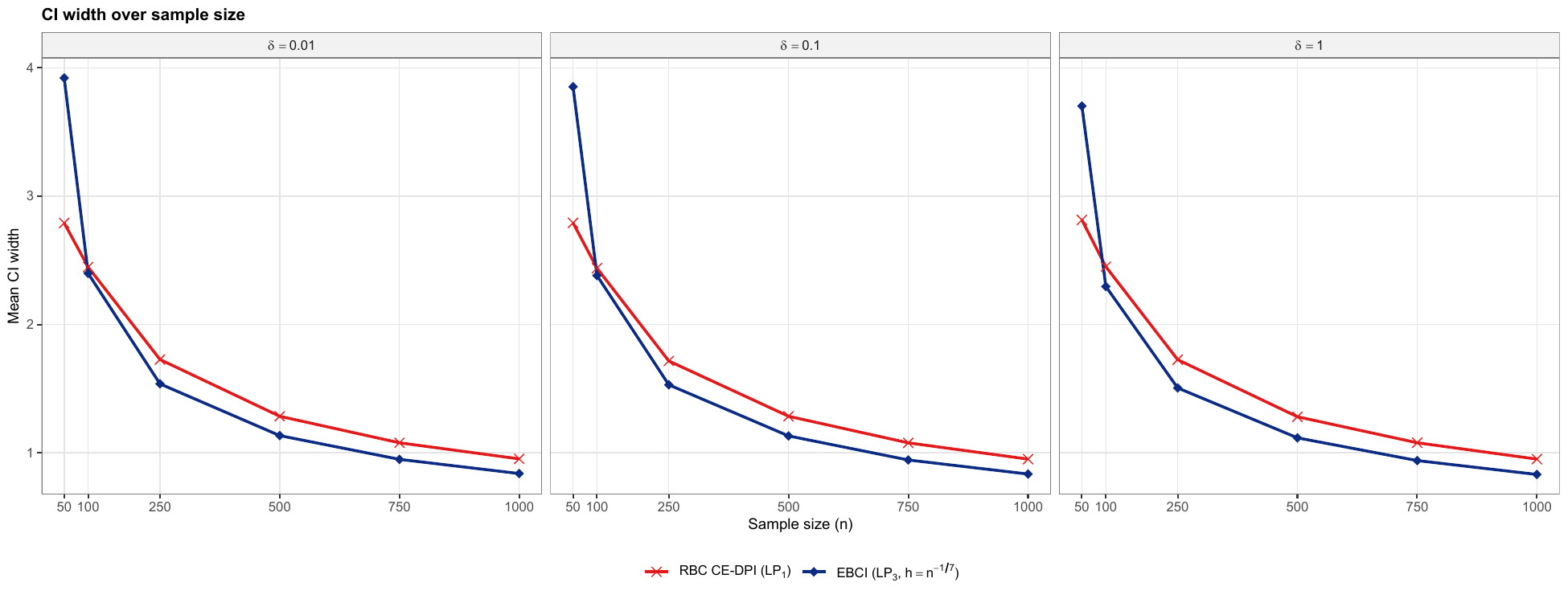}
  \caption{ \it Empirical coverage probabilities (upper panel) and  average confidence interval widths (lower panel) of SNC-based CE-optimal RBC 
  (red) and  eta-free EBCIs (blue). The DGP is the   
polynomial-cusp model \eqref{eq:cusp_dgp} with  $X \sim U[0,1]$ corresponding to the boundary point $x_0 = 0$.}
  \label{fig:cusp_bdy_norm_combined}
\end{figure}
\noindent In the polynomial-cusp designs, where the regression function approaches the boundary of the maintained smoothness class as $\delta$ decreases, SNC-based RBC exhibits the undercoverage pattern predicted by the normalized-bias mechanism. EBCI, despite the conservative settings of implementation described above, keeps coverage close to nominal and its interval length becomes comparable to or shorter than RBC as $n$ grows. 

Therefore, it would be of great interest to check whether similar phenomenon could be observed when the function is infinitely differentiable at the evaluation point. Based on Section~6 of \citeA{calonico2022coverage}, we generate
data from
\begin{align}
    \label{eq:ccf_dgp}
  Y_i = m(X_i) + \varepsilon_i, \qquad
  m(x) = \frac{\sin\!\left(3\pi x/2\right)}
                {1 + 18x^2\bigl(\operatorname{sign}(x) + 1\bigr)},
  \quad X_i \sim U[-1, 1], \quad \varepsilon_i \sim N(0,1),
\end{align}
where $\operatorname{sign}(x) = -1$, $0$, or $1$ for $x < 0$, $x = 0$,
or $x > 0$, respectively. Sample sizes are $n \in \{100, 250, 500, 750, 1000, 2000\}$, matching \citeA{calonico2022coverage}.
As in \citeA{calonico2022coverage}, we evaluate four interior points $x_0 \in \{-0.6, -0.2, 0.2, 0.6\}$, and the two boundary points $x_0 \in \{-1, 1\}$.
As for the polynomial-cusp DGP, the CE-optimal RBC bandwidth is based on the optimal $\rho^* = 0.865$ at interior points and $\rho^* = 0.898$ at boundary points, following Table~2 of \citeA{calonico2022coverage} for the Epanechnikov kernel.

\begin{figure}[H]
  \centering
    \includegraphics[width=0.95\textwidth,height=0.35\textheight,keepaspectratio]{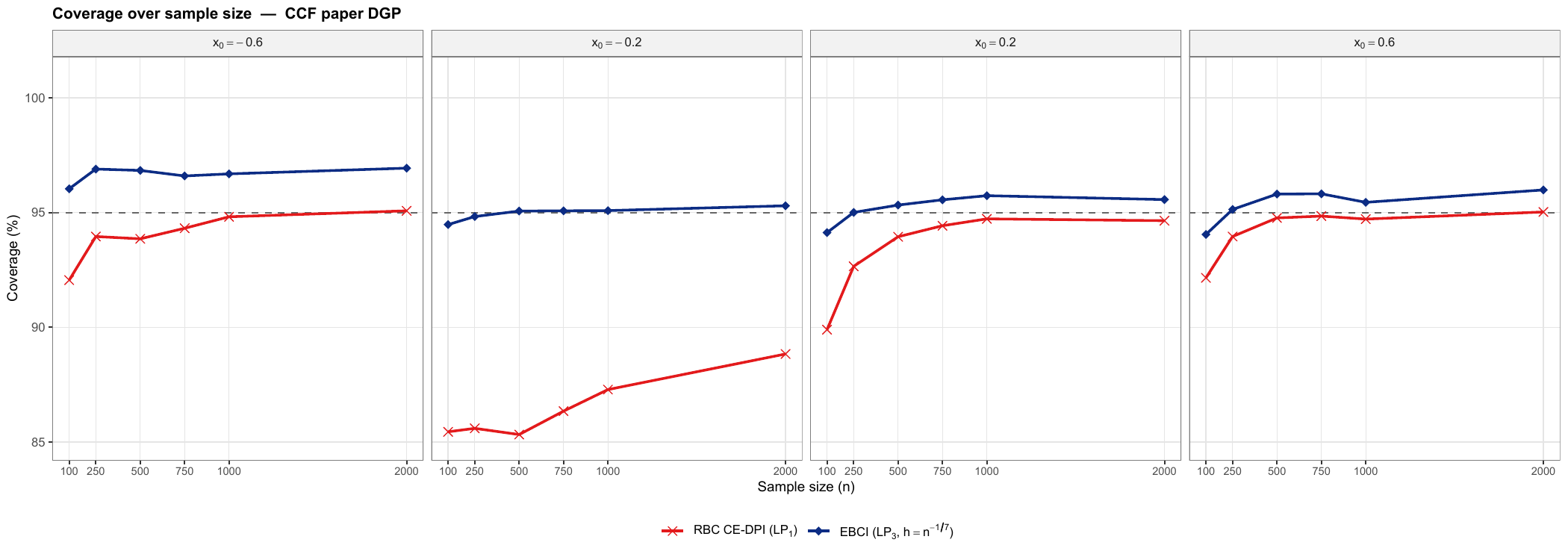}
  \vspace{0.5cm}
    \includegraphics[width=0.95\textwidth,height=0.35\textheight,keepaspectratio]{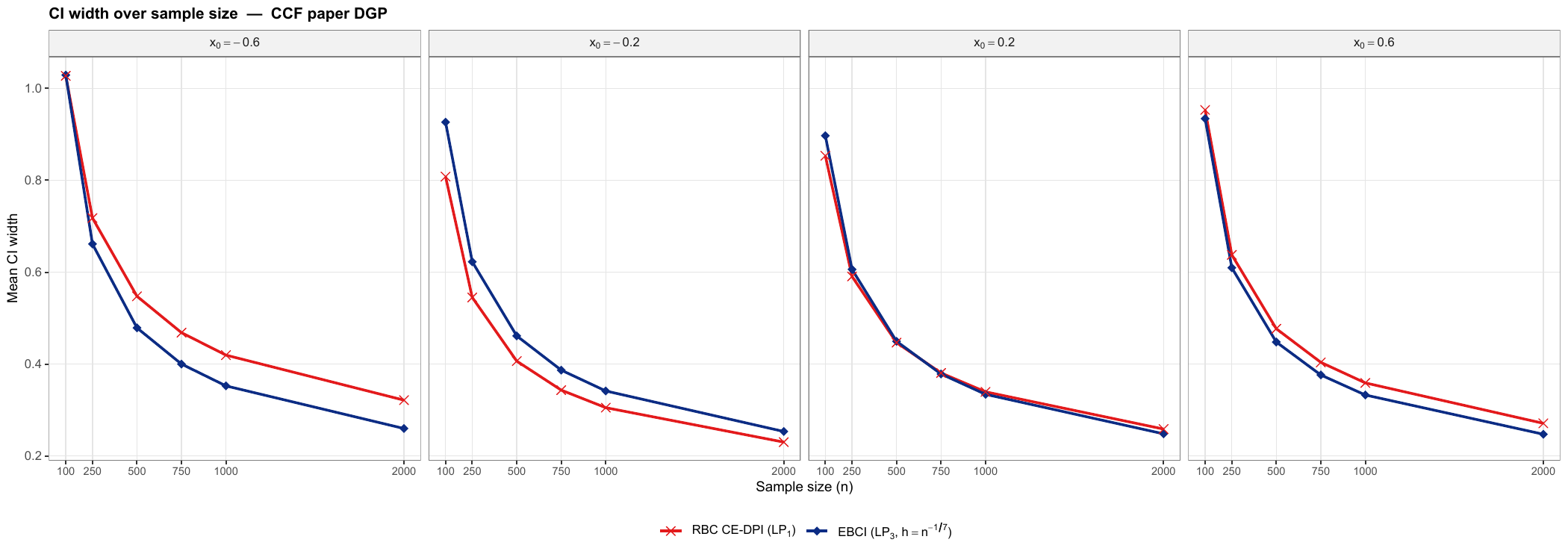}
  \caption{ \it Empirical coverage probabilities (upper panel) and  average confidence interval widths (lower panel) of SNC-based CE-optimal RBC 
  (red) and  eta-free EBCIs (blue). The DGP \eqref{eq:ccf_dgp} from Calonico et al. (2022) is evaluated at interior points.}
  \label{fig:ccf_paper_combined}
\end{figure}

Figures~\ref{fig:ccf_paper_combined} and \ref{fig:ccf_paper_bdy_combined} display empirical coverage probabilities as well as confidence interval width for interior and boundary points. EBCI highlights very robust and honest coverage for interior points, specifically including $x_0=-0.2$, which seems to be quite hard for RBC. Further, the confidence interval is similar, sometimes even improving upon RBC. For boundary points, the results are mixed, but very close for all methods.

\begin{figure}[H]
  \centering
\includegraphics[width=0.65\textwidth,height=0.35\textheight,keepaspectratio]{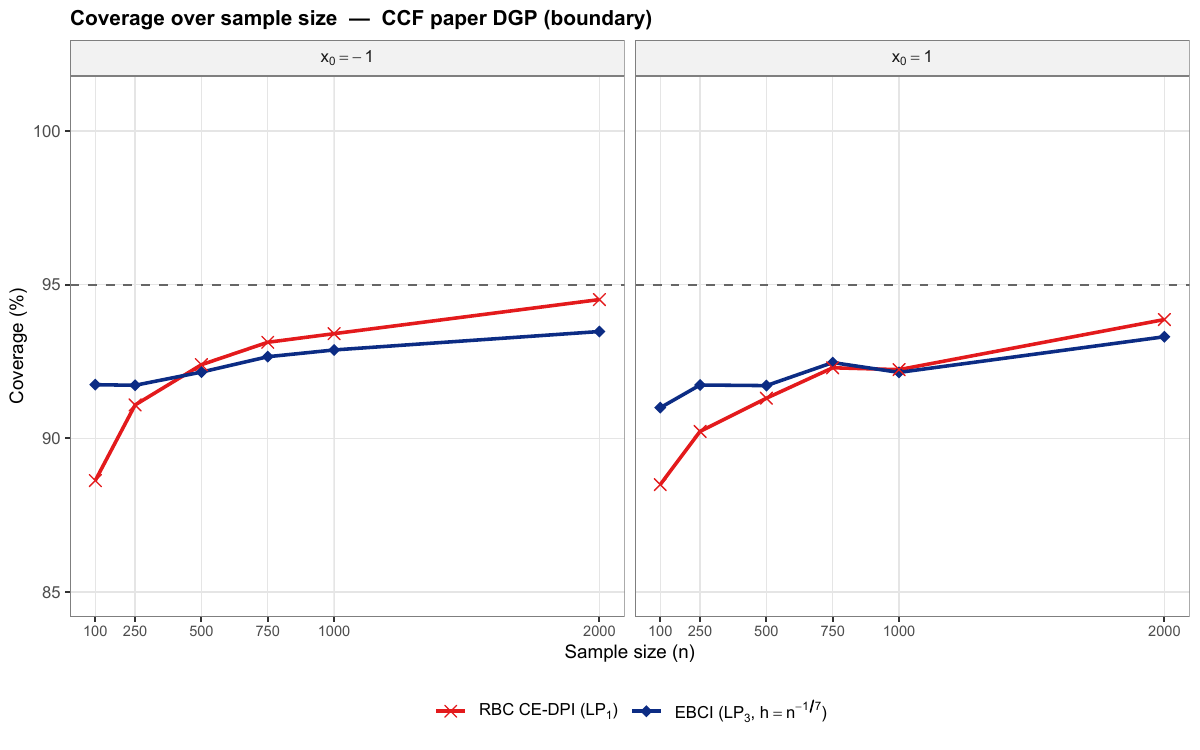}
  \vspace{0.5cm}
\includegraphics[width=0.65\textwidth,height=0.35\textheight,keepaspectratio]{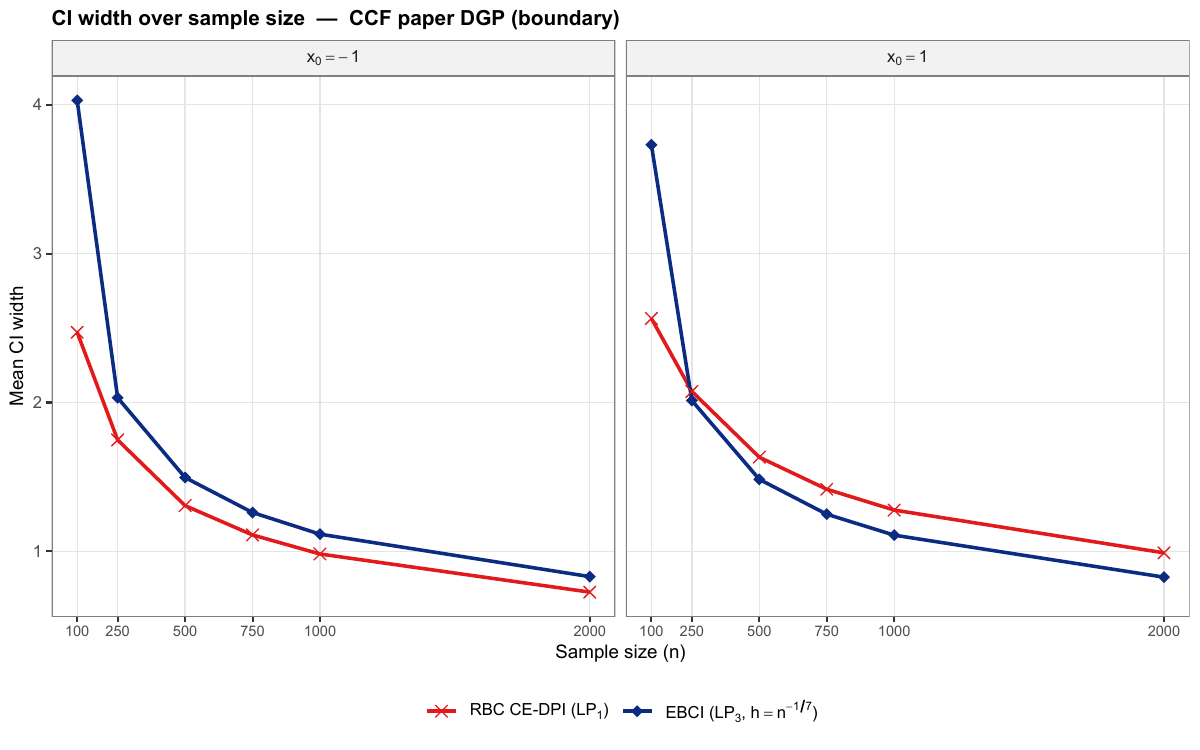}
  \caption{\it Empirical coverage probabilities (upper panel) and  average confidence interval widths (lower panel) of SNC-based CE-optimal RBC 
  (red) and  eta-free EBCIs (blue). The DGP \eqref{eq:ccf_dgp} from Calonico et al. (2022) is evaluated at boundary points.}
  \label{fig:ccf_paper_bdy_combined}
\end{figure}

Overall, even under sufficient smoothness, SNC-based RBC can produce intervals that are comparable to, or longer than, the corresponding EBCI intervals without delivering better coverage. Thus, the issue is not a failure to exploit the available smoothness or a failure of bias correction per se. Rather, standard-error-scale calibration can leave residual normalized bias insufficiently accommodated. In contrast, empirical-Bernstein calibration more effectively converts identified smoothness into both finite-sample coverage and interval efficiency.

\subsection{Regression with Small Noise and Limited Smoothness}
\label{sec 5.2}
Recall that, according to \citeA{schucany1977improvement}, building a confidence interval using a studentized Richardson-extrapolation-type estimator (REE) is a prominent example of an SNC-based confidence interval with exact leading bias correction. Both Proposition \ref{prop RE cov error} and panel (b) of Figure \ref{fig:small-noise-ree} have shown that, when the regression model has limited smoothness and small noise, the standard-normal calibrated confidence interval for REE would suffer from poor coverage. Therefore, this section presents a numerical pressure test for the honesty of the $\eta$-free EBCI introduced in Theorem \ref{th eta free regression}.

We consider the following regression design in which the disturbance has very small variance, and the regression function is only marginally smoother than the $C^2$ smoothness level used by both procedures. 
\begin{align}
    Y_i=m(X_i)+\varepsilon_i,\ m(x)=0.1+x+2x^2+d\max(x,0)^{2+0.01},\ \varepsilon\sim N(0, 0.003^2).
      \label{eq:cusp_dgp 2}
\end{align}
We compare REE-type CI with our eta-free EBCI for target parameter $m(0)$ by using the same Richardson-extrapolation-type local linear estimator, the same bandwidth $h=n^{-\frac{1}{5}}$, and the same $C^2$ smoothness specification. Coefficient $d$ controls the strength of the cusp component and hence the
difficulty created by the remaining bias. Since the noise level and the
weights are fixed across values of $d$, changing $d$ affects coverage through bias difficulty rather than through the stochastic scale or interval radius. In particular, we choose $d\in \{3,6,10\}$ and $n\in\{100, 500, 1000, 1500, 2000\}$. 
  \begin{figure}[!h]
    \centering
\includegraphics[width=\textwidth]{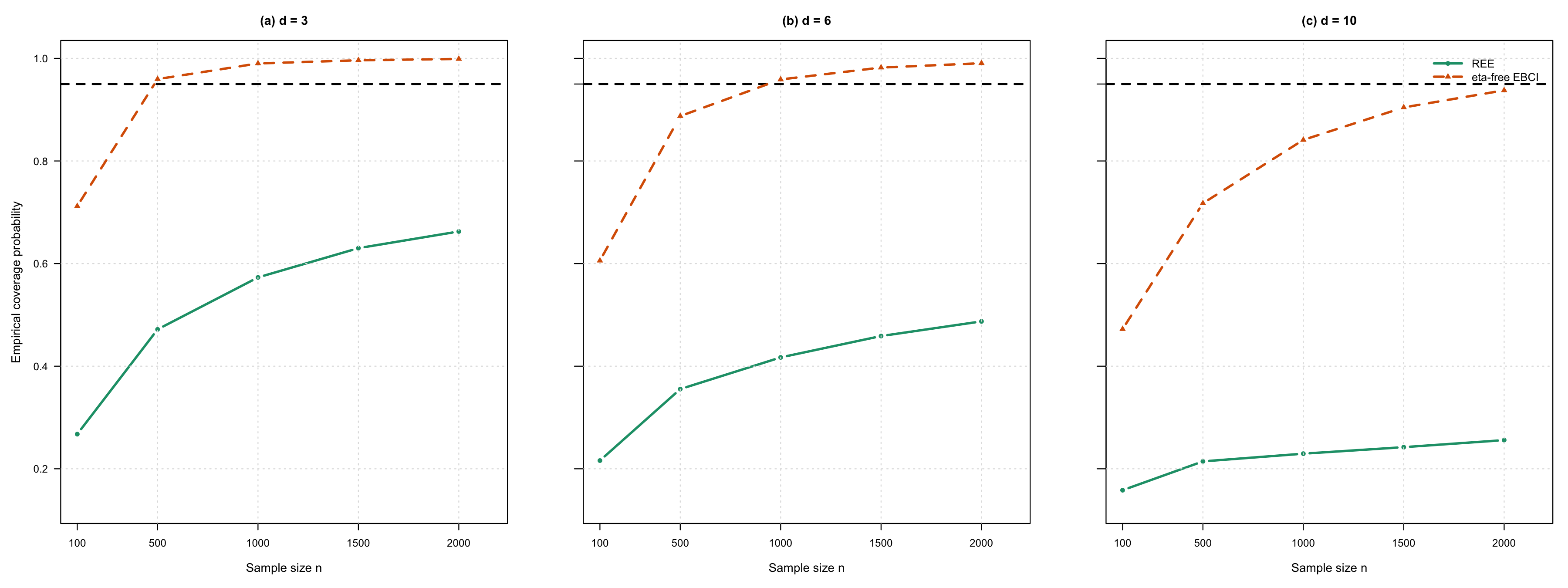}
 \caption{\it Empirical coverage probabilities of REE-type CI (solid, green) and eta-free EBCI (dashed, brown) under the polynomial-cusp model \eqref{eq:cusp_dgp 2}. The panels refer to different strength values of the cusp component $d$.}
    \label{fig:REE vs EBCI}
  \end{figure}

  \begin{table}[!h]
    \centering
    \begin{tabular}{c c c}
        \toprule
        $n$ & REE & eta-free EBCI \\
        \midrule
        100  & 0.001803 & 0.005512 \\
        500  & 0.000915 & 0.002865 \\
        1000 & 0.000690 & 0.002178 \\
        1500 & 0.000585 & 0.001857 \\
        2000 & 0.000521 & 0.001658 \\
        \bottomrule
    \end{tabular}
   \caption{\it Average half-lengths of REE-type CI and eta-free EBCI.}
      \label{tab: ree-ebci-small-noise-length}
\end{table}
A noteworthy point here is that both procedures use the same REE estimator and exploit the same $C^2$ smoothness specification. Hence, the difference in half-length in Table \ref{tab: ree-ebci-small-noise-length} is solely due to the calibration of the stochastic radius. In particular, both radii shrink at the same basic $n^{-2/5}$ scale, up to the minor inflation factor $ n^\tau $ ($\tau=0.01$) in eta-free EBCI. When sample size is small or moderate, the length comparison isolates the difference in the constant, $\sqrt{2\log(\frac{2}{\alpha})}$ and $z_{1-\frac{\alpha}{2}}$. However, according to the coverage shown in Figure \ref{fig:REE vs EBCI}, the slim conservativeness caused by this constant difference is very necessary for honesty. As $d$
increases, the remaining bias difficulty rises while the stochastic scale and the two interval radii remain unchanged. REE then exhibits increasingly severe undercoverage, whereas eta-free EBCI remains stable. Thus, under the same estimator and the same correctly exploited smoothness, empirical-Bernstein calibration converts a constant-scale increase in radius into a first-order improvement in finite-sample coverage.

\subsection{Regression with Skewed Disturbances}
\label{sec 5.3}
We next consider regression designs with strongly skewed disturbances.  Bias-aware FLCIs explicitly accommodate a non-negligible absolute bias through folded-normal calibration. Proposition \ref{prop FLCI normal approximation} points out that their finite-sample coverage, however, still relies on the accuracy of a Gaussian approximation for the standardized estimator. In particular, under skewed disturbances, this distributional approximation can be inaccurate even when bias is handled oracle-wise.

To isolate this issue, we also equip FLCIs with the exact finite-sample absolute
bias and the oracle self-normalizing factor. Hence, any coverage distortion
cannot be attributed to bias-bound estimation, bandwidth plug-in, or variance estimation. The comparison instead isolates the stochastic calibration step: folded-normal calibration for FLCIs versus empirical-Bernstein tail control for EBCIs. Moreover,  we first compare fixed-$\eta$ EBCIs and oracle FLCIs under a common known
deterministic bias budget. This comparison asks whether empirical-Bernstein calibration improves finite-sample reliability when both procedures are given the same oracle bias information. We then compare $\eta$-free EBCIs, which do not use oracle bias knowledge, with oracle FLCIs. The latter comparison assesses whether the practical bias-free EBCI construction can retain reliable coverage. 

Based on \eqref{eq local polynomial}, we consider local linear estimator $\sum_{i=1}^nW_{ih}(x)Y_i$ for the following $n$-dependent fixed-design quadratic regression model,
\begin{equation}
\label{eq:dgp_skewed}
\begin{split}
    Y_i =36(x_i-1)^2\mathrm{sign}(W_{ih}(1)) + \varepsilon_i,
    \qquad \varepsilon_{i}=\frac{(B_i-p)/\sqrt{p(1-p)}+0.1 Z_i}{\sqrt{1+0.1^2}},  \\
    x_i=-1+\frac{2(i-1/2)}{n},\qquad B_i\sim\operatorname{Bernoulli}(p),\qquad Z_i\sim N(0,1),
\end{split}
\end{equation}
where the target parameter is $m_n(1)$ and sample-size grid is $ n\in\{100,300,500, 1000, 1500, 2000, 2500\}$. For DGP, we let  $\{(B_i,Z_i)\}_{i=1}^n$ be i.i.d. and $B_{i}$ be independent of $Z_i$ for all $ i\leq n$. Consequently, we have $\E[\varepsilon_i]=0$, $\text{Var}(\varepsilon_i)=1$ and
\begin{align*}
    \mathrm{skew}(\varepsilon_i)=\frac{1-2p}{\sqrt{p(1-p)}(1+0.1^2)^{3/2}}.
\end{align*}
Thus the disturbance has strong skewness when $p$ is small. Here we would choose $p\in \{0.005, 0.010, 0.025\}$. Additionally, some simple algebra shows that, for a given local linear smoother, the exact finite-sample worst-case bias bound and the oracle standard deviation, denoted as $B_h(x)$ and $\sigma_n(h)$, are 
\begin{align*}
    B_h(1)=36\sum_{i=1}^n|W_{ih}(1)|(x_i-1)^2, \ \sigma_n(h)=\sqrt{\frac{C_V}{n^{4/5}}},\ C_V=9.8.
\end{align*}
For FLCI, we implement the RMSE-optimal local-linear FLCI introduced in \citeA{armstrong2020simple}. We supply the exact finite-grid absolute bias and the oracle self-normalizing factor and select the bandwidth so that their ratio equals $1/2$. The resulting two-sided $95\%$ interval uses the corresponding folded-normal critical value $2.181$. For the fixed-$\eta$ EBCI, we set $\eta=\frac{3M}{32}=6.75,$ which is the boundary local-linear triangular-kernel Taylor-envelope constant under the $C^2$ specification. We use $S=2$, $\xi_n=0$, and the same oracle variance constant $C_V=9.8$ as for the FLCI.
\begin{figure}[!h]
    \centering
\includegraphics[width=\textwidth]{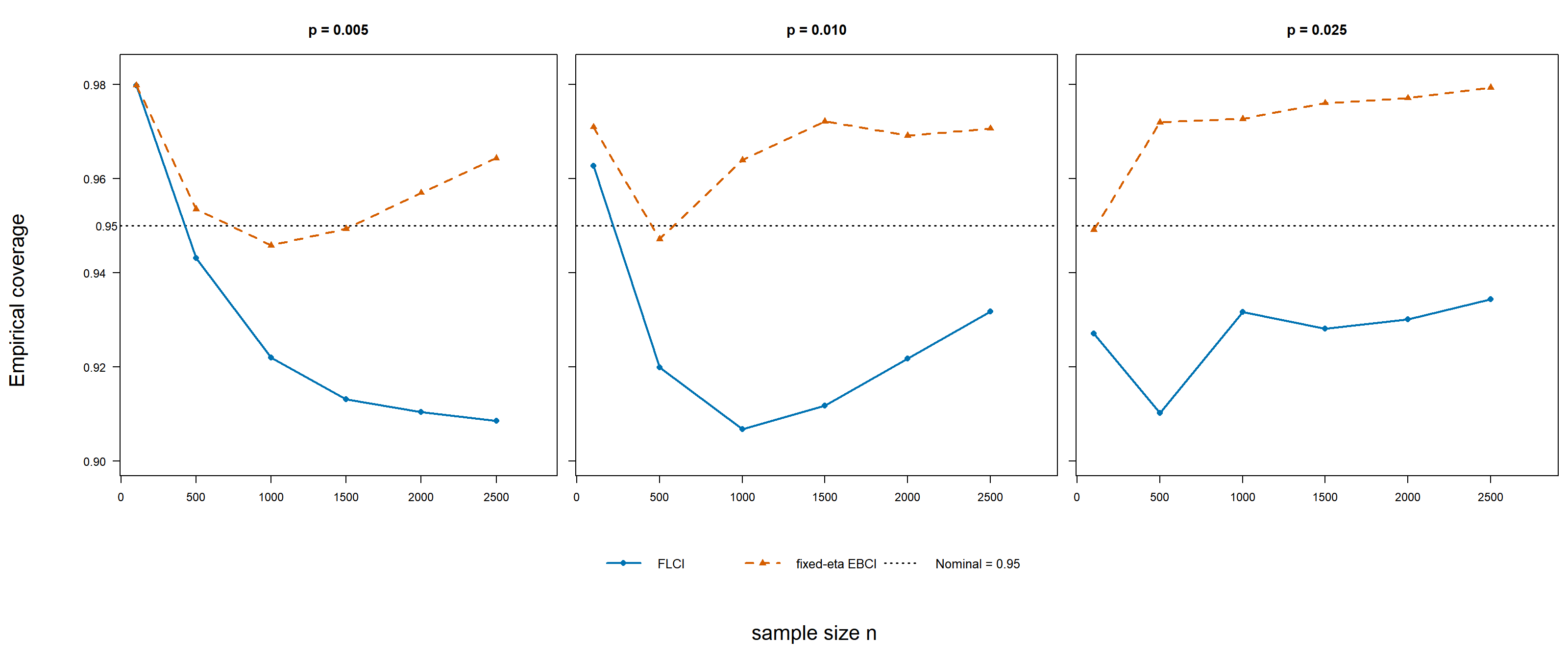}
 \caption{\it Empirical coverage probabilities of FLCI (solid, blue) and fixed-$\eta$ EBCI (dashed, brown) with oracle bias under the skewed disturbance model \eqref{eq:dgp_skewed}. The panels correspond to different degrees of skewness, determined by the parameter $p$.}
    \label{fig:FLCI vs fixed-eta EBCI cov}
\end{figure}

\begin{figure}[!h]
    \centering
\includegraphics[width=0.65\textwidth,height=0.35\textheight,keepaspectratio]{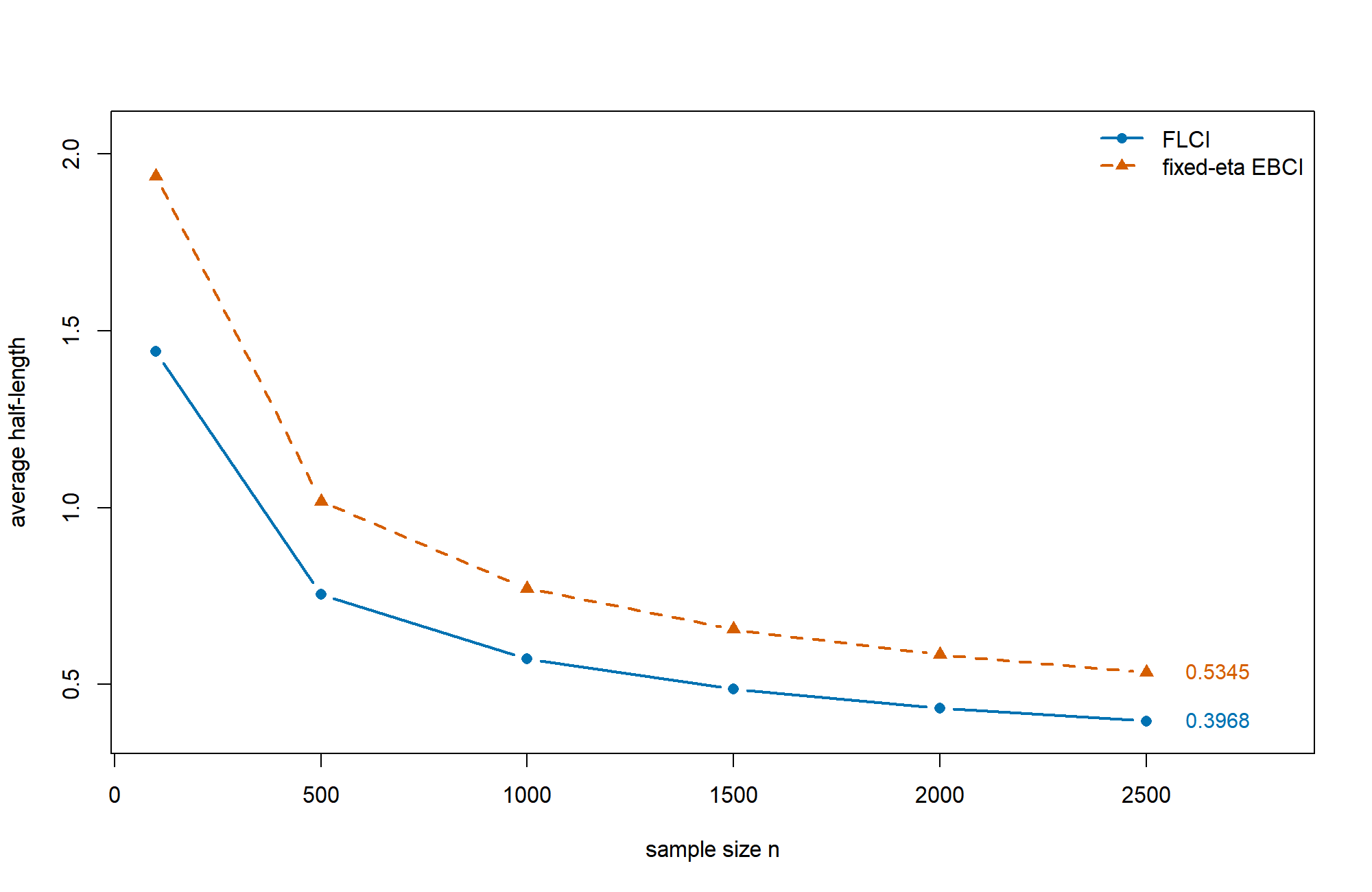}
 \caption{\it Average half-lengths of FLCI (solid, blue) and fixed-$\eta$ EBCI (dashed, brown) with oracle bias.}
    \label{fig:FLCI vs fixed-eta EBCI length}
\end{figure}

Figure \ref{fig:FLCI vs fixed-eta EBCI cov} shows that knowing the exact bias does not remove normal-approximation error. The FLCI is given the exact finite-grid bias and the oracle variance constant, yet it still undercovers when the disturbances are skewed. As pointed out by Proposition \ref{prop FLCI normal approximation}, the reason is simple: once the FLCI allows for a nonzero bias, its folded-normal calibration becomes asymmetric, so the leading skewness term in the normal-approximation error no longer cancels. Smaller values of $p$ generate more skewed disturbances, and the resulting FLCI undercoverage becomes more severe. The following Figure \ref{fig:FLCI vs fixed-eta EBCI length} implies that fixed-$\eta$ EBCI is slightly more conservative, as expected and as discussed in Proposition \ref{prop EBCI FLCI sharpness}. But the coverage plots show why this extra length is needed: it provides the protection required to maintain honesty, especially in small samples, when Gaussian calibration is unreliable.

The preceding comparison shows that exact oracle bias information does not
eliminate the normal-approximation error of FLCI calibration. The fixed-$\eta$ EBCI, however, also relies on a known deterministic bias budget. We therefore
next compare the same oracle FLCI with the $\eta$-free EBCI, which does not use
oracle bias information. This comparison deliberately favors the FLCI. It is supplied with the exact
finite-sample absolute bias and the oracle variance constant, whereas the
$\eta$-free EBCI is implemented without knowledge of the bias bound. The purpose is therefore not to compare two procedures with identical information sets. Rather, it asks whether reliable finite-sample coverage can still be achieved without oracle bias knowledge when the competing FLCI is given the strongest possible bias-aware advantage. The comparison shows that oracle knowledge of bias cannot repair inaccurate Gaussian calibration, while $\eta$-free empirical-Bernstein calibration can retain reliable coverage without requiring such oracle information.

\begin{figure}[!h]
    \centering
\includegraphics[width=\textwidth]{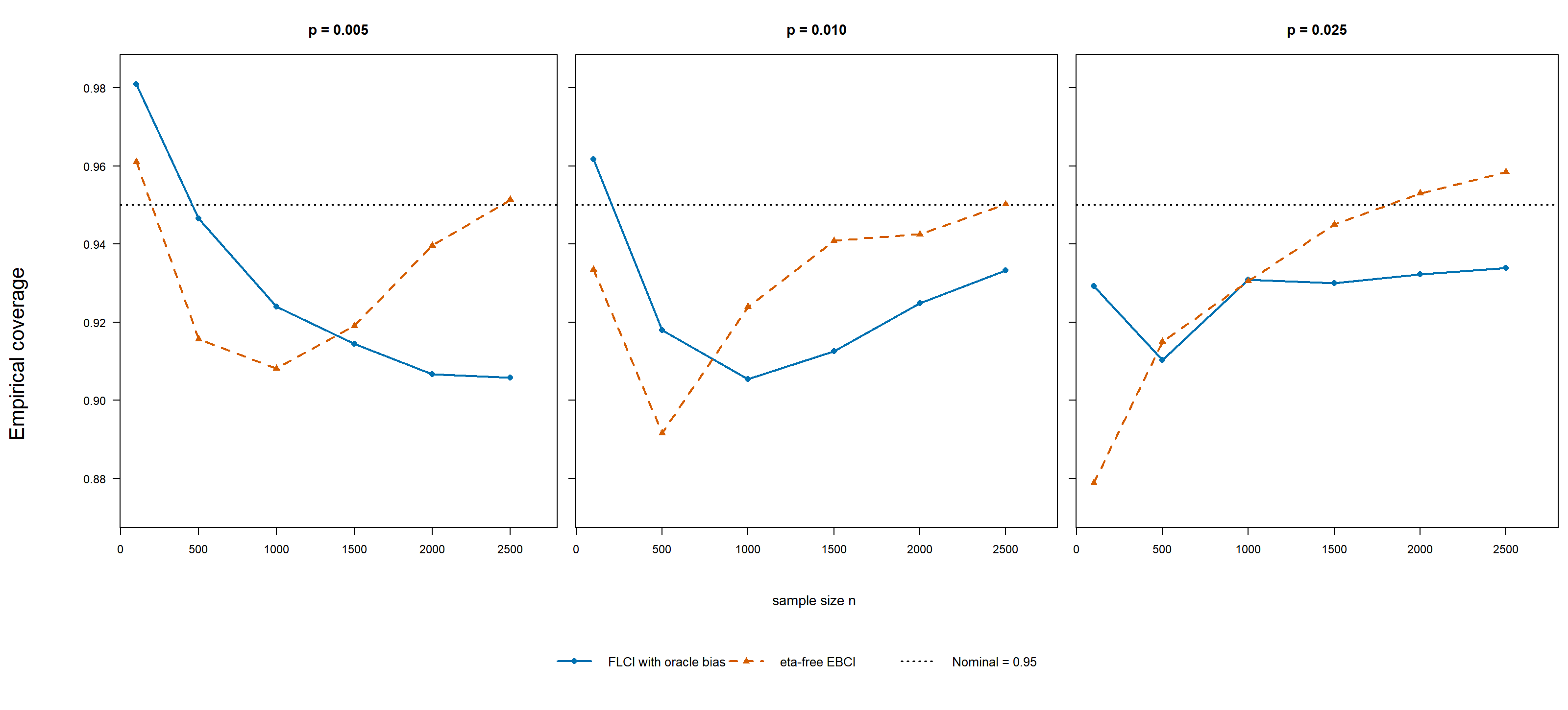}
 \caption{\it Empirical coverage probabilities of FLCI with oracle bias (solid, blue) and $\eta$-free EBCI (dashed, brown) under the skewed disturbance model \eqref{eq:dgp_skewed}. The panels correspond to different degrees of skewness, determined by the parameter $p$. }
    \label{fig:FLCI vs eta-free EBCI cov}
\end{figure}

\begin{figure}[!h]
    \centering
\includegraphics[width=0.75\textwidth,height=0.35\textheight,keepaspectratio]{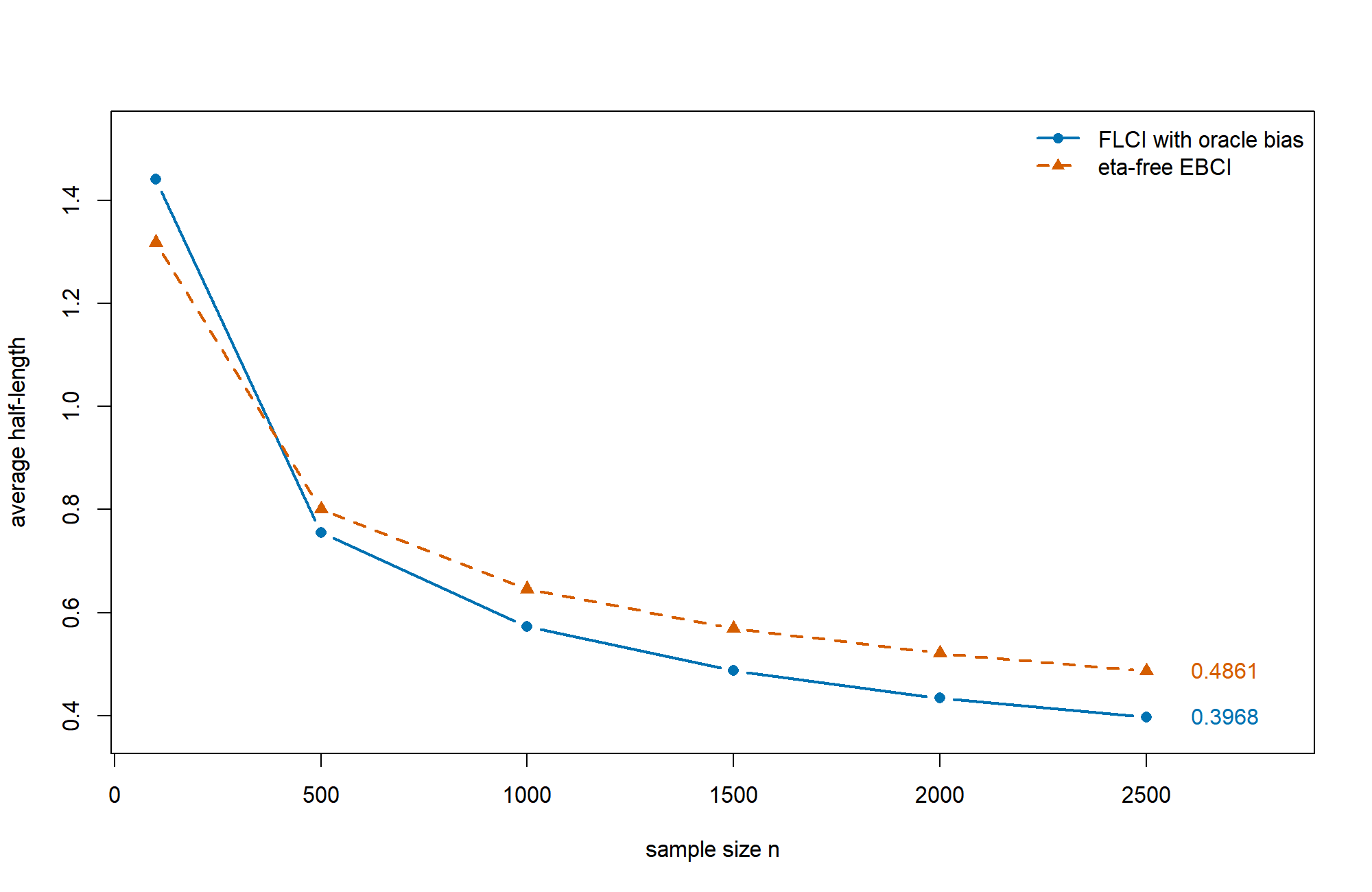}
 \caption{\it Average half-lengths of FLCI (solid, blue) with oracle bias and $\eta$-free EBCI (dashed, brown).}
    \label{fig:FLCI vs eta-free EBCI length}
\end{figure}

The comparison shown in Figures \ref{fig:FLCI vs eta-free EBCI cov} and \ref{fig:FLCI vs eta-free EBCI length} is deliberately asymmetric: the FLCI is supplied with the exact finite-sample absolute bias, whereas the $\eta$-free EBCI uses zero oracle bias information. Hence, any advantage of the $\eta$-free EBCI cannot be
attributed to more favorable bias knowledge. Figure \ref{fig:FLCI vs eta-free EBCI cov} shows that, even under the strongest skewness, as the sample size increases, the coverage of $\eta$-free EBCI
moves back toward the nominal level and eventually exceeds that of the oracle
FLCI. In these skewed designs, even with oracle bias information, the FLCI remains below nominal coverage over the medium-to-large sample range. Thus, oracle bias knowledge cannot repair inaccurate
Gaussian calibration, while $\eta$-free empirical-Bernstein calibration can
recover reliable coverage without requiring such oracle information. Compared with Figure \ref{fig:FLCI vs fixed-eta EBCI length}, Figure \ref{fig:FLCI vs eta-free EBCI length} indicates that $\eta$-free EBCI has more competitive efficiency and is comparable with the efficiency of FLCI equipped with oracle bias.

Taken together, the results highlight the complementary roles of the three
procedures. Oracle FLCIs remain a sharp bias-aware benchmark when Gaussian
calibration is accurate, but exact knowledge of the bias cannot by itself
protect against normal-approximation error. EBCIs provide sound finite-sample honesty with mild conservativeness. Thus, the choice is not between bias awareness and efficiency, but between different forms of protection: FLCIs prioritize sharpness under reliable Gaussian approximation, fixed-$\eta$ EBCIs prioritize finite-sample safety, and eta-free EBCIs prioritize practical efficient honesty without oracle bias knowledge.

\section{Conclusions}

This paper studies honest pointwise inference for nonparametric smoothers from the perspective of calibration. The objective is not only to obtain intervals with minimax shrinking length, but also to provide rate-explicit protection against undercoverage. The results deliver several main messages.

First, standard-normal calibration after bias correction has a residual-bias channel that can be consequential for undercoverage. Even when the leading bias is removed, the residual normalized bias can still be large enough to create undercoverage. This explains why a smaller regression noise is not always helpful for inference. A smaller noise reduces the stochastic scale and shortens the interval, but it can also make the remaining bias more important relative to that scale. 

Second, honest inference should control bias and stochastic error on compatible scales. The confidence intervals constructed by empirical Bernstein calibration proposed in this paper follow this idea. They control the stochastic part by Bernstein-type tail bounds and add a direct allowance for deterministic smoothing bias. The radius is chosen by the same bias-variance logic that motivates fixed-length confidence intervals, but the calibration does not rely on a normal critical value. More importantly, this change in calibration allows us to safely and conveniently convert the specified smoothness into both coverage accuracy and interval-length efficiency. 

Third, empirical Bernstein calibration should be viewed as a modular tool which can be easily connected to existing bias-control methods, not as a replacement for them. More specifically, when combined with RBC-type debiased estimator, RBC-type EBCIs preserve the bias-reduction effect of RBC, while avoiding the coverage-error-driven smoothness allocation and bandwidth restrictions imposed by standard-normal calibration. When combined with worst-case bias control, EBCI trades a mild constant-efficiency premium for explicit finite-sample
undercoverage protection and greater robustness to distributional-calibration error. This
shows that FLCI and EBCI are complementary tools within bias-aware inference: FLCI is
preferable when Gaussian/folded-normal calibration is reliable, whereas EBCI is
attractive when finite-sample coverage protection is paramount.

Summarizing, honest nonparametric inference should separate two tasks. One must control the deterministic smoothing bias, and one must calibrate the stochastic error in a way that protects coverage. Empirical Bernstein calibration provides such a route. It keeps the bias-aware logic of modern nonparametric inference, but replaces normal calibration with a finite-sample tail argument. This makes it useful when coverage is not just an asymptotic slogan, but part of the statistical guarantee.

\section*{Acknowledgment}
 Zihao Yuan gratefully acknowledges the financial support by German Research Foundation (DFG) {\it Debiased/ double machine learning under non-standard settings} (555312726).
 Holger Dette gratefully acknowledges the financial support by the German Research Foundation (DFG): TRR 391 {\it Spatio-temporal Statistics for the Transition of Energy and Transport} (520388526); Research unit 5381 \textit{Mathematical Statistics in the Information Age} (460867398).

\bibliographystyle{apacite}
\bibliography{ref.bib}

\newpage

\appendix
\setcounter{equation}{0}
\renewcommand\theequation{A.\arabic{equation}}
\renewcommand\thetheorem{A.\arabic{theorem}}
\renewcommand\thecorollary{A.\arabic{corollary}}
\renewcommand\thelemma{A.\arabic{lemma}}
\renewcommand\theremark{A.\arabic{remark}}
\setcounter{proposition}{0}
\renewcommand{\theproposition}{A.\arabic{proposition}}

\section{Proof of results in Section \ref{sec 2}}

\textbf{Proof of Proposition \ref{prop RE cov error}: }
Let $z=z_{1-\alpha/2}$ and $a_n=a_n(\theta)=\frac{R_0(h,\theta)}{\sigma_n(h)} .$
By assumption,
\[
    \frac{\widetilde\theta(x_0)-\theta(x_0)}{\sigma_n(h)}
    =
    Z_n+a_n,
    \qquad
    Z_n\sim N(0,1).
\]
Therefore,
\begin{align*}
&\mathbb P_\theta
\left(
\theta(x_0)\in
\left[
\widetilde\theta(x_0)\pm z\sigma_n(h)
\right]
\right)
=
\mathbb P_\theta
\left(
\left|
\widetilde\theta(x_0)-\theta(x_0)
\right|
\le z\sigma_n(h)
\right)\\
&=
\Phi(z-a_n)-\Phi(-z-a_n)
=
\Phi(z-a_n)+\Phi(z+a_n)-1 .
\end{align*}
Since \(1-\alpha=2\Phi(z)-1\), the coverage loss is
\begin{align*}
(1-\alpha)
-
\mathbb P_\theta
\left(
\theta(x_0)\in
\left[
\widetilde\theta(x_0)\pm z\sigma_n(h)
\right]
\right)
&=
2\Phi(z)-\Phi(z-a_n)-\Phi(z+a_n).
\end{align*}

Now apply a second-order Taylor expansion of \(\Phi\) around \(z\).  Since
\[
    \Phi'(z)=\phi(z),
    \qquad
    \Phi''(z)=-z\phi(z),
\]
we have, as \(a_n\to0\),
\[
    \Phi(z\pm a_n)
    =
    \Phi(z)\pm a_n\phi(z)
    -\frac12 a_n^2 z\phi(z)
    +o(a_n^2).
\]
Adding these expansions gives $\Phi(z-a_n)+\Phi(z+a_n)
    =
    2\Phi(z)
    -
    a_n^2 z\phi(z)
    +
    o(a_n^2)$, which asserts
\[
2\Phi(z)-\Phi(z-a_n)-\Phi(z+a_n)
=
z\phi(z)a_n^2+o(a_n^2).
\]
Since \(z\phi(z)>0\) for every fixed \(\alpha\in(0,1)\), there exists \(N_\alpha\) such that, for all \(n\ge N_\alpha\),
\[
2\Phi(z)-\Phi(z-a_n)-\Phi(z+a_n)
\ge
\frac14 z\phi(z)a_n^2 .
\]
Substituting back $a_n=\frac{R_0(h,\theta)}{\sigma_n(h)}$ yields
\[
(1-\alpha)
-
\mathbb P_\theta
\left(
\theta(x_0)\in
\left[
\widetilde\theta(x_0)\pm z_{1-\alpha/2}\sigma_n(h)
\right]
\right)
\ge
\frac14
z_{1-\alpha/2}
\phi(z_{1-\alpha/2})
\left(
\frac{R_0(h,\theta)}{\sigma_n(h)}
\right)^2 .
\]
This proves the proposition.

\section{Proof of Results in Sections \ref{sec 3} and \ref{sec density}}

\def\theequation{B.\arabic{equation}}
\setcounter{equation}{0}

\noindent\textbf{Proof of Theorems \ref{th refinement} and Corollary \ref{corollary 1}:} We only prove Theorem \ref{th refinement}, since Corollary \ref{corollary 1} is its natural consequence. Additionally, the proof is decomposed into two subsections. Moreover, for Theorem \ref{th refinement}, we only need to prove that there exist number sequence $e_n\downarrow 0$ and sufficiently large but fixed $N_0\geq 1$ such that
\begin{align*}
    \inf_{m\in \Theta_{0}(M_S,r_S,\psi_n)}\P_{U}\geq 1-\alpha-e_n,\ n\geq N_0,\\
\text{and}\ \qquad e_nn^{\frac{2S}{2S+1}}\to 0\qquad \text{as $n\to\infty$}.
\end{align*}

By assuming $ h=h^*=(\frac{2\log(\frac{1}{\alpha})C_V}{4S^2\eta^2n})^{\frac{1}{2S+1}}$, Chapter B1 gives an infeasible and fixed-length confidence interval serving as an oracle EBCI. But, for notation convenience, we still use notations $h$ and $h^*$ simultaneously. Chapter B2 aims to replace the infeasible parts with their relatively feasible versions and control the replacement error.\\

\noindent \textbf{Chapter B.1: Oracle EBCI}

\noindent\textbf{B.1 Step 1: } Since $\sum_{i=1}^nW_{ih}(0)=1$ holds for each $n\geq 1$, we have 
\begin{align*}
    &\hat m_{h}(0)-m(0)=\sum_{i=1}^nW_{ih}(0)V^{\frac{1}{2}}(X_i)\varepsilon_i+\sum_{i=1}^nW_{ih}(0)(m(X_i)-m(0))\\
    =&\sum_{i=1}^nW_{ih}(0)V^{\frac{1}{2}}(X_i)(\varepsilon_{iL_n}-\E[\varepsilon_{iL_n}|\mathcal{F}_n])+\sum_{i=1}^nW_{ih}(0)V^{\frac{1}{2}}(X_i)(\varepsilon_{iL^-_n}-\E[\varepsilon_{iL^-_n}|\mathcal{F}_n])+\sum_{i=1}^nW_{ih}(0)(m(X_i)-m(0))\\
    =&:T_1+T_2+B\leq T_1+|T_2|+|B|,
\end{align*}
where $\varepsilon_{iL_n}=\varepsilon_{i}1[|\varepsilon_i|\leq L_n]$, $\varepsilon_{iL^-_n}=\varepsilon_{i}1[|\varepsilon_i|> L_n]$ and $L_n=\frac{n^{\frac{S}{2S+1}}}{\log n}$. $\mathcal{F}_n$ is the sigma algebra generated by $\{X_{i}\}_{i=1}^n$. Meanwhile, we also denote the Peano remainder of the S-order Taylor expansion of $m(X_i)$ at point $0$ as $\rho_m(X_i,0)$. Then, together with the assumption that $m\in \Theta_{0}(M_S,r_S,\psi_n)$, using the reproduction property of local polynomial regression  yields
\begin{align*}
    |B|&=|\sum_{i=1}^nW_{ih}(0)\rho_m(X_i,0)|\leq \sum_{i=1}^n|W_{ih}(0)|\max_{i\in \{j\leq n:|X_j|\leq h\}}\sup_{m\in \Theta_{0}(M_S,r_S,h^*) }|\rho_m(X_i,0)|\\
    &\leq\sum_{i=1}^n|W_{ih}(0)|\max_{i\in \{j\leq n:|X_j|\leq h\}} M_Sr_S(|X_i|)|X_i|^S\leq \max_{1\leq i\leq n}|W_{ih}(0)| N_h M_Sr_S(h)h^S,
\end{align*}
where $N_h=\sum_{i=1}^n1[|X_i|\leq h]$. Define $C_\eta>0$ as a constant depending only on the aforementioned user-defined $\eta$ (More specific definition of $C_{\eta}$ will be specified in the later step.). Since $h\to 0$ and $r_S(h)\to 0$ hold, together with \eqref{eq proof of lemma 6 eq4} and Lemma \ref{lemma 4} for any fixed $(M_S,r_S)$, there exists some integer $N_{\eta1}\geq 1$ independent of $n$ and $h$ such that the following inequality holds for all $n\geq N_{\eta1}$,
\begin{align}
\label{eq proof of the refinement eq1}
   \P\Big( \sup_{m\in \Theta_{0}(M_S,r_S,\psi_n)}|B|\leq C_\eta h^S\Big)\geq 1-(S+1)\exp(-c_3nh)-\exp(-\frac{f^2_X(0)nh}{8},
\end{align}
where $c_3>0$ is some universal constant defined in Lemma \ref{lemma 4}.

\noindent\textbf{B.1 Step 2:} ($|T_2|$) Note that, for any given $t>0$, we have
\begin{align*}
    \P(|T_2|>t)&\leq 2t^{-1}\E\Big[1[\mathcal{E}_{\min}\cap \mathcal{E}_{\max}]\sum_{i=1}^n|W_{ih}(0)|\E[|\varepsilon_{iL^-_n}||\mathcal{F}_n]\Big]+\P(\mathcal{E}_{\max}^c)+\P(\mathcal{E}_{\min}^c)\\
    &\leq \frac{2E_{\varsigma}(\E[1[\mathcal{E}_{\min}\cap \mathcal{E}_{\max}]\sum_{i=1}^nW_{ih}^2(0)])^{\frac{1}{2}}}{L_n^{\varsigma-1}t}+2(S+1)\exp(-(c_1\lor c_2)nh)\\
    &\leq \frac{d_1}{\sqrt{nh}L_n^{\varsigma-1}t}+2(S+1)\exp(-(c_1\lor c_2)nh),
\end{align*}
where $d_1>0$ is the independent of $n,h$. The first inequality follows from the Markov inequality and a basic result in measure theory. The second inequality is based on Assumption \ref{as 3}, $L^p$-norm inequality in $\mathbb{R}^n$, and Jensen inequality for a concave function. By letting $t=C_\eta h^S$, together with $h=O(n^{-\frac{1}{2S+1}})$, the following inequality holds for all $n\geq 1$ and $h>0$,
\begin{align}
\label{eq proof of the refinement eq2}
    \P(|T_2|\leq C_\eta h^S)\geq 1-\frac{(d_{2}/C_\eta)(\log n)^{\varsigma-1}}{n^{\frac{(\varsigma-1)S}{2S+1}}}\geq 1-\frac{(d_{2}/C_\eta)(\log n)^{\varsigma-1}}{n^{\frac{3S+1}{2S+1}}}.
\end{align}
where $d_2>0$ is the independent of $n,h$.

\noindent\textbf{B.1 Step 3:} ($T_1$) Lemma \ref{lemma 2} implies that, on condition of $\mathcal{F}_n$, $\varepsilon_{iL_n}$'s are mutually independent. Hence, using Bernstein inequality with respect to conditional measure $\P(\cdot|\mathcal{F}_n)$ and iterative law of conditional expectation yield that
\begin{align*}
  \P(\mathcal{E}_{\leq }):=  \P\left(T_1\leq\sqrt{2\log(\frac{1}{\alpha})\sum_{i=1}^nW_{ih}^2(0)V(X_i)}+\frac{\log(\frac{1}{\alpha})}{3}\max_{i\leq n}|W_{ih}(0)|L_n\right)\geq1-\alpha
\end{align*}
holds for all $n\geq 1$ and $\alpha\in (0,1)$. Additionally, based on $l_S(u)$ introduced in Lemma \ref{lemma 4}, by letting
\begin{align*}
    \mathcal{E}_V:=\Big\{\Big|\sum_{i=1}^nW_{ih}^2(0)V(X_i)-\frac{V(0)}{nhf_X(0)}\int_{-1}^1l_S^2(u)du\Big|\leq \frac{(c_5+1)}{nh^{0.5}}\Big\},
\end{align*}
 the setting $h=h^*$ and Lemma \ref{lemma 6} implies that there exists some integer $N_*\geq 1$ independent of $n$ and $h$ such that
\begin{align*}
    \P( \mathcal{E}^c_V)\leq \exp(-\frac{f_X^2(0)nh}{8})+c_6\exp(-c_7nh^2),
\end{align*} 
where $c_6,c_7$ are positive constants introduced in Lemma \ref{lemma 6} (independent of $n$ and $h$). Then, together with $\mathcal{E}_{\max W}$ introduced in Lemma \ref{lemma 4}, we have 
\begin{align}
\label{eq proof of the refinement eq3}
  &1-\alpha\leq     \P(\mathcal{E}_\leq )\leq \P(\mathcal{E}_\leq \cap \mathcal{E}_{\max W}\cap \mathcal{E}_V)+\P(\mathcal{E}_{\max W}^c)+\P(\mathcal{E}^c_{V})\notag\\
  \leq &\P(\mathcal{E}_\leq \cap \mathcal{E}_{\max W}\cap \mathcal{E}_V)+(S+1)\exp(-c_3nh)+\exp(-\frac{f_X^2(0)nh}{8})+c_6\exp(-c_7nh^2).
\end{align}
Based on notation $C_V=\frac{V(0)}{f_X(0)}\int_{-1}^1l^2_S(u)du$ (see \eqref{eq th refinement eq2}), together with the fact that $\sqrt{x+y}\leq \sqrt{x}+\sqrt{y}$ holds for all $x,y\geq 0$ some simple algebra shows that 
\begin{align}
\label{eq proof of the refinement eq4}
    \mathcal{E}_\leq \cap \mathcal{E}_{\max W}\cap \mathcal{E}_V\subset \Big\{\sqrt{2\log(\frac{1}{\alpha})\frac{C_V}{nh}}+h^{0.25}\sqrt{\frac{2\log(\frac{1}{\alpha})(c_5+1)}{nh}}+\frac{4\log(\frac{1}{\alpha})\sqrt{S+1}L_n}{3nhf_X(0)\lambda_{\min}(\Gamma_1)}\Big\}.
\end{align}
Recall that $L_n=O(\frac{\sqrt{nh}}{\log n})$ and $h=h^*=O(n^{-\frac{1}{2S+1}})$. Based on \eqref{eq proof of the refinement eq3}, there exists a integer $N_{\eta2}\geq 1$ independent of $n$ and $h$ such that 
\begin{align}
\label{eq proof of the refinement eq5}
        \mathcal{E}_\leq \cap \mathcal{E}_{\max W}\cap \mathcal{E}_V\subset 
        \Big\{\sqrt{2\log(\frac{1}{\alpha})\frac{C_V}{nh}}+C_{\eta}h^S\Big\},\ \forall\ n\geq N_{\eta2}.
\end{align}
\noindent\textbf{B.1 Step 4:} Since $h=h^*$, combining \eqref{eq proof of the refinement eq1}, \eqref{eq proof of the refinement eq2}, \eqref{eq proof of the refinement eq3} and \eqref{eq proof of the refinement eq5} asserts that
\begin{align}
\label{eq oracle}
    \P(m(0)\leq& \hat m_{h^*}(0)+R_{\alpha}(3C_\eta,h^*))\geq 1-\alpha-o(n^{-\frac{2S}{2S+1}}),\ \forall\ n\geq N_{\eta1}\lor N_{\eta 2}\\
   & R_{\alpha}(x,h)=\sqrt{\frac{2\log(\frac{1}{\alpha})C_V}{nh}}+xh^S, x>0.\notag
\end{align}

\noindent \textbf{Chapter B.2: Feasible EBCI}

\noindent
Recall definitions 
\begin{align*}
    \hat{C}_V=ng\mathcal{V}_{ng},\ \ \hat h^*=\Big(\frac{2\log(\frac{1}{\alpha})\hat C_V}{4S^2\eta^2n}\Big)^{\frac{1}{2S+1}},\ \ \tilde h^*=\arg\min_{g\in \mathcal{H}_n}|g-\hat h^*|,
\end{align*}
where $\mathcal{H}_n$ is defined in \eqref{eq finite net}. Based on \eqref{eq oracle}, Chapter B2 aims to make it feasible by proving highly-probability gaps, $\hat C_V-C_V$, $\tilde h^*-h^*$ and $\hat m_{\tilde h^*}(0)-\hat m_{h^*}(0)$, which are crucial for obtaining the coverage error.\\
\noindent\textbf{B.2 Step 1:} ($|\hat C_V-C_V|$) Please notice the following basic decomposition 
\begin{align*}
    \hat C_V-C_V=&ng\Big(\sum_{i=1}^nW^2_{ig}(0)(Y_i-\hat m_{-i}(X_i))^2-\sum_{i=1}^nW^2_{ig}(0)V(X_i)\varepsilon_i^2\Big)\\
    &+ng\sum_{i=1}^nW^2_{ig}(0)V(X_i)(\varepsilon_i^2-1)+\Big(ng\sum_{i=1}^nW^2_{ig}(0)V(X_i)-C_V\Big)\\
    =&:\mathbb{V}_1+\mathbb{V}_2+\mathbb{V}_3.
\end{align*}
For $\mathbb{V}_1$, by denoting $V_i=V(X_i)$, $m_i=m(X_i)$ and $\hat m_i=\hat m_{-i}(X_i)$, some basic algebra shows 
\begin{align*}
    |\mathbb{V}_1|&\leq ng\sum_{i=1}^nW_{ig}^2(0)(m_i-\hat m_i)^2+2ng\Big|\sum_{i=1}^nW_{ig}^2(0)(m_i-\hat m_i)V_i\varepsilon_i\Big|\\
    &\leq ng (\max_{i\leq n}|\hat m_i-m_i|)^2\sum_{i=1}^nW^2_{ig}(0)+2ng\left(\sqrt{\sum_{i=1}^nW_{ig}^2(0)V_i(m_i-\hat m_i)^2}\cdot \sqrt{\sum_{i=1}^nW_{ig}^2(0)V_i\varepsilon_i^2}\right)\\
    &\leq ng (\max_{i\leq n}|\hat m_i-m_i|)^2\sum_{i=1}^nW^2_{ig}(0)\\
    &\ \ \ \ \ \ \ +2ng\bar V^{\frac{1}{2}} (\max_{i\leq n}|m_i-\hat m_i|)\sqrt{\sum_{i=1}^nW^2_{ig}(0)}\cdot\left(\sqrt{|\sum_{i=1}^nW^2_{ig}(0)V_i(\varepsilon_i^2-1)|}+\sqrt{\sum_{i=1}^nW^2_{ig}(0)V_i}\right),
\end{align*}
where the second inequality is based on Cauchy-Schwarz inequality. Since we have setting $g=O(n^{-\frac{1}{2S+1}})$, by applying Lemmas \ref{lemma 5}-\ref{lemma 8} and letting $\beta$ mentioned in Lemma \ref{lemma 7} as $\beta=4$, there exists some integer $N_{\eta3}\geq 1$ and constant $D_V>0$ independent of $n$ and $g$ such that 
\begin{align}
\label{eq proof of the refinement eq7}
    \P(|\hat C_V-C_V|\leq (\log n)^{-4})\geq 1-D_V(\log n)^{4}n^{-\frac{2S(\varsigma-2)}{2S+1}}
\end{align}

\noindent\textbf{B.2 Step 2:} ($|\hat h^*-h^*|$) The definitions of $h^*$ and $\hat h^*$ immediately implies
\begin{align}
    \label{eq proof of the refinement eq8}
    \frac{\hat h^*}{h^*}=\Big(\frac{\hat C_V}{C_V}\Big)^{\frac{1}{2S+1}}.
\end{align}
Meanwhile, \eqref{eq proof of the refinement eq7} implies that there exists some integer $N_{\eta4}\geq 1$ such that
\begin{align*}
    \P\Big(\frac{\hat C_V}{C_V}\in \Big[1\pm \frac{(\log n)^{-4}}{C_V}\Big]\subset [\frac{1}{2},\frac{3}{2}]\Big)\geq 1-D_V(\log n)^{4}n^{-\frac{2S(\varsigma-2)}{2S+1}}.
\end{align*}
Since the mean-value theorem yields
\begin{align*}
 \Big|\frac{\hat h^*}{h^*}-1\Big|=   \Big|\Big(\frac{\hat C_V}{C_V}\Big)^{\frac{1}{2S+1}}-1\Big|\leq C_S\Big|\Big(\frac{\hat C_V}{C_V}\Big)-1\Big|,
\end{align*}
where $C_S=\sup_{\frac{1}{2}\leq u\leq \frac{3}{2}}\frac{1}{2S+1}u^{-\frac{2S}{2S+1}}$, a natural corollary is 
\begin{align}
  \label{eq proof of the refinement eq9}
\P(\mathcal{I}):=  \P\left(\Big|\frac{\hat h^*}{h^*}-1\Big|\leq \frac{d^{**}}{(\log n)^4}\right)\geq 1-D_V(\log n)^{4}n^{-\frac{2S(\varsigma-2)}{2S+1}}
\end{align}
holds for all $n\geq N_{\eta3}\lor N_{\eta 4}$, where $d^{**}>0$ is a constant independent of $n$. According to the definition of $\tilde h^*$, when event $\mathcal{I}$ happens, basic triangle inequality implies that
\begin{align*}
    \Big|\frac{\tilde h^*}{h^*}-1\Big|\leq \Big|\frac{\hat h^*}{h^*}-1\Big|+\Big|\frac{\tilde h^*-\hat h^*}{h^*}\Big|\leq (d^{**}+1)(\log n)^{-4}
\end{align*}
holds for all $n\geq N_{\eta 5}$, where integer $N_{\eta 5}\geq 1$ is independent of $n$. Above all, we obtain
\begin{align}
    \label{eq proof of the refinement eq10}
  \P(H_n):=  \P\Big(  \Big|\frac{\tilde h^*}{h^*}-1\Big|\leq \frac{d^{**}+1}{(\log n)^4}\Big)\geq  1-D_V(\log n)^{4}n^{-\frac{2S(\varsigma-2)}{2S+1}}
\end{align}
for all $n\geq N_{\eta3}\lor N_{\eta 4}\lor N_{\eta 5}.$

\noindent\textbf{B.2 Step 3:} A key trick used to bound the gap $|\hat m_{\tilde h^*}(0)-\hat m_{h^*}(0)|$ is to notice that, according to the definition of finite net $\mathcal{H}_n$ (see \eqref{eq finite net}), there exists sufficiently large $N_{\eta6}\geq 1$ independent of $n$ such that, when event $H_n$ happens, the following basic inequality holds for all $n\geq N_{\eta6}$,
\begin{align*}
       &|\hat m_{\tilde h^*}(0)-\hat m_{h^*}(0) |\notag\\
       \leq &|\sum_{i=1}^n(W_{i\tilde h^*}(0)-W_{i h^*}(0))m(X_i)|+ |\sum_{i=1}^n(W_{i\tilde h^*}(0)-W_{i h^*}(0))V^{\frac{1}{2}}(X_i)\varepsilon_i|\notag\\
       \leq &\max_{h\in \mathcal H^*_n}|\sum_{i=1}^n(W_{ih}(0)-W_{i h^*}(0))m(X_i)|+\max_{h\in \mathcal{H}^*_n}|\sum_{i=1}^n(W_{ih}(0)-W_{i h^*}(0))V^{\frac{1}{2}}(X_i)\varepsilon_i|\\
       =&: \max_{h\in \mathcal{H}^*_n}|B_h|+\max_{h\in \mathcal{H}^*_n}|Z_h|,\\
  \text{where} & \ \ \ \ \ \ \ \    \mathcal{H}^*_n=\{h\in\mathcal{H}_n:|\frac{h}{h^*}-1|\leq (d^{**}+2)(\log n)^{-4}\}.
\end{align*}
For $\max_{h\in \mathcal{H}^*_n}|B_h|$, some simple algebra shows 
 \begin{align*}
     &\max_{h\in \mathcal{H}^*_n} |B_h|=\max_{h\in \mathcal{H}^*_n}|\sum_{i=1}^n(W_{ih}(0)-W_{ih^*}(0))\rho_m(X_i,0)|\\
   \leq &\max_{h\in \mathcal{H}^*_n}\Big(\max_{1\leq i\leq n}|W_{ih}(0)|N_h\Big)\max_{h\in \mathcal{H}^*_n}\Big(M_Sr_S(h)h^S\Big)+\max_{1\leq i\leq n}|W_{ih^*}(0)|N_{h^*}M_{S}r_S(h^*)(h^*)^S,
 \end{align*}
 where, for some deterministic bandwidth $h$, $N_h=\sum_{i=1}^n1[|X_i|\leq h]$. The definition of $\mathcal{H}_n^*$ implies that there exists some integer $M_{\eta 7}\geq 1$ independent of $n$ such that 
 \[
\max_{h\in \mathcal{H}^*_n} |B_h|\leq \sqrt{C_{\eta}}\max_{h\in \mathcal{H}^*_n}\Big(\max_{1\leq i\leq n}|W_{ih}(0)|N_h\Big)M_{S}r_S(h^*)(h^*)^S.
\]
Moreover, \eqref{eq proof of lemma 6 eq4} and Lemma \ref{lemma 4} implies that there exists some $N_{\eta 7}\geq M_{\eta 7}$ independent of $n$ such that the following inequality holds for all $m\in \Theta_0(M_S,r_S,h^*)$ and $ n\geq \max_{3\leq k\le 7}\{N_{\eta k}\}$,
\begin{align}
    \label{eq proof of the refinement eq11}
  \P(\mathcal{B}):=  \P\Big(\max_{h\in \mathcal{H}^*_n}|B_h|\leq C_{\eta}(h^*)^S\Big)\geq 1-D_V(\log n)^{4}n^{-\frac{2S(\varsigma-2)}{2S+1}}.
\end{align}
To bound $\max_{h\in \mathcal{H}^*_n}|Z_h|$, we first use decomposition 
\begin{align*}
    \max_{h\in \mathcal{H}^*_n}|Z_h|\leq &\max_{h\in\mathcal{H}_n}|\sum_{i=1}^n(W_{ih}(0)-W_{ih^*}(0))V^{\frac{1}{2}}(X_i)(\varepsilon_{iL_{n}}-\E[\varepsilon_{iL_{n}}|\mathcal{F}_n])|\\
    &+\max_{h\in\mathcal{H}^*_n}|\sum_{i=1}^n(W_{ih}(0)-W_{ih^*}(0))V^{\frac{1}{2}}(X_i)\varepsilon_{iL^-_{n}}|+\max_{h\in\mathcal{H}^*_n}\sum_{i=1}^n|W_{ih}(0)-W_{ih^*}(0)|V^{\frac{1}{2}}(X_i)\E[|\varepsilon_{iL^-_{n}}||\mathcal{F}_n]\\
    =&:\Xi_1+\Xi_2+\Xi_3,
\end{align*}
where $L_n=n^{\frac{S}{2S+1}}/(\log n)^3$.

\noindent\textbf{B.2 Step 4:} ($\max_{h\in \mathcal{H}_n^*} \sum_{i=1}^n (W_{ih}(0)-W(0))^2$)
Before we start to bound $\Xi_k$, $k=1,2,3$, we first notice the following two simple but important mathematical statements.
\begin{itemize}
    \item [S1] Based on the $l_S(u)=K(u)\e_0^{\top}\Gamma_1^{-1}r(u)$ introduced in Lemma \ref{lemma 4} and the assumption that $K$ has continuous derivative, by denoting $\Psi_h(x)=\frac{1}{hf_X(0)}l_S(\frac{x}{h})$, we have
    \begin{align*}
        &\partial_h\Psi_h(x)=-\frac{1}{h^2f_X(0)}(l_S(\frac{x}{h})+\frac{x}{h}l_S^{(1)}(\frac{x}{h}))=:-\frac{1}{h^2f_X(0)}q(\frac{x}{h}),\\
        l_S^{(1)}(u)=&K^{(1)}(u)\e_0^\top \Gamma^{-1}_1r(u)+K(u)\e_0^\top \Gamma^{-1}_1r^{(1)}(u),\ r^{(1)}(u)=(0,1,2u,...,Su^{S-1})\in \mathbb{R}^{S+1}.
    \end{align*}
    Moreover, $\sup_{u\in [-1,1]}|q(u)|=:C_q<\infty$ holds. 
    \item [S2]When event $H_n$ introduced in \eqref{eq proof of the refinement eq10} happens, by letting $\overline{W}_{ih}(0)=\frac{1}{n}\Psi_h(x)$, for each $i$, there exists $\xi_i\in (h^*,h)$ (or $(h,h^*)$) such that
    \begin{align*}
        |\overline{W}_{ih}(0)-\overline{W}_{ih^*}(0)|\leq C_q\Big|\frac{h-h^*}{nf_X(0)\xi_i^2}\Big|1[|X_i|\leq \xi_i]\leq K_0\Big|\frac{h-h^*}{nf_X(0)(h^*)^2}\Big|1[|X_i|\leq K_1h^*],
    \end{align*}
    where $K_0, K_1>0$ is a fixed constant independent of $n$. 
\end{itemize}
S1 is obtained from a simple calculation of the derivative of $l_S$, and S2 can be obtained using S1, the mean-value theorem, and the fact that $\xi_i=O(h^*)$. Thus, we omit the details here. Then, using S2 immediately implies
\begin{align}
     \label{eq proof of the refinement eq12}
    \max_{h\in \mathcal{H}_n^*} \sum_{i=1}^n (\overline{W}_{ih}(0)-\overline{W}_{ih^*}(0))^2\leq  \Big(\frac{K_0}{(nh^*)f^2_X(0)}\Big)^2\max_{h\in \mathcal{H}_n^*}\Big|\frac{h-h^*}{h^*}\Big|^2N_{K_1h^*},
\end{align}
where $N_{K_1h^*}=\sum_{i=1}^n 1[|X_i|\leq K_1h^*]$. Similar to Step 2 in the proof of Lemma \ref{lemma 6}, by introducing $\mathcal{G}_{h^*}=\{N_{K_1h^*}\leq 6K_1f_X(0)nh^*\}$, definition of $h^*$ implies that there exists integer $N_{\eta8}\geq 1$ independent of $n$ such that
\begin{align*}
    \P(\mathcal{G}_{h^*})\geq 1-\exp(-\frac{K_1f^2_X(0)nh^*}{8}),\ \ \forall\ n\geq N_{\eta 8}.
\end{align*}
 Together with \eqref{eq proof of the refinement eq12}, we obtain 
 \begin{align}
     \label{eq proof of the refinement eq13}
     \P&\left( \max_{h\in \mathcal{H}_n^*} \sum_{i=1}^n (\overline{W}_{ih}(0)-\overline{W}_{ih^*}(0))^2\leq\frac{6K_0^2d^{**}/f_X(0)}{nh^*(\log n)^4} \right)\notag\\
    & \geq 1-\exp(-\frac{K_1f^2_X(0)nh^*}{8})-D_V(\log n)^{4}n^{-\frac{2S(\varsigma-2)}{2S+1}},
 \end{align}
where, according to Assumption \ref{as 3}, $\varsigma>4+\frac{1}{S}$.

Now we focus on bounding $\max_{h\in \mathcal{H}_n^*}\sum_{i=1}^n(W_{ih}(0)-W_{ih^*}(0))^2$. Provided that event $H_n\cap \mathcal{G}_{h^*}$ happens, by letting $\triangle_{ih}(0)=W_{ih}(0)-\overline{W}_{ih}(0)$,  a basic upper bound is
\begin{align}
\label{eq proof of the refinement eq14}
    &\max_{h\in \mathcal{H}_n^*}\sum_{i=1}^n(W_{ih}(0)-W_{ih^*}(0))^2\notag\\
    \leq& 2\max_{h\in \mathcal{H}_n^*} \sum_{i=1}^n (\overline{W}_{ih}(0)-\overline{W}_{ih^*}(0))^2+2\max_{h\in \mathcal{H}_n^*} \sum_{i=1}^n (\triangle_{ih}(0)-\triangle_{ih^*}(0))^2\notag\\
    \leq& \frac{12C_0^2d^{**}/f_X(0)}{nh^*(\log n)^4}+ n(\max_{h\in \mathcal{H}_n^*}(\max_{i}|\triangle_{ih}(0)|)^2+(\max_{i}|\triangle_{ih^*}(0)|)^2).
\end{align}
Recall that $\mathcal{H}_n^*\subset \mathcal{H}_n$ and, according to \eqref{eq finite net}, $\textbf{Card}(\mathcal{H}_n)=(\log n)^4$. Together with the fact $h=O(h^*)=O(n^{-\frac{1}{2S+1}})$, applying Lemma \ref{lemma 4} by setting the $\epsilon$ there as $(\log n)^2h^*$ asserts that
\begin{align}
     \label{eq proof of the refinement eq15}
     \P(n(\max_{h\in \mathcal{H}_n^*}(\max_{i}|\triangle_{ih}(0)|)^2+(\max_{i}|\triangle_{ih^*}(0)|)^2)\leq \frac{d_1(\log n)^4}{n})\geq 1-d_2(\log n)^4e^{-(\log n)^2}
\end{align}
holds for all $n\geq N_{\eta9}$, where $d_1,d_2>0$ and $N_{\eta9}\geq 1$ are independent of $n$.

Above all, combining \eqref{eq proof of the refinement eq13}-\eqref{eq proof of the refinement eq15} asserts
\begin{align}
    \label{eq proof of the refinement eq16}
    \P(\mathcal{E}_{DW}):=\P\left( \max_{h\in \mathcal{H}_n^*} \sum_{i=1}^n (W_{ih}(0)-W_{ih^*}(0))^2\leq\frac{D^{*}}{nh^*(\log n)^4} \right)\geq 1-D^{**}(\log n)^{4}n^{-\frac{2S(\varsigma-2)}{2S+1}},
\end{align}
holds for all $ n\geq N_{\eta 10}$, where $D^*,D^{**}>0$, $N_{\eta10}\geq 1$ are independent of $n$ and $\varsigma>4+\frac{1}{S}$.

\noindent\textbf{B.2 Step 5:} ($\Xi_2$ and $\Xi_3$) To bound $\Xi_3$, based on event $\mathcal{E}_{DW}$ introduced in \eqref{eq proof of the refinement eq16}, combining Assumption \ref{as 3} with Markov, Cauchy-Schwartz and Jensen inequalities yields 
\begin{align*}
\P(\Xi_3>t)&\leq \textbf{Card}(\mathcal{H}_n)\P\Big(\sum_{i=1}^n |W_{ih}(0)-W_{ih^*}(0)|\E[|\varepsilon_{iL_n^-}|]>t\Big)\\
&\leq \frac{(\log n)^5E_{\varsigma}\bar V}{tL_n^{\varsigma-1}}\sqrt{\E[\sum_{i=1}^n (W_{ih}(0)-W_{ih^*}(0))^21[\mathcal{E}_{DW}]]}+(\log n)^4\P(\mathcal{E}_{DW}^c)\\
&\leq \frac{(\log n)^{6.5}E_{\varsigma}\bar VD^*}{tn^{\frac{S(\varsigma-1)}{2S+1}}\sqrt{nh^*}}+D^{**}(\log n)^{8}n^{-\frac{2S(\varsigma-2)}{2S+1}},\ \forall n\geq N_{\eta10}.
\end{align*}
By defining $t=n^{-(\frac{S+1}{2S+1})}$, together with some simple algebra and $h^*=O(n^{-\frac{1}{2S+1}})$, there exists some constant $D_*>0$ and integer $N_{\eta 11}\geq 1$ independent of $n$ such that 
\begin{align}
    \label{eq proof of the refinement eq17}
    \P(\Xi_3\leq n^{-(\frac{S+1}{2S+1})})\geq 1-\frac{D_*}{n^{\frac{3S+1}{2S+1}}}-D^{**}\frac{(\log n)^{7}}{n^2},\ \forall\  n\geq N_{\eta 10}\lor N_{\eta 11}. 
\end{align}
The bounding of $\Xi_2$ is direct since we only need to notice that, for any given $t>0$, 
\begin{align*}
    \{\Xi_2>t\}\subset \bigcup_{i=1}^n\{|\varepsilon_i|\geq L_n\},
\end{align*}
which asserts
\begin{align}
  \label{eq proof of the refinement eq18}
    \P(|\Xi_2|\leq \frac{1}{nh^*})\geq 1- \sum_{i=1}^n\P(\{|\varepsilon_i|>L_n\})\geq 1-\frac{(\log n)^{3\varsigma}E_\varsigma}{n^{1-\frac{\varsigma S}{2S+1}}},\ \forall\ n\geq 1.
\end{align}

\noindent\textbf{B.2 Step 6:} ($\Xi_1$) Based on events $\mathcal{E}_{DW}$ and $\mathcal{E}_{\max W}$ introduced in \eqref{eq proof of the refinement eq16} and Lemma \ref{lemma 4}, combining Lemma \ref{lemma 2}, union bound argument immediately and Bernstein inequality yields
\begin{align*}
    &\P(\Xi_1>t)\\
    \leq& 2(\log n)^5 \max_{h\in \mathcal{H}_n^*} \E\left[1[\mathcal{E}_{DW}\cap \mathcal{E}_{\max W}]\exp\left(-\frac{t^2}{\bar V\sum_{i=1}^n(W_{ih}(0)-W_{ih^*}(0))^2+\frac{tn^{\frac{S}{2S+1}}}{3(\log n)^3}\max_{i\leq n}|W_{ih}(0)-W_{ih^*}(0)|}\right)\right]\\
    &\ \ \ \ \ \ \  \ \ \ \ \ \ + (\log n)^5 (\P(\mathcal{E}^c_{DW})+\P(\mathcal{E}^c_{\max W}))\\
    \leq &2(\log n)^5\exp\left(-C_*\frac{t^2}{n^{-\frac{2S}{2S+1}}(\log n)^{-4}+\frac{t}{3}n^{-\frac{S}{2S+1}}(\log n)^{-3}}\right)+ \frac{(\log n)^8D^{**}}{n^2}+(S+1)e^{-c_3nh},
\end{align*}
where $C_*>0$ is independent of $n$. $D^{**}$ and $c_3$ are introduced in \eqref{eq proof of the refinement eq16} and Lemma \ref{lemma 4} respectively. Therefore, by setting $t=n^{-\frac{S}{2S+1}}(\log n)^{-\frac{1}{2}}$, we immediately obtain
\begin{align}
    \label{eq proof of the refinement eq19}
    \P(\Xi_1\leq n^{-\frac{S}{2S+1}}(\log n)^{-\frac{1}{2}})\geq 1-\frac{C_{**}}{n^2},
\end{align}
where $C_{**}>0$ is independent of $n$. Then, together with the fact that $h^*=O(n^{-\frac{1}{2S+1}})$, combining \eqref{eq proof of the refinement eq17}-\eqref{eq proof of the refinement eq19} implies
\begin{align}
    \label{eq proof of the refinement eq20}
  \P(\mathcal{Z}):=  \P(\max_{h\in \mathcal{H}_n^*}|Z_h|\leq C_{\eta}(h^*)^S)\geq 1-o(n^{-\frac{2S}{2S+1}}),\ \forall\ n\geq \max_{k\leq 11}\{N_{\eta_{k}}\}.
\end{align}
\noindent\textbf{B.2 Step 7:} (Optimization) Note that, when $\mathcal{Z}\cap \mathcal{B}$ happens, result shown in \eqref{eq oracle} immediately implies that, for all $n\geq \max_{k\leq 11}\{N_{\eta_{k}}\} $, the following inequality holds
\begin{align*}
    m(0)\leq &\hat m_{\tilde h^*}(0)+\max_{h\in \mathcal{H_n^*}}|B_h|+\max_{h\in \mathcal{H_n^*}}|Z_h|+R_{\alpha}(3C_{\eta},h^*)\\
    \leq & \hat m_{\tilde h^*}(0)+ 2C_{\eta}(h^*)^S+ R_{\alpha}(3C_{\eta},h^*)\\
    =&\hat m_{\tilde h^*}(0)+\sqrt{\frac{2\log(\frac{1}{\alpha})C_V}{nh^*}}+5C_{\eta}(h^*)^S\\
    =:&\hat m_{\tilde h^*}(0)+ R_{\alpha}(\eta,h^*),
\end{align*}
where the last nomination is because $C_{\eta}$ is user-defined, whose anonymous property allows us to denote $\eta=5C_{\eta}$. Obviously, for each fixed $\eta>0$, some basic algebra shows that
\begin{align*}
R_{\alpha}(\eta,h^*)=\min_{h>0}R_{\alpha}(\eta,h)
\end{align*}
holds for all $n\geq1$ and $\alpha\in (0,1)$, which yields the following \textbf{oracle one-sided EBCI} 
\begin{align}
    \label{eq oracle EBCI}
    \P\Big(m(0)\leq \hat m_{\tilde h^*}(0)+(2S+1)\eta^{\frac{1}{2S+1}}\Big(\frac{2\log(\frac{1}{\alpha})C_V}{4S^2n}\Big)^{\frac{S}{2S+1}}\Big)\geq 1-\alpha- o(n^{-\frac{2S}{2S+1}}).
\end{align}
\noindent\textbf{B.2 Step 8:} (Feasible EBCI) Since \eqref{eq proof of the refinement eq7} and Assumption \ref{as 3} imply that 
\begin{align*}
    \P(C_V\leq \hat C_V+(\log n)^{-4})\geq 1-o(n^{-\frac{2S}{2S+1}}),
\end{align*}
together with \eqref{eq oracle EBCI}, we know there exists $N_1\geq 1$ independent of $n$ such that 
\begin{align}
    \label{eq feasible EBCI}
\P\Big(m(0)\leq &\hat m_{\tilde h^*}(0)+(2S+1)(1+\xi_n)\eta^{\frac{1}{2S+1}}\Big(\frac{2\log(\frac{1}{\alpha})\hat C_V}{4S^2n}\Big)^{\frac{S}{2S+1}}\Big)\geq 1-\alpha- o(n^{-\frac{2S}{2S+1}})
\end{align}
holds for all $n\geq N_1$, where $\xi_n=(\log n)^{-3}$. Therefore, we finish the proof.\\

\noindent\textbf{Proof of Theorem \ref{th eta free regression}:} The proof is very direct and we thus only sketch the key steps. According to \eqref{eq oracle} exhibited in the proof of Theorem \ref{th refinement} Chapter B.1 Step 4, by letting $\eta=3C_{\eta}$, for each fixed but user-defined $\eta>0$, the following infeasible EBCI holds for all
\begin{align*}
   \P\Big( m(0)\leq \hat m_{h^*(\eta)}(0)+(2S+1)\eta^{\frac{1}{2S+1}}\Big(\frac{2\log(1/\alpha)C_V}{4S^2n}\Big)^{\frac{S}{2S+1}} \Big)\geq 1-\alpha-e_n,\ \forall\ n\geq m_{1}
\end{align*}
where $h^*(\eta)=(\frac{2\log(\frac{1}{\alpha})C_V}{4S^2\eta^2n})^{\frac{1}{2S+1}}$. $e_n=o(n^{-\frac{2S}{2S+1}})$ and $m_{1}\geq 1$ is an integer independent of $n$. Meanwhile, a noteworthy point is that, for each given $S\geq 1$, there exists an $\eta_S$ such that 
\[
 \P\Big( m(0)\leq \hat m_{h^*(\eta_S)}(0)+\Big(\frac{2\log(1/\alpha)C_V}{n}\Big)^{\frac{S}{2S+1}} \Big)\geq 1-\alpha-e_n,\ \forall\ n\geq m_{\eta1}.
\]
Suppose $d_n\geq 3$ is a diverging number sequence such that $d_n/\log n\to\infty$. By letting $\triangle_0=|\hat m_{h^*(\eta_S)}(0)-\hat m_{h_{nv}}(0)|$, where $h_{nv}=n^{-\frac{1}{2S+1}}$, we have 
\[
\P\Big(\triangle_0\leq \frac{d_n}{3}\Big(\frac{2\log(1/\alpha)C_V}{n}\Big)^{\frac{S}{2S+1}}\Big)\geq 1-e_n',\ \forall\ n\geq m_{2},
\]
where $e_n'>0$ satisfying $e_{n}'=o(n^{-\frac{2S}{2S+1}})$ and $m_2\geq 1$ is independent of $n$. This further asserts that, for all $\forall\ n\geq m_{1}\lor m_2$,
\begin{align}
    \label{eq proof of th eta free regression eq1}
    \ \P\Big( m(0)\leq \hat m_{h_{nv}}(0)+\frac{2d_n}{3}\Big(\frac{2\log(\frac{1}{\alpha})(C_V/\hat C_{V,0})\hat C_{V,0}}{n}\Big)^{\frac{S}{2S+1}} \Big)\geq 1-\alpha-(e_n+e_{n'}).
\end{align}
Recall that the function $x\to x^{\frac{S}{2S+1}}$ has a bounded derivative on some fixed neighborhood of point $1$. Based on Lemma \ref{lemma 10}, by choosing some fixed $\beta>1$, we have
\begin{align*}
   \P(E_C):=    \P(|C_V/\hat C_{V,0}-1|\leq 2^{-\frac{2S+1}{S}})\geq 1-e_n'',\ \forall\ n\geq m_3,
\end{align*}
where $e_n''>0$ satisfies $e_n''=o(n^{-\frac{2S}{2S+1}})$ and $m_3\geq 1$ is a fixed integer independent of $n$. Furthermore, when event $E_C$ happens, we have
\begin{align}
           \label{eq proof of th eta free regression eq2}
           \frac{2d_n}{3}\Big(\frac{2\log(\frac{1}{\alpha})\hat C_{V,0}}{n}\Big)^{\frac{S}{2S+1}}\Big(\frac{C_V}{\hat C_{V,0}}-1\Big)^{\frac{S}{2S+1}}\leq \frac{d_n}{3}\Big(\frac{2\log(\frac{1}{\alpha})\hat C_{V,0}}{n}\Big)^{\frac{S}{2S+1}}.
\end{align}
Then, combining \eqref{eq proof of th eta free regression eq1} and \eqref{eq proof of th eta free regression eq2} finishes the proof of Theorem \ref{th eta free regression}.\\ 

\noindent\textbf{Proof of Theorem \ref{th linear sieve}} By letting $\lambda=\frac{1}{K}\asymp n^{-\frac{1}{2S+1}}$, we make use of the template result, Theorem \ref{th WAE} in Appendix \ref{app D}, to finish the proof.

\noindent \textbf{Step 1} (Verification of Assumption \ref{as 4})  Condition $\P(|\hat V-V|\leq  n^{-\gamma})\geq 1-o(n^{-\frac{2S}{2S+1}})$ in Theorem \ref{th linear sieve} and Lemma \ref{lemma 11} have already confirmed the first two inequalities of Assumption \ref{as 4} hold with $e_{v,n}$ and $e_{w,n}$ decaying at rate $o(n^{-\frac{2S}{2S+1}})$. Now we focus on obtaining the existence of the random variable $C_W$ mentioned in Assumption \ref{as 4}. For any fixed constant $c>0$, define  
\begin{align*}
 &C_W= \frac{n}{K}\sum_{i=1}^n\widetilde W_{iK}^2(0)+c=\frac{1}{K}b_K^{\top}(0)\widetilde Q_K^{-1} b_K(0)+c  ,\\
& \widetilde W_{iK}(0)= \frac{1}{n}b^\top_K(0) \widetilde Q^{-1}_{K} b_K(X_i'),\ \widetilde Q_{K}=\frac{1}{n}\sum_{i=1}^nb_{K}(X'_i)b^\top_K(X'_i), 
\end{align*}
 where $\{X_i'\}_{i=1}^n$ is a set of random variables independent of $\{(X_i,Y_{i})\}_{i=1}^n$ and $(X'_1,...,X'_n)$ is identically distributed as $(X_1,...,X_n)$. Thus, $C_W$ is a measurable mapping with respect to $\sigma(X'_i:i\leq n)$ but independent of $\sigma((X_i,Y_i):i\leq n)$.
Since $c>0$ is a fixed constant, to specify the third inequality in Assumption \ref{as 4}, it suffices to investigate the high-probability bound of 
\begin{align*}
  \frac{n}{K} |\sum_{i=1}^n\widetilde W_{iK}^2(0)-\sum_{i=1}^nW_{iK}^2(0)|=\frac{1}{K}\Big|b_K^\top(0)(\widetilde Q_K^{-1}-\hat Q_K^{-1})b_K(0)\Big|.
\end{align*}
Moreover, some simple linear algebra yields 
\begin{align*}
    &\frac{1}{K}\Big|b_K^\top(0)(\widetilde Q_K^{-1}-\hat Q_K^{-1})b_K(0)\Big|\\
    \leq &\frac{||b_K(0)||^2}{K}||\widetilde Q^{-1}_K||\cdot ||\hat Q_K^{-1}||\cdot (||\widetilde Q_{K}-Q_K||+||\hat Q_K-Q_K||)\\
    \leq & C^2_{\infty}|| \widetilde Q^{-1}_K||\cdot ||\hat Q_K^{-1}||\cdot (||\widetilde Q_{K}-Q_K||+||\hat Q_K-Q_K||),
\end{align*}
where the last inequality is according to Assumption \ref{as 8}. Using Theorem 1.6 of \citeA{tropp2015introduction} asserts
\begin{align*}
    \P\Big(||\widetilde Q_{K}-Q_K||+||\hat Q_K-Q_K||\leq 2(\log K)\sqrt{\frac{K}{n}}\Big)\geq 1-4K\exp\Big(-c(\log K)^2\Big),\ \forall\ n\geq 1,
\end{align*}
 where $c>0$ is a constant independent of $n$ and $K$. Together with 
\eqref{eq proof of lemma 11 eq6}, we obtain that 
\begin{align*}
 \P\Big(
    \frac{1}{K}\Big|b_K^\top(0)(\widetilde Q_K^{-1}-\hat Q_K^{-1})b_K(0)\Big|\leq  2C^2_{\infty}C_I^2(\log K)\sqrt{\frac{K}{n}} \Big)\geq 1-O(K\exp(-(\log K)^2)),
\end{align*}
where $K\asymp n^{\frac{1}{2S+1}}$. Then, it is easy to obtain that, for any given $\eta_n\nearrow\infty$, 
\begin{align}
    \label{eq proof of th linear sieve eq1}
    \P\Big(\frac{n}{K}\sum_{i=1}^nW_{iK}^2(0)\leq C_W\leq \eta_n \frac{n}{K}\sum_{i=1}^nW_{iK}^2(0)\Big)\geq 1-O(K\exp(-(\log K)^2)).
\end{align}

\noindent\textbf{Step 2} (verification of Assumption \ref{as 5}) For any given $m\in \mathcal{F}_S$, by letting $m_K$ as the solution of $\arg\min_{g\in \mathcal{B}_K}||g-m||$, identity property of orthogonal projection and some simple algebra yield 
\begin{align*}
    &\Big|\sum_{i=1}^nW_{iK}(0)m(X_i)-m(0)\Big|\leq \Big|\sum_{i=1}^nW_{iK}(0)(m(X_i)-m_K(X_i))\Big|+\Big|m_K(0)-m(0)\Big|\\
    &\leq \Big(1+\Big|\sum_{i=1}^nW_{iK}(0)\Big|\Big)\cdot||m_K-m||_{\infty}\leq \Big(1+\Big|\sum_{i=1}^nW_{iK}(0)\Big|\Big)\sup_{f\in \mathcal{F}_S}\inf_{g\in\mathcal{B}_K}||g-f||,
\end{align*}
where the upper bound is independent of  $m$. Combining Assumption \ref{as 11} and Lemma \ref{lemma 11} asserts that 
\begin{align}
        \label{eq proof of th linear sieve eq2}
        \P\Big(\Big|\sum_{i=1}^nW_{iK}(0)m(X_i)-m(0)\Big|\leq (1+B\max)C_SK^{-S}\Big)\geq 1-O(\exp(-n/K)).
\end{align}
We therefore verify Assumption \ref{as 5}.

\noindent\textbf{Step 3} (verification of Assumption \ref{as 6}) For any deterministic $K\neq K'$ such that $K\asymp K'\asymp n^{\frac{1}{2S+1}}$, a direct approach is to given high probability bounds of $|\hat m_K(0)-m(0)|$ and $|\hat m_{K'}(0)-m(0)|$. It suffices to consider only one of them. The proof is thus quite direct. Note that $D_n=\xi_n\sqrt{\log n}$, where $\xi_n>0$ is arbitrary given diverging sequence. Thus, combining Assumption \ref{as 3}, truncation method, Bernstein inequality and \eqref{eq proof of th linear sieve eq2} immediately obtain that 
\begin{align}
    \label{eq proof of th linear sieve eq3}
    \P\Big(|\hat m_{K}(0)-m(0)|\leq D_nn^{-\frac{S}{2S+1}}\Big)\geq 1- o(n^{-\frac{2S}{2S+1}}).
\end{align}
We thus finish the verification of Assumption \ref{as 6}.

\noindent \textbf{Step 4} Note that $d_n$ mentioned in \eqref{eq proof of th linear sieve eq1} is allowed to be arbitrary diverging number sequence. We thus set $\eta_n=D_n$. Then, together with  the results obtained in the three steps above, by letting $d_n$ using Theorem \ref{th WAE} asserts
\begin{align}
    \inf_{m\in\mathcal F_S}\P\Big(m(0)\pm \Big[\hat m_K(0)\pm \underbrace{\xi_n\sqrt{\log n}}_{=:d_n}\Big(\frac{2\log(\frac{2}{\alpha}\hat V\hat C_W)}{n}\Big)\Big]\Big)\geq 1-\alpha-o(n^{-\frac{2S}{2S+1}}),
\end{align}
where $d_n$ satisfies the requirement mentioned in Theorem \ref{th linear sieve}.

\medskip

\noindent\textbf{Proof of Theorem \ref{th density}} 
The proof of Theorem \ref{th density} is similar to that of Theorem \ref{th refinement}. We thus only sketch it here. For a deterministic bandwidth $h$, write
\[
  \widehat f_h(0)
  = \sum_{i=1}^n w_{i,h}(0)Y_i,
  \qquad
  w_{i,h}(0)=\frac{1}{nh}K_I\!\left(\frac{X_i}{h}\right),
  \qquad
  Y_i\equiv 1.
\]
Thus, the argument follows the same bias--stochastic-error decomposition as
the proof of Theorem~1, with the additional simplification that the localized
summands are bounded.

Let $V_0:=f_X(0)\int_I K_I^2(u)\,du.$ For every deterministic $h\to0$, the moment restrictions on $K_I$ and the definition of $\Theta_0(M_S,r_S,\psi_n)$ imply
\begin{align*}
  \sup_{f_X\in\Theta_0(M_S,r_S,\psi_n)}
  \left|\mathbb E\!\left[\widehat f_h(0)\right]-f_X(0)\right|
  \le
  M_S r_S(h)h^S\int_I |u|^S|K_I(u)|\,du=o(h^S).
\end{align*}
Consequently, for every fixed $\eta>0$, the absolute bias is bounded by an
arbitrarily small fixed multiple of $\eta h^S$, uniformly over
$\Theta_0(M_S,r_S,\psi_n)$, for all sufficiently large $n$.

Next, set $Z_{i,h}=h^{-1}K_I(X_i/h)$. A change of variables gives
\begin{align*}
  h\operatorname{Var}(Z_{i,h})=\int_I K_I^2(u)f_X(hu)\,du
    -h\left\{\int_I K_I(u)f_X(hu)\,du\right\}^2 =V_0+O(h).
\end{align*}
Since $|Z_{i,h}-\mathbb EZ_{i,h}|\le \frac{2\|K_I\|_\infty}{h}$, the Bernstein inequality yields, for every fixed $t\in(0,1)$,
\[
\begin{aligned}
  \mathbb P\!\Biggl(
    \pm\left\{\widehat f_h(0)-\mathbb E\widehat f_h(0)\right\}
    >
    \sqrt{\frac{2\log(1/t)\{V_0+O(h)\}}{nh}} +\frac{2\|K_I\|_\infty\log(1/t)}{3nh}
  \Biggr)
  \le t.
\end{aligned}
\]
At bandwidths of order $n^{-1/(2S+1)}$, the second term and the contribution
of the $O(h)$ variance remainder are $o(h^S)$. Hence the leading
fixed-bandwidth radius is
\[
  R_t(h)=\sqrt{\frac{2\log(1/t)V_0}{nh}}+\eta h^S.
\]
Its minimizer and minimum are, respectively,
\[
  h_0^*(t)
  =\left\{
    \frac{2\log(1/t)V_0}{4S^2\eta^2n}
  \right\}^{1/(2S+1)}
\]
and
\[
  R_t\!\left(h_0^*(t)\right)
  =(2S+1)\eta^{1/(2S+1)}
  \left\{
    \frac{2\log(1/t)V_0}{4S^2n}
  \right\}^{S/(2S+1)}.
\]
The first-order condition also gives the useful identity
\[
  \sqrt{\frac{2\log(1/t)V_0}{nh_0^*(t)}}
  =2S\eta\{h_0^*(t)\}^S.
\]

It remains to replace $V_0$ and $h_0^*(t)$ by their feasible counterparts.
Define
\[
  Q_i(g)
  :=\left\{
    K_I\!\left(\frac{X_{2i}}{g}\right)
    -K_I\!\left(\frac{X_{2i-1}}{g}\right)
  \right\}^2.
\]
The variables $\{Q_i(g)\}_{i=1}^{n/2}$ are independent and bounded, and
\[
  \mathbb E[V_{ng}]
  =\frac{1}{g}
    \left[
      \mathbb E\{K_I^2(X/g)\}
      -\{\mathbb EK_I(X/g)\}^2
    \right]
  =V_0+O(g),
\]
whereas $\text{Var}\{Q_i(g)\}=O(g)$. Put $a_n=(\log n)^{-4}$, $g=n^{-1/(2S+1)}$ and
\[
  \mathcal G_n
  =\left\{
    \left|\frac{V_{ng}}{V_0}-1\right|\le Ca_n
  \right\}.
\]
A further application of the Bernstein inequality gives an event indicates that, for some constants $C_1,C_2,\kappa>0$,
\[
  \sup_{f_X\in\Theta_0(M_S,r_S,\psi_n)}
  \mathbb P_{f_X}(\mathcal G_n^c)
  \le C_1\exp\{-C_2nga_n^2\}
  =O\!\left(\exp\{-n^\kappa\}\right).
\]
Note that, when event $\mathcal G_n$ happens,
\[
  \frac{\widehat h_0(t)}{h_0^*(t)}
  =\left(\frac{V_{ng}}{V_0}\right)^{1/(2S+1)}
  =1+O(a_n).
\]
Since the relative mesh size of $\mathcal H_n$ around $h_0^*(t)$ is
$O(a_n)$, the projected bandwidth satisfies $\frac{\widetilde h_0(t)}{h_0^*(t)}=1+O(a_n)$ on $\mathcal G_n$.

We now make explicit the probability slack used in the finite-grid step.
Choose a fixed $t_-\in(0,t)$ that is sufficiently close to $t$. By the preceding first-order
identity, $t_-$ can be chosen so that
\[
  \sqrt{\frac{2\log(1/t_-)V_0}{nh_0^*(t)}}
  -\sqrt{\frac{2\log(1/t)V_0}{nh_0^*(t)}}
  \le \frac{\eta}{2}\{h_0^*(t)\}^S.
\]
Repeating the finite-grid projection argument from Chapter~B.2 in the proof of Theorem \ref{th refinement} gives events
$\mathcal D_n$ satisfying
\[
  q_n
  :=\sup_{f_X\in\Theta_0(M_S,r_S,\psi_n)}
    \mathbb P_{f_X}(\mathcal D_n^c\cap\mathcal G_n)
  =o(1),
\]
such that, on $\mathcal D_n\cap\mathcal G_n$, the centered stochastic bandwidth-projection error is bounded by an arbitrarily small fixed multiple
of $\{h_0^*(t)\}^S$. Together with the uniform $o(h^S)$ bias bound and the negligible Bernstein remainder above, these errors can be chosen to use at
most the remaining $\frac{\eta}{2}\{h_0^*(t)\}^S$ of the fixed bias budget.

Moreover, because $a_n=o(\xi_n)$ and
$\xi_n=(\log n)^{-3}$, when event $\mathcal G_n$ happens, $ \overline r_\eta(t)
  \ge R_t\!\left(h_0^*(t)\right)$ holds for all sufficiently large $n$. Consequently,
\begin{align*}
  \sup_{f_X\in\Theta_0(M_S,r_S,\psi_n)}
  \mathbb P_{f_X}\!\left(f_X(0)>\widehat f_{\widetilde h_0(t)}(0)+\overline r_\eta(t)
  \right) \le t_-+q_n+\sup_{f_X\in\Theta_0(M_S,r_S,\psi_n)}
    \mathbb P_{f_X}(\mathcal G_n^c).
\end{align*}
Since $t-t_->0$ is fixed and $q_n=o(1)$, we have
$t_-+q_n\le t$ for all sufficiently large $n$. Therefore,
\[
  \inf_{f_X\in\Theta_0(M_S,r_S,\psi_n)}
  \mathbb P_{f_X}\!\left(
    f_X(0)\le
    \widehat f_{\widetilde h_0(t)}(0)+\overline r_\eta(t)
  \right)
  \ge 1-t-O\!\left(\exp\{-n^\kappa\}\right).
\]
The lower-tail statement follows identically. Taking $t=\alpha$ proves the
two one-sided statements. Taking $t=\alpha/2$ and applying the union bound
proves the two-sided statement.

Finally, the preceding variance-proxy argument implies
$V_{ng}\xrightarrow{P}V_0$. Hence
\[
  n^{S/(2S+1)}\overline r_\eta(t)
  \xrightarrow{P}
  (2S+1)\eta^{1/(2S+1)}
  \left\{
    \frac{2\log(1/t)V_0}{4S^2}
  \right\}^{S/(2S+1)},
\]
which establishes the minimax shrinking rate.\\

\noindent\textbf{Proof of Theorem \ref{th eta free density}}
Let $h_{\mathrm{nv}}=n^{-1/(2S+1)}$. The deterministic-bandwidth
argument in the proof of Theorem~4 implies that, for every fixed
$t\in(0,1)$, there exists a finite constant $A_t>0$ such that, for all
sufficiently large $n$,
\[
    \inf_{f_X\in\Theta_0(M_S,r_S,\psi_n)}
    \mathbb P\left(
        f_X(0)
        \leq
        \widehat f_{h_{\mathrm{nv}}}(0)
        +
        A_t n^{-\frac{S}{2S+1}}
    \right)
    \geq
    1-t-O\!\left(\exp\{-n^\kappa\}\right),
\]
and the analogous lower-tail inequality also holds. Here the
deterministic bias is $o(h_{\mathrm{nv}}^S)$, while the stochastic
error is of order
\[
    (nh_{\mathrm{nv}})^{-1/2}
    =
    n^{-\frac{S}{2S+1}}.
\]

On the event $V_{ng}\geq V_0/2$, the radius in Theorem~5 satisfies
\[
    \bar r(t)
    =
    d_n
    \left(
        \frac{2\log(1/t)V_{ng}}{n}
    \right)^{\frac{S}{2S+1}}
    \geq
    c_t d_n n^{-\frac{S}{2S+1}}
\]
for some constant $c_t>0$. Since $d_n\to\infty$, this radius is larger
than $A_tn^{-S/(2S+1)}$ for all sufficiently large $n$. Moreover, the
variance-proxy argument above gives
\[
    \mathbb P(V_{ng}<V_0/2)
    =
    O\!\left(\exp\{-n^\kappa\}\right).
\]
Combining these facts with the fixed-bandwidth inequalities proves
the two one-sided statements by taking $t=\alpha$. Taking
$t=\alpha/2$ and applying the union bound proves the two-sided
statement.

Finally,
\[
    \bar r(t)
    =
    O_{\mathbb P}\left(
        d_n n^{-\frac{S}{2S+1}}
    \right).
\]
In particular, for $d_n=\log\log n\sqrt{\log n}$ and every fixed $\tau>0$,
\[
    \frac{\bar r(t)}
         {n^{-\frac{S}{2S+1}+\tau}}
    =
    O_{\mathbb P}\left(
        \frac{\log n\log\log n}{n^\tau}
    \right)
    \longrightarrow 0,
\]
which establishes the nearly-minimax shrinking rate.

\section{Proof of Results in Section \ref{sec 4}}

\def\theequation{C.\arabic{equation}}
\setcounter{equation}{0}

\noindent\textbf{Proof of Proposition \ref{prop RBC EBCI}:}
The proof is direct and we only sketch it here. Note that it suffices to show 
\begin{align}
\label{eq of prop RBC EBCI eq1}
     \P\left(  m(0)\leq \hat m^{\mathrm{rbc}}_{g_n,l_n}(0)+ \sigma\sqrt{2\log(2/\alpha)\sum_{i=1}^n(w_{i,h,b}^{\mathrm{rbc}}(0))^2}\right)\geq 1-\frac{\alpha}{2}-\P(A_n\cup B_n\cup C_n),
\end{align}
Since $ m(0)-\hat m^{\mathrm{rbc}}_{g_n,l_n}(0)\leq \sum_{i=1}^nw^{\mathrm{rbc}}_{i,g_n,l_n}(0)\varepsilon_i+|\text{Bias}_{\mathrm{rbc}}(\beta)|$,
 for any fixed $\alpha\in (0,1)$, using Lemma \ref{lemma 2} and conditional Bernstein inequality implies that, for every realization of $(X_1,..,X_n)$, the following inequality holds with probability no less than $1-\alpha/2$,
\begin{align*}
     m(0)-\hat m^{\mathrm{rbc}}_{g_n,l_n}(0)
     \leq& \sigma\sqrt{2\log(2/\alpha)\sum_{i=1}^n(w_{i,g_n,l_n}^{\mathrm{rbc}}(0))^2}\notag\\
     &+ \frac{B\max_{i\leq n}|w_{i,g_n,l_n}^{\mathrm{rbc}}(0)|}{3}\log(\frac{2}{\alpha})+|\text{Bias}_{\mathrm{rbc}}(\beta)|.
\end{align*}
Together with the definition of $A_n$ and $B_n$, for all $n\geq 1,$ we have
\begin{align}
    \label{eq of prop RBC EBCI eq2}
   \P\left( \left\{ m(0)\leq \hat m^{\mathrm{rbc}}_{g_n,l_n}(0)+ 2\sigma\sqrt{2\log(2/\alpha)\sum_{i=1}^n(w_{i,g_n,l_n}^{\mathrm{rbc}}(0))^2}\right\}\right)\geq 1-\frac{\alpha}{2}-\P(A_n\cup B_n).
\end{align}
Then, together with $C_n$, \eqref{eq of prop RBC EBCI eq2} implies \eqref{eq of prop RBC EBCI eq1}, which asserts 
\begin{align*}
     \P\left( m(0)\leq \hat m^{\mathrm{rbc}}_{g_n,l_n}(0)+ \sigma\sqrt{2\log(1/\alpha)\sum_{i=1}^n(w_{i,b,h}^{\mathrm{rbc}}(0))^2}\right)\geq 1-\frac{\alpha}{2}-\P(A_n\cup B_n)-\P(C_n).
\end{align*}
This completes the proof.
\medskip

\noindent\textbf{Proof of Proposition \ref{prop FLCI normal approximation}:}
Let $  \Delta_\alpha(t):=q(c-t)-q(c+t)$. By the definition of $c_\alpha(t)$, we have $\P\bigl(\lvert Z+t\rvert\le \operatorname{cv}_{1-\alpha}(t)\bigr)=1-\alpha$, where $Z\sim N(0,1)$. Since $Z$ is symmetric, the same identity holds with $t$ replaced by $-t$,
\begin{align}
    \P\bigl(\lvert Z+st\rvert\le c\bigr)=1-\alpha,\qquad s\in\{-1,1\}.
    \label{eq: folded normal calibration identity signed}
\end{align}
Recall that $\mathrm{FLCI}_2(\alpha)=\left[\hat\theta_h-c_\alpha(t)\sigma_n(h), \hat\theta_h+c_\alpha(t)\sigma_n(h)\right].$
For any given $\theta\in\Theta$, by denoting $b_\theta:=\frac{\E[\hat\theta_h]-\theta}{\sigma_n(h)}$, we obtain
\begin{align*}
    \P_\theta\bigl(\theta\in\mathrm{FLCI}_2(\alpha)\bigr)
    =
    \P_\theta\bigl(-\operatorname{cv}_{1-\alpha}(t)-b_\theta\le Z_n\le \operatorname{cv}_{1-\alpha}(t)-b_\theta\bigr) \\
    =
    \P_\theta(Z_n\le \operatorname{cv}_{1-\alpha}(t)-b_\theta)
    -
    \P_\theta(Z_n\le -\operatorname{cv}_{1-\alpha}(t)-b_\theta).
\end{align*}
\textbf{Proof of \eqref{eq prop FLCI normal approximatio eq1}: }
Note that, for each $s\in\{-1,1\}$, Assumption~\ref{as 7} (A2) implies that there exists $\theta_s\in\Theta$ such that $b_{\theta_s}=st.$ Hence, by Assumption~\ref{as 7} (A1) and the fact that $q(z)=(1-z^2)\phi(z)$ is even,
\begin{align}
    \P_{\theta_s}\bigl(\theta_s\in\mathrm{FLCI}_2(\alpha)\bigr)
    =&\Phi(\operatorname{cv}_{1-\alpha}(t)-st)-\Phi(\operatorname{cv}_{1-\alpha}(t)+st)\notag\\
    &+\frac{\kappa_{3,n}}{6}\left\{q(\operatorname{cv}_{1-\alpha}(t)-st))-q(\operatorname{cv}_{1-\alpha}(t)+st)\right\}\notag\\
    &+ r_n(\operatorname{cv}_{1-\alpha}(t)-st)-r_n(-\operatorname{cv}_{1-\alpha}(t)-st)\notag\\
    =&: 1-\alpha+\frac{s\kappa_{3,n}}{6}\Delta_\alpha(t)+\rho_{s,n},
    \label{eq proof of prop 3 eq1}
\end{align}
where the last inequality is according to \eqref{eq: folded normal calibration identity signed}. Since all arguments in $\rho_{s,n}$ are fixed and Assumption~\ref{as 7} (A1) implies
$r_n(z)=o(|\kappa_{3,n}|)$ for every fixed $z\in\mathbb R$, we have $  \rho_{s,n}=o(|\kappa_{3,n}|)$, $ s\in\{-1,1\}.$ For each $n\geq 1$, there exists $s_n\in\{-1,1\}$ such that $s_n\kappa_{3,n}\Delta_\alpha(t)=-\left|\kappa_{3,n}\right|\left|\Delta_\alpha(t)\right|.$
If $\kappa_{3,n}\Delta_\alpha(t)=0$, choose either sign. Then,
\eqref{eq proof of prop 3 eq1} yields
\begin{align*}
    \inf_{\theta\in\Theta}
    \P_\theta\bigl(\theta\in\mathrm{FLCI}_2(\alpha)\bigr)
    &\le
    \P_{\theta_{s_n}}
    \bigl(\theta_{s_n}\in\mathrm{FLCI}_2(\alpha)\bigr) \\
    &=
    1-\alpha
    -
    \frac{|\kappa_{3,n}|}{6}
    |\Delta_\alpha(t)|
    +
    o(|\kappa_{3,n}|),
\end{align*}
which proves \eqref{eq prop FLCI normal approximatio eq1}.
\medskip

\noindent
\textbf{Proof of \eqref{eq prop FLCI normal approximatio eq2}:}
Define
\[
    b_0
    :=
    \frac{\E[\hat\theta_h]-\theta_0}{\sigma_n(h)},
    \qquad
    |b_0|=t,
\]
and write $s_0:=\operatorname{sgn}(b_0)\in\{-1,1\}$. Applying
\eqref{eq proof of prop 3 eq1} with $s=s_0$ gives
\begin{align*}
    \P_{\theta_0}\bigl(\theta_0\in\mathrm{FLCI}_2(\alpha)\bigr)
    =
    1-\alpha
    +
    \frac{s_0\kappa_{3,n}}{6}\Delta_\alpha(t)
    +
    o(|\kappa_{3,n}|).
\end{align*}
Under the stated adverse-sign condition $\operatorname{sgn}(b_0)\kappa_{3,n}\Delta_\alpha(t)<0$, we have $s_0\kappa_{3,n}\Delta_\alpha(t) =-|\kappa_{3,n}||\Delta_\alpha(t)|.$
Therefore, $ \P_{\theta_0}\bigl(\theta_0\in\mathrm{FLCI}_2(\alpha)\bigr)
    \le
    1-\alpha
    -
    \frac{|\kappa_{3,n}|}{6}
    |\Delta_\alpha(t)|
    +
    o(|\kappa_{3,n}|), $
which proves \eqref{eq prop FLCI normal approximatio eq2}.\\

\noindent \textbf{Proof of Proposition \ref{prop EBCI FLCI sharpness}:}  The result  follows immediately from the discussion in Section \ref{sec 2.3}. We thus omit it here. 
\section{Transporting $\eta$-free EBCI to Weighted Average Regression}
\label{app D}

\def\theassumption{D.\arabic{assumption}}
\setcounter{assumption}{0}

\def\thetheorem{D.\arabic{theorem}}
\setcounter{theorem}{0}

\def\theequation{D.\arabic{equation}}
\setcounter{equation}{0}

In this section, we discuss \(\eta\)-free EBCIs for weighted-average estimators for the target parameter $m(0)$. Our goal is not to re-derive the low-level approximation theory for every linear smoother, but to isolate the sufficient conditions under which EBCI calibration becomes immediately usable. 

For the sake of simplicity, we state our result under the additional assumption of homoskedasticity, that is $V(x)\equiv V>0$. We show that, given some user-defined inflation factor $d_n$, $\eta$-free EBCI provides an efficient method for constructing confidence intervals for a large class of weighted average estimators, including local polynomial, sieve and nearest neighbor estimators, under high-level verifiable conditions on the weight function and variance estimation. More specifically, the weighted average estimator is defined by  
\begin{align*}
    \hat m_\lambda(0)=\sum_{i=1}^nW_{i,\lambda}(0)Y_i,
\end{align*}
where $\lambda>0$ is a user-defined and deterministic tuning parameter satisfying $\lambda\asymp n^{-\frac{1}{2S+1}}$.  For example, $\lambda$ may represent the bandwidth for local polynomial estimation, the inverse dimension of a sieve estimator, or the number of neighbors divided by the sample size for a nearest neighbor estimator. The following three assumptions are required to derive a result in the spirit of Section \ref{sec 3.3} for this general class of estimators.

\begin{assumption}[weight and variance estimation]
\label{as 4}
There exist positive sequences $(c_{v,n})_{n\in \mathbb{N}}$, $(e_{v,n})_{n\in \mathbb{N}}$ and $(e_{w,n})_{n\in \mathbb{N}} $  converging to $0$ as $n\to \infty$, constant $C_M>0$ and an estimator $\hat V$ of $V$ such that the following inequalities hold 
    \begin{align*}
        &\P(|\hat V- V |\leq c_{v,n})\geq 1-e_{v,n},\\
        & \P\Big (\max_{1\leq i\leq n} |W_{i,\lambda}(0)|\leq \frac{C_M}{n\lambda} \Big )\geq 1-e_{w,n}.
\end{align*}
Moreover, for any given user-defined $d_n\nearrow\infty$, there exist a vanishing positive number sequence $(e(\eta_n))_{n\in \mathbb{N}}$ and a random variable $C_W>0$ independent of the sigma algebra generated collected samples such that
\begin{align*}
         \max\Big\{\P\Big ( C_W> \eta_nn\lambda\sum_{i=1}^n W_{i,\lambda}^2(0) \Big ), \P\Big ( C_W< n\lambda\sum_{i=1}^n W_{i,\lambda}^2(0) \Big )\Big\}\leq e(\eta_n).
\end{align*}
\end{assumption}

\begin{assumption} [smoothness and bias]
\label{as 5}
Let $\mathcal{F}_S$ denote a fixed function class. There exists a positive sequence  $(e_{B,n})_{n \in \mathbb{N}}$ converging to $0$  such that 
    \begin{align*}  
\P\Big (\sup_{m\in\mathcal F_S}\Big  |\sum_{i=1}^n W_{i,\lambda}(0)m(X_i)-m(0)\Big |\le C_B\lambda^S\Big )\ge 1-e_{B,n}, 
\end{align*}
where the constant $C_B>0$ does not  depend on $n$ and $\lambda$.
\end{assumption}

\begin{assumption} [stability of tuning]
\label{as 6}
  Suppose that $\lambda/\lambda'\to c>0$ holds as $n\to \infty$. By letting $D_n=\frac{1}{2}\xi_n\sqrt{\log n}(2\log(\frac{1}{\alpha})VC_W)^{\frac{S}{2S+1}}$, where $\xi_n>0$ is any given slow diverging sequence, we assume there exists a positive vanishing sequence $(e(D_n))_{n \in \mathbb{N}}$  such that
  \begin{align*}
     \P\Big(|\hat m_{\lambda}(0)-\hat m_{\lambda'}(0)|\leq D_n n^{-\frac{S}{2S+1}} \Big)\geq 1-e(D_n).
  \end{align*}
\end{assumption}
Assumptions \ref{as 4}--\ref{as 6} are satisfied for many widely used linear smoothers, like sieve estimators. According to the classical results in \citeA{newey1997convergence} and \citeA{chen2015optimal}, Assumption \ref{as 4} can be specified based on the applications of the equivalent kernel of sieve estimators and Bernstein inequalities for Hermitian random matrices (see, for example,  \citeA{tropp2015introduction}). $C_W$ is an auxiliary weight-scale benchmark used only in the proof. It is independent of the observed sample so that the associated oracle tuning parameter can be treated conditionally as fixed. The reported interval does not depend on $C_W$. Assumption \ref{as 5} is also mild for sieve estimators. The corresponding weights define the empirical \(L^2\)-projection evaluated at point $0$. Hence, the pointwise bias is controlled by the sieve approximation error multiplied by a pointwise Lebesgue factor. According to \citeA{chen2015optimal}, for localized sieve bases such as splines or compactly supported wavelets, this pointwise projection factor is bounded under standard Gram-stability conditions. By letting integer $K$ be the dimension of sieve basis and $\lambda=\frac{1}{K}$, an \(S\)-smooth class often yields the bias rate, $\lambda^S$ (or $K^{-S}$) with high probability. Assumption \ref{as 6} is a mild stochastic stability condition and  typically follows from (uniform) standard weight, variance, and bias control for standard linear smoothers.

\begin{theorem}
    \label{th WAE}
    Suppose Assumptions \ref{as 1} and \ref{as 4}-\ref{as 6} hold, let  $\mathcal{F}_S$ be a given function class and $\alpha\in (0,1)$. Then, based on the $d_n$ and $D_n$ introduced in Assumption \ref{as 5} and \ref{as 6},  we have 
    \begin{align*}
       &\inf_{m\in\mathcal{F}_S}\P\Big(m(0)\in \Big[\hat m_{\lambda}(0)\pm \Big(\frac{D_n+d_n}{2}\Big)\Big(\frac{2\log(2/\alpha)\hat V\hat C_W}{n}\Big)^{\frac{S}{2S+1}}\Big]\Big)\geq 1-\alpha- \textbf{CE}_n\ .
     \end{align*}  
     where $\hat C_W  =n^{\frac{2S}{2S+1}}\sum_{i=1}^nW_{i,\lambda}^2(0)$, $\textbf{CE}_n\asymp e(D_n)+e(\eta_n)+e_{v,n}+e_{w,n}+e_{B,n}+E_{\varsigma}n^{-\frac{(\varsigma-2)S-1}{2S+1}},$
and the constant $\varsigma$ is introduced in Assumption \ref{as 3}.
\end{theorem}
Theorem \ref{th WAE} should be read as a verification theorem for transporting EBCI beyond kernel smoothers. Once a weighted-average estimator satisfies Assumptions \ref{as 4}--\ref{as 6}, no further estimator-specific calibration argument is needed and the \(\eta\)-free EBCI immediately delivers honest pointwise inference. Moreover, for sieve smoothers, the preceding discussion of Assumptions \ref{as 4}--\ref{as 6} shows that the $\eta$-free EBCI of sieve estimators could attain \ref{G1} and nearly the minimax optimal rate by repeating the techniques used in \citeA{chen2015optimal}. \\

\noindent \textbf{Proof of Theorem \ref{th WAE}:}
The proof is similar to that of Theorem \ref{th eta free regression}, and we thus only sketch it here. Additionally, it suffices to focus only on the one-sided case. First, by letting $\varepsilon_{iL_n}=\varepsilon_i1[|\varepsilon_i|\leq L_n]$ and $\varepsilon_{iL^-_n}=\varepsilon_i1[|\varepsilon_i|> L_n]$, we have decomposition 
\begin{align*}
    m(0)-\hat m_{\lambda}(0)&\leq \sum_{i=1}^nW_{i,\lambda}(0)\varepsilon_{iL_n}+\Big|\sum_{i=1}^nW_{i,\lambda}(0)\varepsilon_{iL^-_n}\Big|+\Big|\sum_{i=1}^nW_{i,\lambda}(0)m(X_i)-m(0)\Big|\\
    &=: T_{1n}+T_{2n}+B_n.
\end{align*}
Together with Assumption \ref{as 3}, some simple algebra shows that, for all $n\geq M_2$, the following event happens with probability no less than $1-e_{B,n}-E_{\varsigma}n^{-\frac{(\varsigma-2)S-1}{2S+1}}$,
\begin{align}
\label{eq proof of th WAE eq1}
     m(0)\leq \hat m_{\lambda}(0)+ T_{1n}+2C_B\lambda^S.
\end{align}
Meanwhile, using Lemma \ref{lemma 2} and conditional Bernstein inequality immediately asserts
\begin{align*}
    \P\left(T_{1n}\leq \sqrt{2\log(1/\alpha)V\sum_{i=1}^nW_{i,\lambda}^2(0)}+\max_{i\leq n}|W_{i,\lambda}(0)|\frac{\log(1/\alpha)L_n}{3}\right)\geq 1-\alpha.
\end{align*}
By letting $L_n=\sqrt{n\lambda}/\log n$, using Assumption \ref{as 4} yields that, for all $n\geq M_1$, we have
\begin{align}
\label{eq proof of th WAE eq2}
    \P\left(T_{1n}\leq \sqrt{\frac{2\log(1/\alpha)VC_W}{n\lambda}}+\frac{C_M\log(1/\alpha)}{3(\log n)\sqrt{n\lambda}}\right)\geq 1-\alpha-e_{w,n}-e(\eta_n).
\end{align}
Combining \eqref{eq proof of th WAE eq1} and \eqref{eq proof of th WAE eq2} asserts that the following inequality holds with probability no less than $1-\alpha-e_{w,n}-e(\eta_n)-e_{B,n}-E_{\varsigma}n^{-\frac{(\varsigma-2)S-1}{2S+1}}$, 
\begin{align*}
     m(0)\leq \hat m_{\lambda}(0)&+ \sqrt{\frac{2\log(1/\alpha)VC_W}{n\lambda}}+\frac{C_M\log(1/\alpha)}{3(\log n)\sqrt{n\lambda}}+2C_B\lambda^S,\ \forall\ n\geq M_1\lor M_2.
\end{align*}
Since $\lambda\asymp n^{-\frac{1}{2S+1}}$ and $d_n\nearrow\infty$,  there exists sufficiently large $M'\geq 1$ independent of $n,\lambda$ such that
\begin{align}
    \label{eq proof of th WAE eq3}
    m(0)\leq \hat m_{\lambda}(0)+\sqrt{\frac{2\log(1/\alpha)VC_W}{n\lambda}}+3C_B\lambda^S, \forall\ n\geq M'.
\end{align}
According to the independent property of $C_W$, taking $\eta=3C_B$ and $\lambda^*=\Big(\frac{2\log(1/\alpha)VC_W}{4S^2\eta ^2n}\Big)^{\frac{1}{2S+1}}$ asserts
\begin{align}
    \label{eq proof of th WAE eq4}
   \P\Big( m(0)&\leq \hat m_{\lambda^*}(0)+(2S+1)\eta^{\frac{1}{2S+1}}\Big(\frac{2\log(\frac{1}{\alpha})  V C_W}{4S^2n}\Big)^{\frac{S}{2S+1}}\Big)\notag\\
   &\geq  1-\alpha-e_{w,n}-e(\eta_n)-e_{B,n}-E_{\varsigma}n^{-\frac{(\varsigma-2)S-1}{2S+1}}.
\end{align}
Direct application of Assumption \ref{as 6} yields
\begin{align*}
      &\P\Big( m(0)\leq \hat m_{\lambda}(0)+A+B\Big)\geq 1-\alpha-e_{w,n}-e(\eta_n)-e(D_n)-e_{B,n}-E_{\varsigma}n^{-\frac{(\varsigma-2)S-1}{2S+1}},\\
&\text{where}\qquad A=\frac{D_n}{2}(\frac{2\log(1/\alpha)VC_W}{n})^{\frac{S}{2S+1}},\ B=(2S+1)\eta^{\frac{1}{2S+1}}\Big(\frac{2\log(\frac{1}{\alpha})  V C_W}{4S^2n}\Big)^{\frac{S}{2S+1}}.
\end{align*}
Meanwhile, according to Assumption \ref{as 4}, we obtain
\begin{align*}
    &\P\Big(B\leq \frac{d_n}{2}\Big(\frac{2\log(1/\alpha)\hat V\hat C_W}{n}\Big)^{\frac{S}{2S+1}}\Big)\geq 1- e_{v,n}-e(\eta_n),\\
      &\P\Big(A\leq \frac{D_n}{2}\Big(\frac{2\log(1/\alpha)\hat V\hat C_W}{n}\Big)^{\frac{S}{2S+1}}\Big)\geq 1- e_{v,n}-e(\eta_n),\ \forall\ n\geq M'',
\end{align*}
where $\hat C_W=n^{\frac{2S}{2S+1}}\sum_{i=1}^nW_{i,\lambda}^2(0)$ and $M''\geq 1$ is an integer independent of $n,\lambda$. We then finish the proof of Theorem \ref{th WAE}.

\section{Technical Lemmas}
\label{app E}

\def\thelemma{E.\arabic{lemma}}
\setcounter{lemma}{0}

\def\theequation{E.\arabic{equation}}
\setcounter{equation}{0}
\begin{lemma}
    \label{lemma 1}
    Based on some given probabilistic space $(\Omega,\mathcal{F},\P)$, define $\{(X_i,Y_i)\}_{i=1}^n$ as a group of mutually independent random variables taking values in some Borel space and denote $X_i\in (\mathcal{X}_i, \mathcal{A}_i)$ and $Y_i\in (\mathcal{Y}_i,\mathcal{B}_i)$. For any given measurable mapping $g_i:\mathcal{Y}_i\to\mathbb{R}$ such that, for every $i$ and $n$, $g_i$ and $\prod_{j=1}^ng_j\in L^1(\P)$ are $\P$-integrable, we have
    \begin{align*}
     \E[\prod_{i=1}^ng_i(Y_i)|\sigma(X_1,...,X_n)]=\prod_{i=1}^n\E[g_i(Y_i)|\sigma(X_i)]\ \ \ \ \text{a.s.},
    \end{align*}
    where $\sigma(X)$ denotes the sigma algebra generated by random variable $X$.
\end{lemma}

\begin{lemma}
    \label{lemma 2}
    Under the conditions of Lemma \ref{lemma 1}, provided that $\{(X_i,Y_i)\}_{i=1}^n$ is defined on probability space $(\Omega,\mathcal{F},\P)$, for any $B_i\in \mathcal{B}_i$, we have
    \begin{align*}
        \P((Y_1,...,Y_n)\in \prod_{i=1}^nB_i|\sigma(X_1,...,X_n))=\prod_{i=1}^n\P(Y_i\in B_i|\sigma(X_i))=\prod_{i=1}^n\P(Y_i\in B_i|\sigma(X_1,...,X_n)).
    \end{align*}
\end{lemma}

\begin{lemma}
    \label{lemma 3}
  Provided that Assumption \ref{as 1} holds with $[a,b]=[-1,1]$, based on $\M_{1h}$ defined in \eqref{eq local polynomial}, the following inequality holds for all $h\leq H_0:=\max\{h>0:\inf_{|x|\leq h}f_X(x)\geq \frac{f_X(0)}{2}, \sup_{|x|\leq h}f_X(x)\leq \frac{3f_X(0)}{2}\}$ and $n\geq 1$,
  \begin{align}
     \label{eq lemma 3 eq1}
      \P(\mathcal{E}_{\min}):=&\P\Big(\lambda_{\min}(\M_{1h})\geq \frac{1}{4}nhf_X(0)\lambda_{\min}(\Gamma_1)\Big)\geq 1-(S+1)\exp(-c_1nh),\\
    \label{eq lemma 3 eq2}
      \P(\mathcal{E}_{\max}):=&\P\Big(\lambda_{\max}(\M_{2h})\leq \frac{3}{2}nhf_X(0)\lambda_{\max}(\Gamma_2)\Big)\geq 1-(S+1)\exp(-c_2nh),
  \end{align}
  where $c_1,c_2>0$ is a fixed constant depending only on $S,K,f_X(0)$ and matrices $\Gamma_1$, $\Gamma_2$.
    \end{lemma}

\begin{corollary}
\label{corollary a1}
    Provided that Assumption \ref{as 1} holds with $[a,b]=[0,1]$, based on $\M_{kh}$, $\Gamma'_k$ and $r(u)$ defined in \eqref{eq local polynomial}, the following inequality holds for all $h\leq H_0':=\max\{h>0:\inf_{x\leq h}f_X(x)\geq \frac{f_X(0)}{2}, \sup_{x\leq h}f_X(x)\leq \frac{3f_X(0)}{2}\}$ and $n\geq 1$,
    \begin{align}
        \label{eq corollary a1 eq1}
      \P(\mathcal{E}_{\min}'):=&\P\Big(\lambda_{\min}(\M_{1h})\geq \frac{1}{4}nhf_X(0)\lambda_{\min}(\Gamma'_1)\Big)\geq 1-(S+1)\exp(-c'_1nh),\\
    \label{eq corollary a1 eq2}
      \P(\mathcal{E}_{\max}'):=&\P\Big(\lambda_{\max}(\M_{2h})\leq \frac{3}{2}nhf_X(0)\lambda_{\max}(\Gamma'_2)\Big)\geq 1-(S+1)\exp(-c'_2nh),
    \end{align}
  where $c'_1,c'_2>0$ is a fixed constant depending only on $S,K,f_X(0)$ and matrices $\Gamma'_1$, $\Gamma'_2$.
\end{corollary}

\begin{lemma}
    \label{lemma 4}
    Based on the conditions and notations introduced in Lemma \ref{lemma 3} and the assumption that $\Gamma_1$ is positive definite, the following inequality holds for all $n\geq 1$ and $h\leq H_0$,
    \begin{align}
    \label{eq lemma 4 eq1}
        \P(\mathcal{E}_{\max W}):=\P\Big(\max_{1\leq i\leq n}|W_{ih}(0)|\leq \frac{4\sqrt{S+1}}{nhf_X(0)\lambda_{\min}(\Gamma_1)}\Big)\geq 1-(S+1)\exp(-c_3nh),
    \end{align}
where $c_3>0$ is a fixed constant depending only on $S,K,f_X(0)$ and matrix $\Gamma_1$. Moreover, by letting $l_S(u)=K(u)\e_0^\top \Gamma_1^{-1}r(u)$, for all $\epsilon>0$ such that $ \epsilon+L_f(S+1)h\leq \frac{1}{2}f_X(0)\lambda_{\min}(\Gamma_1)$, we have
    \begin{align}
    \label{eq lemma 4 eq2}
        &\P\Big(\max_{1\leq i\leq n}\Big|W_{ih}(0)-\frac{1}{nhf_X(0)}l_S(\frac{X_i}{h})\Big|\leq \frac{2\sqrt{S+1}}{nhf_X^2(0)\lambda_{\min}(\Gamma_1)}(L_fh+\epsilon)\Big)\notag\\
        &\  \ \ \ \ \ \ \ \ \ \ \ \ \  \ \  \geq 1-2(S+1)\exp\left(-\frac{nh\epsilon^2}{3f_X(0)(S+1)^2+\frac{8}{3}(S+1)\epsilon}\right).
    \end{align}
\end{lemma}

\begin{corollary}
\label{corollary a2}
 Based on the conditions and notations introduced in Corollary \ref{corollary a1} and the assumption that $\Gamma'_1$ is positive definite, the following inequality holds for all $n\geq 1$ and $h\leq H_0'$,
     \begin{align}
    \label{eq lcorollary a2 eq1}
        \P(\mathcal{E}'_{\max W}):=\P\Big(\max_{1\leq i\leq n}|W_{ih}(0)|\leq \frac{4\sqrt{S+1}}{nhf_X(0)\lambda_{\min}(\Gamma'_1)}\Big)\geq 1-(S+1)\exp(-c'_3nh),
    \end{align}
where $c'_3>0$ is a fixed constant depending only on $S,K,f_X(0)$ and matrix $\Gamma'_1$. Moreover, by letting $l'_S(u)=K(u)\e_0^\top (\Gamma'_1)^{-1}r(u)$, for all $\epsilon>0$ such that $ \epsilon+L_f(S+1)h\leq \frac{1}{2}f_X(0)\lambda_{\min}(\Gamma'_1)$, we have
    \begin{align}
    \label{eq corollary a2 eq2}
        &\P\Big(\max_{1\leq i\leq n}\Big|W_{ih}(0)-\frac{1}{nhf_X(0)}l'_S(\frac{X_i}{h})\Big|\leq \frac{2\sqrt{S+1}}{nhf_X^2(0)\lambda_{\min}(\Gamma'_1)}(L_fh+\epsilon)\Big)\notag\\
        &\  \ \ \ \ \ \ \ \ \ \ \ \ \  \ \  \geq 1-2(S+1)\exp\left(-\frac{nh\epsilon^2}{3f_X(0)(S+1)^2+\frac{8}{3}(S+1)\epsilon}\right).
    \end{align}
\end{corollary}

\begin{lemma}
    \label{lemma 5}
  Based on the conditions and notations introduced in Lemma \ref{lemma 3} and the assumption that $\Gamma_1$ is positive definite, for all  $h\leq H_0$, $\epsilon>0$ and $n\geq 1$, we have
   \begin{align}
   \label{eq lemma 5 eq1}
       \P&\left(\Big|\sum_{i=1}^nW_{ih}^2(0)-\frac{1}{nhf_X(0)}\int_{-1}^1l_S^2(u)du\Big|>\frac{C^{*}}{nh}(\epsilon+h)\right)\notag\\
       &\leq  (S+1)e^{-c_{1}nh}+4(S+1)\exp\Big(-c_4nh\min\Big\{\epsilon^2,\epsilon\Big\}\Big),
   \end{align}
  where $l_S(u)$ is introduced in Lemma \ref{lemma 4} and $c_1$ is defined in \eqref{eq lemma 3 eq1}. $c_4, C^*>0$ are constants depending only on $S,K, f_X(0),\ \Gamma_1$ and $\Gamma_2$. Moreover, according to the $\mathcal{E}_{\min}$ and $\mathcal{E}_{\max}$ defined in Lemma \ref{lemma 3}, we further have the following inequality holds almost surely,
   \begin{align}
       \label{eq lemma 5 eq2}
    \sum_{i=1}^nW_{ih}^2(0)\mathbf{1}[\mathcal{E}_{\min}\cap\mathcal{E}_{\max}]\leq \frac{6\lambda_{\max}(\Gamma_2)}{nhf_X(0)\lambda_{\min}^2(\Gamma_1)}\ \ \rm{a.s.}.
   \end{align}
\end{lemma}

\begin{corollary}
    \label{corollary a3} 
    Define $l_{S}'(u)=K(u)\e_0^{\top}(\Gamma_1')^{-1}r(u)$, where $\Gamma_k'$ is introduced in \eqref{eq notation}. Then, based on the conditions and notations used in Corollary \ref{corollary a1}, the following inequality holds for all $h\leq H'_0$, $\epsilon>0$ and $n\geq 1$,
    \begin{align}
   \label{eq corollary a3 eq1}
       \P&\left(\Big|\sum_{i=1}^nW_{ih}^2(0)-\frac{1}{nhf_X(0)}\int_{0}^1(l_S'(u))^2du\Big|>\frac{C^{**}}{nh}(\epsilon+h)\right)\notag\\
       &\leq  (S+1)e^{-c'_{1}nh}+4(S+1)\exp\Big(-c_4nh\min\Big\{\epsilon^2,\epsilon\Big\}\Big),
   \end{align}
   where $c_1'$ and $c_4$ are introduced in Corollary \ref{corollary a1} and Lemma \ref{lemma 5}. $C^{**}>0$ is a constant depending only on $S,K, f_X(0),\ \Gamma'_1$ and $\Gamma'_2$. Moreover, according to the $\mathcal{E}'_{\min}$ and $\mathcal{E}'_{\max}$ defined in Corollary \ref{corollary a1}, we further have inequality
   \begin{align}
       \label{eq corollary a3 eq2}
    \sum_{i=1}^nW_{ih}^2(0)\mathbf{1}[\mathcal{E}'_{\min}\cap\mathcal{E}'_{\max}]\leq \frac{6\lambda_{\max}(\Gamma'_2)}{nhf_X(0)\lambda_{\min}^2(\Gamma'_1)}.
   \end{align}
\end{corollary}

\begin{lemma}
    \label{lemma 6}
      Based on the conditions and notations introduced in Lemma \ref{lemma 4}, there exists constants $c_5,c_6,c_7>0$ independent of $n$ and $h$ such that the following inequality holds for all $n\geq 1$, $h<H_0$, and $0<\epsilon<\epsilon_0$, where $\epsilon_0:=\min\Big\{1-h,\ \frac{1}{2}f_X(0)\lambda_{\min}(\Gamma_1)-L_f(S+1)h\Big\}$,
      \begin{align}
          \label{eq lemma 6 eq1}
          \P\Big(&\Big|nh\sum_{i=1}^nW^2_{ih}(0)V(X_i)-\frac{V(0)}{f_X(0)}\int_{-1}^1l_S^2(u)du\Big|>c_5h+\epsilon\Big)\notag\\
          &\leq \exp\Big(-\frac{f_X^2(0)nh}{8}\Big)+c_6\exp\Big(-c_7nh\min\{\epsilon^2,\epsilon\}\Big).
      \end{align}
\end{lemma}

\begin{corollary}
    \label{corollary a4}
    Based on the notations and conditions introduced in Corollary \ref{corollary a2},  there exists constants $c'_5,c'_6,c'_7>0$ independent of $n$ and $h$ such that the following inequality holds for all $n\geq 1$, $h<H_0$, and $0<\epsilon<\epsilon'_0$, where $\epsilon'_0:=\min\Big\{1-h,\ \frac{1}{2}f_X(0)\lambda_{\min}(\Gamma'_1)-L_f(S+1)h\Big\}$, 
    \begin{align}
          \label{eq corollary a4 eq1}
          \P\Big(&\Big|nh\sum_{i=1}^nW^2_{ih}(0)V(X_i)-\frac{V(0)}{f_X(0)}\int_{0}^1(l'_S(u))^2du\Big|>c_5'h+\epsilon\Big)\notag\\
          &\leq \exp\Big(-\frac{f_X^2(0)nh}{8}\Big)+c_6'\exp\Big(-c'_7nh\min\{\epsilon^2,\epsilon\}\Big).
      \end{align}
\end{corollary}

\begin{lemma}
    \label{lemma 7}
    Under Assumptions \ref{as 1} and \ref{as 3}, for any given $\beta\geq 2$, when $h=O(n^{-\frac{1}{2S+1}})$, there exists some $c_8,c_9>0$ such that the following inequality holds for all $n\geq 1$ and $h<H_0$,
    \begin{align}
        \label{eq lemma 7 eq1}
        \P\Big(\sum_{i=1}^n&W_{ih}^2(0)V(X_i)(\varepsilon_i^2-1)\leq \frac{(\log n)^{-\beta}}{nh}\Big)\notag\\
        &\geq 1-c_8(\log n)^{\beta}n^{-\frac{2S(\varsigma-2)}{2S+1}}-10(S+1)e^{-(c_1\lor c_2)nh}-e^{-c_9(\log n)^{\beta}}.
    \end{align}
    Moreover, \eqref{eq lemma 7 eq1} holds whenever $0$ is an interior or boundary point. 
\end{lemma}

\begin{lemma}
\label{lemma 8}
Suppose Assumptions \ref{as 1}-\ref{as 3} hold. Based on $\hat m_{-i}(X_i)$ and pilot bandwidth $b=O(n^{-\frac{1}{3}})$ introduced in Theorem \ref{th refinement}, there exist $\kappa>0$ and $N_{*}\in \mathbb{N}^+$ independent of $n$ and $b$ such that the following inequality holds for all $n\geq N_{*}$ and $b>0$,
\begin{align}
    \label{eq lemma 8 eq1}
    &\P\Big(\max_{i\in \{j\leq n:|X_j|\leq g\}}|\hat m_{-i}(X_i)-m(X_i)|\leq n^{-\frac{1}{4}}\Big)\geq 1- n^{-(\frac{2S}{2S+1}+\kappa)},
\end{align}
If the errors are sub-Gaussian, the remainder can be strengthened to $1-e^{-n^\kappa}$.
\end{lemma}

\begin{lemma}
    \label{lemma 10}
    Under Assumptions \ref{as 1}-\ref{as 3}, whenever $0$ is an interior or boundary point, by letting $b=n^{-\frac{1}{2S+1}}$, there exists some constants $c_{10}>0$ depends only on $(S,K,\Gamma_k,f_X(0),L_m)$ such that
    \begin{align*}
        \P\Big(|\hat m_{b}(0)-m(0)|\leq  \frac{\log n\log\log n}{n^{S/(2S+1)}} \Big)\geq 1- c_{10}n^{-(\frac{(\varsigma-2)S-1}{2S+1})}, \ \forall\ n\geq 1,
    \end{align*}
    where $\varsigma>4+\frac{1}{S}$ is introduced in Assumption \ref{as 3}. If the errors are sub-Gaussian, the remainder can be strengthened to $1-e^{-n^\kappa}$. 
\end{lemma}

\begin{lemma}
    \label{lemma 9}
    Suppose Assumptions \ref{as 1}-\ref{as 3} hold and $0$ is an interior point. Based on $H_0$ and $\hat C_{V,0}$ defined in Lemma \ref{lemma 3} and Theorem \ref{th eta free regression}, for all $\beta>0$ and $n\geq 1$ such that $h_{nv}=n^{-\frac{1}{2S+1}}\leq H_0$, we have
    \begin{align*}
  \P\Big(|\hat C_{V,0}-C_{V}|\leq \frac{c_{11}}{(\log n)^{\beta-1}}\Big)\geq 1-c_{12} (\log n)^{\beta(\varsigma-1)}n^{-\frac{(\varsigma-2)S-1}{2S+1}},
    \end{align*} 
    where $c_{11},c_{12}>0$ are fixed constants independent of $n$.
\end{lemma}

\begin{lemma}
    \label{lemma 11}
    Suppose Assumptions \ref{as 1} and \ref{as 8}--\ref{as 10} hold. Based on $W_{iK}(0)$ defined in \eqref{eq linear sieve}, there exist constants $B_{\max}, B_2, B_1, c, C>0$ independent of $K$ and $n$ such that the following inequality holds for all $n\geq 1$
    \begin{align*}
        &\P\Big(\bigcap_{i=1}^3\mathcal{G}_i\Big)\geq 1-C\exp\Big(-\frac{cn}{K}\Big),\\
  \text{where}\qquad   &\mathcal{G}_1=\Big\{\sum_{i=1}^n|W_{iK}(0)|\leq B_{1}\Big\},\\
  &\mathcal{G}_2=\Big\{\sum_{i=1}^nW_{iK}^2(0)\leq B_{2}K/n\Big\},\\
  &\mathcal{G}_3=\Big\{\max_{i\leq n}|W_{iK}(0)|\leq B_{\max}K/n\Big\},
    \end{align*}
\end{lemma}


\noindent \textbf{Proof of Lemmas \ref{lemma 1} and \ref{lemma 2}} Regarding Lemma \ref{lemma 2} is a natural corollary of Lemma \ref{lemma 1}, we focus solely on the proof of Lemma \ref{lemma 1}. By denoting $\prod_{i=1}^ng_i(Y_i)=Z$ and $\prod_{i=1}^n\E[g_i(Y_i)|\sigma(X_i)]=M$, it suffices to prove the following identity holds for every given $C\in \mathcal{F}_n:=\sigma(X_1,...,X_n)$,
\begin{align}
    \label{eq appC a1}
   \int_{C}Zd\P=\int_CMd\P.
\end{align}
 More specifically, our proof is decomposed into two steps and the backbone is “$\pi-\lambda$” Theorem (or Dynkin's Theorem) in measure theory.

\textbf{Step 1} By defining $\mathcal{C}:=\{\bigcap_{i=1}^n\{X_i\in A_i\}:A_i\in \mathcal{A}_i\}$, we can show that the following identity holds for every $C\in \mathcal{C}$, 
\begin{align*}
    \int_CZd\P&=\E[\prod_{i=1}^n1[X_i\in A_i]g_i(Y_i)]\\
    &=\E[\prod_{i=1}^n(1[X_i\in A_i]\E[g_i(Y_i)|\sigma(X_i)])]\\
    &=\E\Big[\underbrace{\prod_{i=1}^n1[X_i\in A_i]}_{C}\underbrace{\prod_{i=1}^n\E[g_i(Y_i)|\sigma(X_i)]}_{M}\Big]=\int_{C}Md\P.
\end{align*}
Thus, we prove that \eqref{eq appC a1} holds for every $C\in \mathcal{C}$.

\textbf{Step 2} By letting $\mathcal{D}=\{D\in \mathcal{F}_n:\int_{D}Zd\P=\int_{D}Md\P\}$, together with \textbf{Step 1}, we know $\mathcal{C}$ is a $\pi$ system such that $\sigma(\mathcal{C})=\mathcal{F}_n$ and $\mathcal{C}\subset \mathcal{D}\subset \mathcal{F}_n$. Thus, we only need to prove that $\mathcal{D}$ is a $\lambda$-system since“$\pi-\lambda$” Theorem implies $\mathcal{F}_n=\sigma(\mathcal{C})=\lambda(\mathcal{C})\subset \mathcal{D}\subset \mathcal{F}_n$. First, it is obvious that $\Omega\in \mathcal{D}$. Second, for any $D\in \mathcal{D}$, by denoting $\Omega\backslash D=D^c$, we have 
\begin{align*}
    \int_{D^c}Zd\P=\int_\Omega Zd\P-\int_{D}Zd\P=\int_\Omega Md\P-\int_{D}Zd\P=\int_\Omega Md\P-\int_{D}Md\P=\int_{D^c}Md\P.
\end{align*}
At last, provided that $\{D_k\}_{k=1}^m\subset \mathcal{D}$ is pair-wise disjoint, some simple algebra yields 
\begin{align*}
    \int_{\cup_{k=1}^mD_k}Zd\P=\int_{\cup_{k=1}^mD_k}Md\P.
\end{align*}
Then, together with the definition of $\lambda$-system, we finish the proof.\\

\noindent\textbf{Proof of Lemma \ref{lemma 3}} By letting $W_{k,i}=K^k_{ih}r(\frac{X_i}{h})r^\top(\frac{X_i}{h})$, definitions of $\M_{kh}$ (see \eqref{eq local polynomial}) imply $\M_{kh}=\sum_{i=1}^nW_{k,i}$ and $\{W_{k,i}\}_{i=1}^n$ is IID. Thus, for all $h<H_0$, some simple algebra shows that 
\begin{align}
\label{eq proof of lemma 3 eq1}
    &\lambda_{\min}(\E[\M_{1h}])=n\lambda_{\min}(\E[W_{1,1}])\geq \frac{1}{2}nhf_X(0)\lambda_{\min}(\Gamma_1),\\
\label{eq proof of lemma 3 eq2}
      &\lambda_{\max}(\E[\M_{2h}])=n\lambda_{\max}(\E[W_{2,1}])\leq \frac{3}{2}nhf_X(0)\lambda_{\max}(\Gamma_2).
\end{align}
Meanwhile, applying Theorem 5.1.1 of \citeA{tropp2015introduction} asserts 
\begin{align}
\label{eq proof of lemma 3 eq3}
    \P(\lambda_{\min}(\M_{1h})\leq \frac{1}{2}\lambda_{\min}(\E[\M_{1h}]))\leq (S+1)\exp(-c_1nh),\\
\label{eq proof of lemma 3 eq4}
    \P(\lambda_{\max}(\M_{2h})\geq \frac{3}{2}\lambda_{\max}(\E[\M_{2h}]))\leq (S+1)\exp(-c_2nh),
\end{align}
where $c_1, c_2>0$ are constants depending only on $S,K,f_X(0)$ and $\Gamma_1,\Gamma_2$.
Thus, combining \eqref{eq proof of lemma 3 eq1} (\eqref{eq proof of lemma 3 eq2}) with \eqref{eq proof of lemma 3 eq3} (\eqref{eq proof of lemma 3 eq4}) finishes the proof.\\

\noindent\textbf{Proof of Lemma \ref{lemma 4}} The proof of \eqref{eq lemma 4 eq1} is a natural corollary of Lemma \ref{lemma 3}. First, please note that 
\begin{align}
    \label{eq proof of lemma 4 eq1}
    \max_{i,n}|W_{ih}(0)|\leq ||K||_{\infty}||\e_0||\ ||\M_{1h}^{-1}||\ \max_{i,n}||r(\frac{X_i}{h})||\leq \frac{\sqrt{S+1}}{\lambda_{\min}(\M_{1h})}.
\end{align}
Then, combining \eqref{eq proof of lemma 4 eq1} with \eqref{eq lemma 3 eq1} finishes the proof of \eqref{eq lemma 4 eq1}. The proof of \eqref{eq lemma 4 eq2} is a bit more delicate. We decompose the proof into multiple steps. 

\noindent \textbf{Step 1} We first introduce the following notations for the convenience of the statement later.
\begin{align*}
    \A_{1h}:=\frac{1}{nh}\M_{1h},\ \ \A_1:=\frac{1}{nh}\M_1=f_X(0)\Gamma_1,\ \ \triangle_{1h}=\A_{1h}-\A_1.
\end{align*}
If the event $E_1:=\{\lambda_{\min}(\A_{1h})>0\}$ happens, the following results immediately hold
\begin{align}
    \M_{1h}^{-1}=\frac{1}{nh}\A_{1h}^{-1},\ \Gamma_1^{-1}=&f_X(0)\A_1^{-1},
    W_{ih}(0)=\frac{1}{nh}K_{ih}\e_0^\top\A_{1h}^{-1}r(\frac{X_i}{h}),\ \ \frac{1}{nh}l_S(\frac{X_i}{h})=\frac{1}{nh}K_{1h}\e_0^\top\A_1^{-1}r(\frac{X_i}{h}),\notag\\
     \label{eq proof of lemma 4 eq3}
     \max_{1\leq i\leq n}&\Big|W_{ih}(0)-\frac{1}{nhf_X(0)}l_S(\frac{X_i}{h})\Big|\leq \frac{\sqrt{S+1}}{nh}||\A_{1h}^{-1}-\A_1^{-1}||.
\end{align}
Thus, it suffices to give a high probability bound of the term $||\A_{1h}^{-1}-\A_1^{-1}||$.

\noindent \textbf{Step 2} Please note that, when event $E_2:=\{||\A_1^{-1}\triangle_{1h}||<\frac{1}{2}\}$ happens, the inverse of  $\mathbf{I}+\A_1^{-1}\triangle_{1h}$ exists and 
\begin{align}
        \label{eq proof of lemma 4 eq4}
        ||(\mathbf{I}+\A_1^{-1}\triangle_{1h})^{-1}||\leq \sum_{k=1}^\infty ||\A_1^{-1}\triangle_{1h}||^k=\frac{1}{1-||\A_1^{-1}\triangle_{1h}||}\leq 2.
\end{align}
Meanwhile, since $ \A_{1h}^{-1}=(\mathbf I+\A_{1}^{-1}\triangle_{1h})^{-1}\A_{1}^{-1}$ implies
\begin{align}
    \label{eq proof of lemma 4 eq5}
    \A^{-1}_{1h}-\A^{-1}_{1}=((\mathbf I+\A_{1}^{-1}\triangle_{1h})^{-1}-\mathbf{I})\A_{1}^{-1}=-(\mathbf I+\A_{1}^{-1}\triangle_{1h})^{-1}(\A_1^{-1}\triangle_{1h})\A_{1}^{-1},
\end{align}
combining \eqref{eq proof of lemma 4 eq4} and \eqref{eq proof of lemma 4 eq5} asserts 
\begin{align}
     \label{eq proof of lemma 4 eq6}
     || \A^{-1}_{1h}-\A^{-1}_{1}||&\leq \frac{||A_1||^2||\triangle_{1h}||}{1-||\A_1^{-1}\triangle_{1h}||}\leq 2||A_1||^2||\triangle_{1h}||=\frac{2||\triangle_{1h}||}{f^2_X(0)\lambda^2_{\min}(\Gamma_1)}\notag\\
     &\leq \frac{||\A_{1h}-\E[\A_{1h}]||+L_f(S+1)h}{f^2_X(0)\lambda^2_{\min}(\Gamma_1)},
\end{align}
where, according to Assumption \ref{as 1}, the last inequality is based on inequality
\begin{align*}
     ||\triangle_{1h}||&\leq ||\A_{1h}-\E[\A_{1h}]||+||\E[\A_{1h}]-\A_1||\\
    &\leq ||\A_{1h}-\E[\A_{1h}]||+L_f(S+1)h.
\end{align*}

\noindent\textbf{Step 3} Since
\begin{align*}
    ||\A_1^{-1}\triangle_{1h}||\leq ||\A_{1}^{-1}||\ ||\triangle_{1h}||\leq \frac{||\A_{1h}-\E[\A_{1h}]||+L_f(S+1)h}{f_X(0)\lambda_{\min}(\Gamma_1)},
\end{align*}
 by denoting $E_3:=\{||\A_{1h}-\E[\A_{1h}]||\leq \epsilon\}$,  it is obvious that $E_3\subset E_2$ holds for all $\epsilon\in \{\epsilon>0: \epsilon+L_f(S+1)h\leq \frac{1}{2}f_X(0)\lambda_{\min}(\Gamma_1)\}=:\mathcal{E}$. Additionally, according to the notations introduced in Step 1, we have $\lambda_{\min}(\A_1)=f_X(0)\lambda_{\min}(\Gamma_1)$, which indicates that $  E_3\subset E_1$ holds for all $\epsilon\in\mathcal{E}$ as well. Therefore, we have
\begin{align}
    \label{eq proof of lemma 4 eq7}
    E_3\subset E_1\cap E_2.
\end{align}

\noindent\textbf{Step 4} Note that 
\begin{align*}
    \A_{1h}=\sum_{i=1}^n\frac{1}{nh}K_{ih}r(\frac{X_i}{h})r(\frac{X_i}{h})^\top:=\sum_{i=1}^{n}\mathbf Z_{i}
\end{align*}
and $\mathbf{Z}_i$'s are independent $\mathbb{R}^{(S+1)\times (S+1)}$-valued symmetric random matrix such that $||\mathbf{Z}_i-\E[\mathbf{Z}_i]||_{op}\leq 2||K||_{\infty}(S+1)=2(S+1)$, where $||\cdot||_{op}$ is the operator norm. Thus, by combining Matrix Bernstein inequality (e.g., Theorem 1.6.2 in \citeA{tropp2015introduction}), \eqref{eq proof of lemma 4 eq3}, \eqref{eq proof of lemma 4 eq6} and \eqref{eq proof of lemma 4 eq7}, we finish the proof.\\

\noindent \textbf{Proof of Lemma \ref{lemma 5}} The proof of \eqref{eq lemma 5 eq1} is divided into multiple steps.\\ 
\noindent\textbf{Step 1} (some math facts) In this step, we present some important mathematical facts which will be used in the later steps.
\begin{itemize}
    \item [F1] For $k=1,2$, we have
    \begin{align*}
        |\E[\M_{kh}]-\M_{k}|=|\E[\M_{kh}]-nhf_X(0)\Gamma_{k}|\leq L_{f}(S+1)nh^2\mu_{k},
    \end{align*}
    where $\mu_k=\int_{-1}^1K^k(u)du$.
    \item [F2] For each $k=1,2$, there exists a constant $d_k$ depending only on $S,K,f_X(0), \Gamma_k$ such that the following inequality holds for all $n\geq 1$ and $\epsilon>0$, 
    \begin{align*}
        \P(|\M_{kh}-\E[\M_{kh}]|\geq \epsilon nh)\leq 2(S+1)\exp(-d_knh\{\epsilon^2,\epsilon\}).
    \end{align*}
    \item [F3] According to the assumption that $\Gamma_1$ is invertible (implying $\M_{1}^{-1}$ exists), when event $\{\lambda_{\min}(\M_{ih})\}$ happens, we have
    \begin{align*}
        &\Big|\sum_{i=1}^nW^2_{ih}(0)-\frac{1}{nhf_X(0)}\int_{-1}^1l^2_S(u)du\Big|\\
        \leq &||\M_{1h}^{-1}||\ ||\M_{2h}-\M_2||+ (||\M_{1h}^{-1}||+||\M_1^{-1}||)\ ||\M_2||\ ||\M_{1}^{-1}||\ ||\M_{1h}^{-1}||\  ||\M_{1h}-\M_{1}||. 
    \end{align*}
\end{itemize}
F1 is just a natural consequence of some simple algebra and F2 is a direct corollary of Theorem 1.6.2 of \citeA{tropp2015introduction}. Hence, we only prove F3 in this step. Actually, to prove F3, we only need to notice that 
\begin{align*}
    \triangle:=&\sum_{i=1}^nW_{ih}^2(0)-\frac{1}{nhf_X(0)}\int_{-1}^1l^2_S(u)du=\e_0^{\top}(\M^{-1}_{1h}\M_{2h}\M^{-1}_{1h}-\M_{1}^{-1}\M_2\M_{1}^{-1})\e_{0}\\
    =&\e_0^{\top}\M^{-1}_{1h}(\M_{2h}-\M_2)\M^{-1}_{1h}+\e_{0}^{\top}(\M^{-1}_{1h}\M_{2}(\M^{-1}_{1h}-\M_1^{-1})+(\M_{1h}^{-1}-\M^{-1}_{1})\M_2\M^{-1}_{1})\e_0,
\end{align*}
which asserts
\begin{align*}
    |\triangle|\leq ||\M_{1h}^{-1}||^2\ ||\M_{2h}-\M_2||+ (||\M_{1}^{-1}||+||\M_{1h}^{-1}||)\ ||\M_2||\ ||\M_{1h}^{-1}-\M_{1}^{-1}||.
\end{align*}
Then, together with the fact that $\M_{1h}^{-1}-\M_{1}^{-1}=\M_{1h}^{-1}(\M_{1h}-\M_{1})\M_{1}^{-1}$, we finish the proof F3.

\noindent\textbf{Step 2} For any given $\epsilon>0$, define event
\begin{align*}
  & \ \ \ \  \ \  \ \ \  \ \ \ \ \ \ \ \   \mathcal{E}=\mathcal{E}_{\min}\cap \mathcal{E}_{1}(\epsilon)\cap \mathcal{E}_{2}(\epsilon),\\
  \text{where} \ \    &\mathcal{E}_{\min}=\{\lambda_{\min}(\M_{1h})\geq \frac{nhf_{X}(0)}{4}\lambda_{\min}(\Gamma_{1})\},\\
    \mathcal{E}_{1}(\epsilon)=\{||&\M_{1h}-\E[\M_{1h}]||\leq \epsilon nh\},\ \mathcal{E}_{2}(\epsilon)=\{||\M_{2h}-\E[\M_{2h}]||\leq \epsilon nh\}.
\end{align*}
Then, Lemma \ref{lemma 3} and F2 imply that 
\begin{align}
    \label{eq proof of lemma 5 eq1}
    \P(\mathcal{E}^c)\leq (S+1)e^{-c_1nh}+4(d+1)\exp\Big(-(d_1\lor d_2)nh\min\{\epsilon^2,\epsilon\}\Big).
\end{align}
Additionally, when event $\mathcal{E}$ happens, F1 implies
\begin{align}
       \label{eq proof of lemma 5 eq2}
       ||\M_{kh}-\M_{k}||\leq ||\M_{kh}-\E[\M_{kh}]||+||\E[\M_{kh}]-\M_k|| \leq (\epsilon +L_f(S+1)\mu_k)nh,
\end{align}
and we also have $||\M_{1h}^{-1}||\leq \frac{1}{\lambda_{\min}(\M_{1h})}\leq \frac{4}{nhf_X(0)\lambda_{\min}(\Gamma_{1})}$. Recall that $||\M_{1}^{-1}||\leq \frac{1}{nhf_{X}(0)\lambda_{\min}(\Gamma_{1})}$ and $||\M_2||=nhf_{X}(0)||\Gamma_2||$.

Above all, according to point F1 in Step 1,  there exists a $c_4,C^*>0$ depends only on $(S,K,f_X(0),\Gamma_1,\Gamma_2)$ such that 
\begin{align}
    |\triangle|\leq \frac{C^*}{nh}(\epsilon+h),
\end{align}
once event $\mathcal{E}$ happens. Thus, using \eqref{eq proof of lemma 5 eq1} finishes the proof of \eqref{eq lemma 5 eq1}. 

As for the proof of \eqref{eq lemma 5 eq2}, we only need to notice that 
\begin{align*}
    \sum_{i=1}^nW^2_{ih}(0)\leq \frac{\lambda_{\max}(\M_{2h})}{\lambda^2_{\min}(\M_{1h})}.
\end{align*}
Then, together with the definitions of $\mathcal{E}_{\min}$ and $\mathcal{E}_{\max}$, some direct algebra yields \eqref{eq lemma 5 eq2}.\\

\noindent\textbf{Proof of Lemma \ref{lemma 6}} The proof is divided into multiple steps. Before we start, we first introduce the following result, which is a natural consequence of Lemma \ref{lemma 4}, \eqref{eq lemma 4 eq2}. For all $n\geq 1$, $h\leq H_0$ and $0<\epsilon\leq \epsilon_0$, there exists $d_1,d_2,C_{\triangle}>0$ independent of sample size $n$ and bandwidth $h$ such that
\begin{align}
    \label{eq proof of lemma 6 eq1}
   \P(\mathcal{G}_{\epsilon}):= \P\Big(\max_{1\leq i\leq n}|W_{ih}(0)-\frac{1}{nhf_X(0)}l(\frac{X_i}{h})|\leq \frac{C_{\triangle}(h+\epsilon)}{nh}\Big)\geq 1-d_1\exp(-d_2nh\min\{\epsilon^2,\epsilon\}),
\end{align}
where $H_0$ and $\epsilon_0$ are introduced in Lemmas \ref{lemma 3} and \ref{lemma 6}.

\noindent \textbf{Step 1} (Decomposition) By defining $L_i=\frac{1}{nhf_X(0)}l(\frac{X_i}{h})$ and $\triangle_i=W_{ih}(0)-L_i$, we have the following basic decomposition
\begin{align*}
    T_n:&=nh\sum_{i=1}^nW^2_{ih}(0)V(X_i)=nh\sum_{i=1}^n(L_i+\triangle_i)^2V(X_i)\\
    &=nh\sum_{i=1}^nL_i^2V(X_i)+\underbrace{nh\sum_{i=1}^n\triangle_i^2V(X_i)+2nh\sum_{i=1}^nL_i\triangle_iV(X_i)}_{=:r_n}=:\widetilde{T}_n+r_n.
    \end{align*}
By letting $T_0=\frac{V(0)}{f_X(0)}\int_{-1}^1l^2_S(u)du$, we have 
\begin{align}
\label{eq proof of lemma 6 eq2}
    T_n-T_0=r_n+(\widetilde{T}_n-\E[\widetilde{T}_n])+(\E[\widetilde{T}_n]-T_0).
\end{align}

\noindent\textbf{Step 2} ($r_n$) Since $l_S(u)\neq 0$ if and only if $|u|\leq 1$, we immediately have identity
\begin{align}
    \label{eq proof of lemma 6 eq3}
    r_n=nh\sum_{i=1}^n(2L_i\triangle_i+\triangle_i^2)1[|X_i|\leq h]V(X_i).
\end{align}
Thus, by introducing random variable $N_{h}=\sum_{i=1}^n1[|X_i|\leq h]$ and event $\mathcal{G}_{h}=\{N_h\leq 6f(0)nh\}$, using Bernstein inequality yields that
\begin{align}
  \label{eq proof of lemma 6 eq4}
    \E[N_h]=n\P(|X_1|\leq h)\leq 3f_X(0)nh\ \text{and}\  \P(\mathcal{G}_h)\geq 1-\exp\Big(-\frac{f^2_X(0)nh}{8}\Big)
\end{align}
hold for all $h\leq H_0$ and $n\geq 1$. Recall that $\sup_{u\in [-1,1]}|l_S(u)|=:l_{\infty}<\infty$ and Assumption \ref{as 2} implies $\sup_{u\in [-1,1]}|V(u)|=:\bar V<\infty$. Then, for all $h\leq H_0$ and $\epsilon\leq \epsilon_0$, when event $\mathcal{G}_h\cap \mathcal{G}_{\epsilon}$ happens, there exists some constant $C_r>0$ independent of $n$ and $h$ such that
\begin{align*}
    |r_n|&\leq nhN_h\bar V(2\max_{i\leq n}|L_i|\max_{i\leq n}|\triangle_i|+(\max_{i\leq n}|\triangle_i|)^2)\\
    &\leq 6f_X(0)(nh)^2(2\frac{l_{\infty}C_{\triangle}(h+\epsilon)}{(nh)^2}+\frac{C_{\triangle}^2(h+\epsilon)^2}{(nh)^2})\\
    &\leq C_{r}(h+\epsilon),
\end{align*}
where the last inequality is because $\epsilon<\epsilon_0$ implies $h+\epsilon\leq 1$. Above all, combining \eqref{eq proof of lemma 6 eq1} and \eqref{eq proof of lemma 6 eq4} yields that
\begin{align}
\label{eq proof of lemma 6 eq5}
    \P(|r_n|\leq C_r(h+\epsilon))\geq 1-\exp(-\frac{f_X^2(0)nh}{8})-d_1\exp(-d_2nh\min\{\epsilon^2,\epsilon\})
\end{align}
holds for all $n\geq 1$, $\epsilon<\epsilon_0$ and $h<H_0$.

\noindent\textbf{Step 3} ($\widetilde{T}_n-\E[\widetilde{T}_n]$) Note that
\begin{align*}
    \widetilde{T}_n=nh\sum_{i=1}^nL_i^2V(X_i)=\frac{1}{n}\sum_{i=1}^n\frac{l_S^2(X_i/h)}{hf^2_{X}(0)}V(X_i)=:\frac{1}{n}\sum_{i=1}^n\xi_{i}
\end{align*}
and $\xi_i$'s are mutually independent. Meanwhile, we have
\begin{align*}
    |\xi_i|\leq \frac{l_\infty^2\bar V}{hf^2_X(0)}=:\frac{B_{\xi}}{h},\ \ \E|\xi^2_{i}|=\E[\xi^2_{i}1[|X_i|\leq h]]\leq \frac{3B_{\xi}^2f_X(0)}{2h}=:\frac{V_{\xi}}{h}.
\end{align*}
Thus, there exists a $d_3>0$ independent of $n$ and $h$ such that 
\begin{align}
    \label{eq proof of lemma 6 eq6}
    \P(|\widetilde{T}_n-\E[\widetilde{T}_n]|>t)\leq \exp(-d_3nh\min\{t^2,t\}),\ \forall\ t>0.
\end{align}

\noindent\textbf{Step 4}($\E[\widetilde{T}_n]-T_0$) According to the notations introduced in Step 3, some simple algebra yields that
\begin{align}
\label{eq proof of lemma 6 eq7}
   | \E[\widetilde{T}_n]-T_0|&=\Big|\frac{1}{hf^2_X(0)}\E[l_S(\frac{X}{h})V(X)]-\frac{1}{f^2_X(0)}\int_{-1}^1l_S^2(u)V(0)f_X(0)du\Big|\notag\\
    &=\frac{1}{f^2_X(0)}\int_{-1}^1l_S^2(u)|V(uh)f_X(uh)-V(0)f_X(0)|du\notag\\
    &\leq \Big(\frac{1}{f^2_X(0)}(\frac{3}{2}L_Vf_X(0)+V(0)L_f)\int_{-1}^1l^2_S(u)du\Big)h=:c_Bh,\ \forall\ h<H_0.
\end{align}

Finally, according to \eqref{eq proof of lemma 6 eq5}-\eqref{eq proof of lemma 6 eq7}, there exists $c_5,c_6,c_7>0$ independent of $n$ and $h$ such that 
\begin{align*}
   \P( |T_n-T_0|\leq c_5h+\epsilon)\geq 1-\exp(-\frac{f_X^2(0)nh}{8})-c_6\exp(-c_7nh\min\{\epsilon^2,\epsilon\})
\end{align*}
holds for all $n\geq 1$, $h<H_0$ and $\epsilon<\epsilon_0$, which completes the proof.\\

\noindent\textbf{Proof of Lemma \ref{lemma 7}} We first focus on proving the case where $0$ is the interior.i.e., The support of $X$ is $[-1,1]$. 
\noindent\textbf{Step 1} By denoting $\mathcal F_{n}=\sigma(X_1,...,X_n)$, Lemma \ref{lemma 1} implies the following important identity.
\begin{align*}
    &\sum_{i=1}^nW_{ih}^2(0)V(X_i)(\varepsilon_i^2-1)= \sum_{i=1}^nW_{ih}^2(0)V(X_i)(\varepsilon_i^2-\E[\varepsilon_i^2|\mathcal{F}_n])\\
    =&\sum_{i=1}^nW_{ih}^2(0)V(X_i)(\varepsilon_{iL_n}^2-\E[\varepsilon_{iL_n}^2|\mathcal{F}_n])+\sum_{i=1}^nW_{ih}^2(0)V(X_i)(\varepsilon_{iL^-_n}^2-\E[\varepsilon_{iL^-_n}^2|\mathcal{F}_n])\\
    =&:\mathbf{W}_1+\mathbf{W}_2,
\end{align*}
where $\varepsilon_{iL_n}=\varepsilon_{i}1[|\varepsilon_i|\leq L_n]$, $\varepsilon_{iL_n}=\varepsilon_{i}1[|\varepsilon_i|> L_n]$ and $L_n=n^{\theta}$ for some $\theta>0$ specified in later steps. More specifically, we only need to prove Assumption \ref{as 3} asserts the existence of a non-empty admissible set for such $\theta$.

\noindent\textbf{Step 2} ($\mathbf{W}_2$) Based on the $\mathcal{E}_{\max}$ and $\mathcal{E}_{\min}$ defined in Lemma \ref{lemma 3}, combining formula \eqref{eq lemma 5 eq2}, Assumption \ref{as 3}, Markov inequality and Lemmas \ref{lemma 1}, \ref{lemma 3}, we have 
\begin{align*}
    \P(|\mathbf{W}_2|>t)&\leq \E[1[\mathcal{E}_{\max}\cap \mathcal{E}_{\min}]\P(|\mathbf{W}_2|>t|\mathcal{F}_n)]+\P(\mathcal{E}_{\max}^c)+\P(\mathcal{E}_{\min}^c)\\
    &\leq 2t^{-1}\E\left[\sum_{i=1}^nW^2_{ih}(0)1[\mathcal{E}_{\max}\cap \mathcal{E}_{\min}]\E[\varepsilon^2_{iL_n^-}]\right]+(S+1)\exp(-(c_1\lor c_2)nh)\\
    &\leq \frac{2E_{\varsigma}\E[\sum_{i=1}^nW^2_{ih}(0)1[\mathcal{E}_{\max}\cap \mathcal{E}_{\min}]]}{tn^{\theta(\varsigma-2)}}+(S+1)\exp(-(c_1\lor c_2)nh)\\
    &\leq \frac{12E_{\varsigma}\lambda_{\max}(\Gamma_2)}{nhtf_X(0)\lambda^2_{\min}(\Gamma_1)n^{\theta(\varsigma-2)}}+(S+1)\exp(-(c_1\lor c_2)nh),
\end{align*}
where $c_1$ and $c_2$ are introduced in Lemma \ref{lemma 3}. Moreover, letting $t=\frac{(\log n)^{-\beta}}{nh}$ yields
\begin{align*}
    \P\Big(|\mathbf{W}_2|>\frac{(\log n)^{-\beta}}{2nh}\Big)\leq \frac{d_1(\log n)^{\beta}}{n^{\theta(\varsigma-2)}}+(S+1)\exp(-(c_1\lor c_2)nh),
\end{align*}
where $d_1>0$ is independent of $n$ and $h$. Using Assumption \ref{as 3} and condition $h=O(n^{-\frac{1}{2S+1}})$ yields the following implication, 
\begin{align}
    \label{eq proof of Lemma 7 eq1}
   \theta>\frac{2S}{(2S+1)(\varsigma-2)} \Rightarrow   \P\Big(|\mathbf{W}_2|\leq\frac{(\log n)^{-\beta}}{2nh}\Big)\geq 1-\frac{d_{\beta}}{nh},
\end{align}
where $d_\beta>0$ is independent of $n$ and $h$. 

\noindent\textbf{Step 3}($\mathbf{W}_1$) According to Lemma \ref{lemma 5}, we know there exists some $d_2>0$ independent of $n$ and $h$ such that 
\begin{align*}
  \P(\mathcal{E}_{W^2}):=  \P(\sum_{i=1}^nW^2_{ih}(0)\leq\frac{d_2}{nh})\geq1- 5(S+1)e^{-(c_1\lor c_4)nh}
\end{align*}
holds for all $h<H_0$. Together with the event $\mathcal{E}_{\max,W}$ introduced in Lemma \ref{lemma 4}, for every $h<H_0$ and $t>0$, we have
\begin{align*}
    &\P(|\mathbf{W}_1|>t)\leq \E[1[\mathcal{E}_{\max,W}\cap \mathcal{E}_{W^2}]\P(|\mathbf{W}_1|>t|\mathcal{F}_n)]+\P(\mathcal{E}^c_{\max,W})+\P(\mathcal{E}_{W^2}^c)\\
    \leq&\E\left[1[\mathcal{E}_{\max,W}\cap \mathcal{E}_{W^2}]\exp\Big(-\frac{t^2}{\sum_{i=1}^nW_{ih}^4(0)\E[\varepsilon_i^4|\mathcal{F}_n]+\frac{t}{3}(\max_{i\leq n}|W_{ih}(0)|)^2n^{2\theta}}\Big)\right]+ 9(S+1)e^{-(c_1\lor c_3\lor c_4)nh},
\end{align*}
where $c_3$ is from \eqref{eq lemma 4 eq1}. Then, when $t=\frac{(\log n)^{-\beta}}{2nh}$ and $h=O(n^{-\frac{1}{2S+1}})$ some simple algebra yields that 
\begin{align}
  \label{eq proof of Lemma 7 eq2}
    \P\Big(|W_1|>\frac{(\log n)^{-\beta}}{2nh}\Big)&\leq \exp\Big(-d_3\frac{(\log n)^{\beta}}{n^{2\theta-\frac{2S}{2S+1}}}\Big)+9(S+1)e^{-(c_1\lor c_3\lor c_4)nh}
\end{align}
holds for all $n\geq 1$ and $h<H_0$. An immediate consequence of \eqref{eq proof of Lemma 7 eq2} is 
\begin{align}
        \label{eq proof of Lemma 7 eq3}
        \theta\leq \frac{S}{2S+1}\Rightarrow \P\Big(|W_1|\leq \frac{(\log n)^{-\beta}}{2nh}\Big)\geq e^{-d_3(\log n)^{\beta}}+9(S+1)e^{-(c_1\lor c_3\lor c_4)nh}.
\end{align}
 Note that the admissible set for $\theta$ satisfying \eqref{eq proof of Lemma 7 eq1} and \eqref{eq proof of Lemma 7 eq3} simultaneously is 
\begin{align*}
    \Big\{\frac{2S}{(2S+1)(\varsigma-2)}<\theta\leq \frac{S}{2S+1}\Big\}.
\end{align*}
Obviously, Assumption \ref{as 3} implies that this is not an empty set. Thus,  for the case where $0$ is interior, we finish the proof of \eqref{eq lemma 7 eq1} by letting $\theta=\frac{S}{2S+1}$. As for the boundary case, the proof strategy is nearly the same except replacing the applications of Lemmas \ref{lemma 3}-\ref{lemma 5} with Corollaries \ref{corollary a1}-\ref{corollary a3}.\\

\noindent\textbf{Proof of Lemma \ref{lemma 8}} Lemma \ref{lemma 8} is standard for local polynomial regression. Its proof is a simple combination of Lipschitz continuity and moment conditions assumed in Assumptions \ref{as 2} and \ref{as 3}, the techniques used in the proof of Lemma \ref{lemma 7} and the reproduction property of local polynomial regression. We thus omit the details.\\

\noindent\textbf{Proof of Corollary \ref{corollary a1}} The proof of Corollary \ref{corollary a1} is nearly the same. We thus only highlight the difference. Actually, based on the $W_{k,i}$ introduced in the proof of \ref{lemma 3}, some simple algebra yields that, when the evaluation point is $0$ and the support is $[0,1]$,  
\begin{align}
   & \lambda_{\min}(\E[\M_{1h}])=n\lambda_{\min}(\E[W_{1,1}])\geq \frac{1}{2}nhf_X(0)\lambda_{\min}(\Gamma_1'),\\
  &  \lambda_{\max}(\E[\M_{2h}])=n\lambda_{\max}(\E[W_{1,1}])\leq \frac{3}{2}nhf_X(0)\lambda_{\max}(\Gamma_{2}').
\end{align}
Then, repeating the arguments used in the proof of Lemma \ref{lemma 3} can finish the proof.\\

\noindent\textbf{Proof of Lemma \ref{lemma 9}} According to \eqref{eq th refinement eq3}, when $0$ is an interior point, $C_V=\frac{V(0)}{f_X(0)}\int_{-1}^1l^2_S(u)du$, where $l_S(u)$ is defined in Lemma \ref{lemma 4}. By letting 
\[
C_{Vg,0}=ng\sum_{i=1}^nW^2_{ig}(0)(Y_i-m(0))^2, \qquad C_{Vg}=ng\sum_{i=1}^nW_{ig}^2(0)(Y_i-m(X_i))^2,
\]
an immediate basic triangle inequality is as follow,
\begin{align*}
    |\hat C_{V,0}-C_{V}|&\leq |\hat C_{V,0}-C_{V,g}|+|C_{Vg,0}-C_{Vg}|+|C_{Vg}-C_V|.
\end{align*}
It suffices to deliver a high-probability bound to each term of right side of this inequality. 

\noindent\textbf{Step 1} ($|C_{Vg,0}-C_{Vg}|$) By denoting $\triangle_{m,i0}=m(X_i)-m(0)$, some simple algebra yields 
\begin{align*}
    |C_{Vg,0}-C_{Vg}|&=ng\Big|\sum_{i=1}^nW_{ig}^2(0)[(Y_i-m(X_i))^2-(Y_i-m(0))^2]\Big|\\
    &\leq 2ng\Big|\sum_{i=1}^nW_{ig}^2(0)(m(X_i)-m(0))Y_i\Big|+2ng||m||_{\infty}\Big|\sum_{i=1}^nW_{ih}^2(0)(m(X_i)-m(0))\Big|\\
    &\leq 2ng\Big|\sum_{i=1}^nW_{ig}^2(0)\triangle_{m,i0}V^{\frac{1}{2}}(X_i)\varepsilon_i\Big|+4ng||m||_{\infty}\Big|\sum_{i=1}^nW_{ig}^2(0)\triangle_{m,i0}\Big|=:\mathbf{W}_1+\mathbf{W}_2,
\end{align*}
Note that $W_{ig}(0)=W_{ig}(0)1[|X_i|\leq g]$, using the Lipschitz condition in Assumption \ref{as 2} asserts 
\begin{align*}
\mathbf{W}_2\leq4ng^2||m||_{\infty} L_m \sum_{i=1}^nW_{ig}^2(0).
\end{align*}
Then, together with Lemma \ref{lemma 5}, we know there exists some constants $M,\kappa>0$ independent of $n,g$ such that the following inequality holds for all $n\ge 1$ and $g\leq H_0$, 
\begin{align}
\label{eq proof of lemma 9 eq1}
    \P\Big(\mathbf{W}_2\leq Mg\Big)\geq 1-e^{-n^{\kappa}}.
\end{align}
For $\mathbf{W}_1$, by letting $\varepsilon_{il_n}=\varepsilon_{i}1[|\varepsilon_i|\leq l_n]$, $\varepsilon_{il_n^{-}}=\varepsilon_i1[|\varepsilon|>l_n]$ and $l_n=n^{\frac{S}{2S+1}}/\log^\beta n$, where $\beta>0$ is user-defined. Together with the fact that $\E[\varepsilon_{il_n}]+\E[\varepsilon_{il_n^{-1}}]=0$, we obtain 
\begin{align*}
    \mathbf{W}_1&\leq ng\Big|\sum_{i=1}^nW^2_{ig}(0)\triangle_{m,i0}(\varepsilon_{il_n}-\E[\varepsilon_{il_n}])\Big|+ng\Big|\sum_{i=1}^nW^2_{ig}(0)\triangle_{m,i_0}\varepsilon_{il_n^{-}}\Big|+ng^2L_m\sum_{i=1}^nW^2_{ig}(0)\E|\varepsilon_{il^-_n}|\\
    &=: \mathbf{W}_{11}+\mathbf{W}_{12}+\mathbf{W}_{13},
\end{align*}
where the third term on the right side is because $m$ is Lipschitz continuous with constant $L_m$ and $\triangle_{m,i0}=\triangle_{m,i0}1[|X_i|\leq g]$. For any given $t>0$, $n\geq 1$ and $g\leq H_0$, using Lemma \ref{lemma 3} and \eqref{eq lemma 5 eq2} yields 
\begin{align*}
    &\P(\mathbf{W}_{12}\geq t/2)+\P(\mathbf W_{13}\geq t/2)\\
    \leq &\sum_{i=1}^n\P(|\varepsilon_{i}|>l_n)+L_mg\frac{12\lambda_{\max}(\Gamma_2)}{tf_X(0)\lambda_{\min}(\Gamma_1)}\frac{E_{\varsigma}}{l^{\varsigma-1}_n}+\P(\mathcal{E}^c_{\min})+\P(\mathcal{E}^c_{\max})\\
    \leq & \frac{E_{\varsigma}}{l_n^{\varsigma-1}}\Big(\frac{n}{l_n}+\frac{12gL_m\lambda_{\max}(\Gamma_2)}{tf_X(0)\lambda_{\min}(\Gamma_1)}\Big)+2(S+1)\exp(-(c_1\lor c_2)ng),\ \forall\ n\geq 1, t>0,
\end{align*}
where $c_1,c_2>0$ are introduced in Lemma \ref{lemma 3}. This implies
\begin{align}
    \label{eq proof of lemma 9 eq2}
   & \P(\mathbf{W}_{12}+\mathbf{W}_{13}\leq g)\notag\\
    \geq &1-E_{\varsigma}\Big(\frac{\log^\beta n}{n^{\frac{ S}{2S+1}}}\Big)^{\varsigma-1}\Big(n^{\frac{S+1}{2S+1}}\log n+\frac{12L_m\lambda_{\max}(\Gamma_2)}{f_X(0)\lambda_{\min}(\Gamma_1)}\Big)+2(S+1)\exp(-(c_1\lor c_2)ng)\notag\\
  \geq &1-C(\log^\beta n)^{\varsigma}n^{-\frac{(\varsigma-2)S-1}{2S+1}}+ 2(S+1)\exp(-(c_1\lor c_2)n^{\frac{2S}{2S+1}}),\ \forall\ n\geq 1,
\end{align}
 where $C>0$ is a constant that depends only on $(L_m,\Gamma_{k},S,E_{\varsigma},f_{X}(0))$. For $\mathbf{W}_1$, combining Lemma \ref{lemma 2} and conditional Bernstein inequality yields
 \begin{align}
     \label{eq proof of lemma 9 eq3}
      \P(\mathbf{W}_{11}>t)\leq &2\E\left[1[\mathcal{E}_{\min}\cap\mathcal{E}_{\max}\cap\mathcal{E}_{\max,W}]\exp\left(-\frac{t^2}{g||V||_{\infty}\sum_{i=1}^nW_{ig}^4(0)+\frac{tl_n}{3}\max_{i\leq n}W_{ig}^2(0)g }\right)\right]\notag \\
       &\qquad + \P(\mathcal{E}_{\min}^c) +\P(\mathcal{E}^c_{\max})+\P(\mathcal{E}_{\max,W}).
 \end{align}
Together with Lemmas \ref{lemma 3}-\ref{lemma 4}, some simple algebra shows that 
\begin{align}
     \label{eq proof of lemma 9 eq4}
     \P(\mathbf{W}_{11}\leq g)\geq 1-C'\exp(-C''n^{\frac{2S-1}{2S+1}}),\ \forall\ n\geq 1,
\end{align}
where $C'.C''>0$ depend only on $(S,K,f_X(0),\Gamma_{k})$. Combining \eqref{eq proof of lemma 9 eq3} and \eqref{eq proof of lemma 9 eq4} yields
\begin{align}
     \P(\mathbf{W}_1\leq 2g)\geq 1-D(\log n)^{\varsigma\beta}n^{-\frac{(\varsigma-2)S-1}{2S+1}}, \forall\ n\geq 1,
\end{align}
where $D>0$ is a constant independent of $n,g$ and $\beta>0$ is a user-defined constant. Together with \eqref{eq proof of lemma 9 eq1}, we obtain 
\begin{align}
    \label{eq proof of lemma 9 E1}
    \P(|C_{Vg,0}-C_{Vg}|\leq (M+2)g)\geq 1- D(\log n)^{\varsigma \beta}n^{-\frac{(\varsigma-2)S-1}{2S+1}},\ \forall\ n\geq 1.
\end{align}

\noindent\textbf{Step 2}($|\hat C_{V,0}-C_{V,0}|$) By letting $\triangle_{m,0}=\hat m_{b}(0)-m(0)$, where $b=g=n^{-\frac{1}{2S+1}}$, we have 
\begin{align*}
    |\hat C_{V,0}-C_{V,0}|\leq& 2ng\Big|\sum_{i=1}^nW_{ig}^2(0)\triangle_{m,0}V(X_i)\varepsilon_i\Big|+4ng||m||_{\infty}|\triangle_{m,0}|\Big|\sum_{i=1}^nW_{ig}^2(0)\Big|\\
    =&: \mathbf{C}_1+\mathbf{C}_2
\end{align*}
Since $g=b=n^{-\frac{1}{2S+1}}$, combining Lemmas \ref{lemma 5} and \ref{lemma 10} yields the following inequality for all $n\geq 1$,
\begin{align}
    \label{eq proof of lemma 9 eq5}
    \P\left(\mathbf{C}_2\leq \frac{24D_n||m||_{\infty}\lambda_{\max}(\Gamma_2)}{n^{\frac{S}{2S+1}}f_X(0)\lambda_{\min}^2(\Gamma_1)}\right)\geq 1-2(S+1)\exp(-(c_1\lor c_2)n^{\frac{2S}{2S+1}})-c_{10}n^{-\frac{(\varsigma-3)S}{2S+1}},
\end{align}
where $D_n=\log n(\log\log n)$ and $\varsigma>4+\frac{1}{S}$ is introduced in Assumption \ref{as 3}.

To deliver a high probability bound of $\mathbf{C}_1$, we only need to notice that 
\begin{align*}
    \mathbf{C}_1&\leq 2ng\Big|\sum_{i=1}^nW_{ig}^2(0)\triangle_{m,0}V(X_i)\varepsilon_{il_n}\Big|+2ng\Big|\sum_{i=1}^nW_{ig}^2(0)\triangle_{m,0}V(X_i)\varepsilon_{il^-_n}\Big|\\
    &\leq 2ng|\triangle_{m,0}|||V||_{\infty}\frac{n^{\frac{S}{2S+1}}}{\log^\beta n}\sum_{i=1}^nW_{ig}^2(0)+2ng\Big|\sum_{i=1}^nW_{ig}^2(0)\triangle_{m,0}V(X_i)\varepsilon_{il^-_n}\Big|=:\mathbf{C}_{11}+\mathbf{C}_{12}.
\end{align*}
For $\mathbf{C}_{12}$, an immediate bound is 
\begin{align}
        \label{eq proof of lemma 9 eq6}
        \P(\mathbf{C}_{12}\leq t)\geq 1-\sum_{i=1}^n\P(|\varepsilon_i|>l_n) \geq 1-E_{\varsigma}\frac{(\log n)^{\beta\varsigma}}{n^{\frac{(\varsigma-2)S-1}{2S+1}}}, \forall\ n\geq 1, \forall\ t>0.
\end{align}
For $\mathbf{C}_{11}$, using Lemmas \ref{lemma 3}, \ref{lemma 5} \eqref{eq lemma 5 eq2} and \ref{lemma 10} yields 
\begin{align}
      \label{eq proof of lemma 9 eq7}
      \P\left(\mathbf{C}_{11}\leq \frac{12D_n||V||_{\infty}\lambda_{\max}(\Gamma_2)}{(\log n)^\beta f_X(0)\lambda_{\min}^2(\Gamma_1)}\right)\geq 1-2(S+1)e^{-(c_1\lor c_2)ng}-c_{10}n^{-\frac{(\varsigma-2)S-1}{2S+1}},
\end{align}
where $c_1,c_2$ and $c_{10}$ are introduced in Lemmas \ref{lemma 3} and \ref{lemma 10} respectively.  Combining \eqref{eq proof of lemma 9 eq5}--\eqref{eq proof of lemma 9 eq7} yields that, for any fixed $\beta>1$, we have
\begin{align}
\label{eq proof of lemma 9 E2}
\P\Big(|\hat C_{V,0}-C_{V,0}|\leq \frac{C}{(\log n)^{\beta-1}}\Big)\geq 1-C' (\log n)^{\beta\varsigma}n^{-\frac{(\varsigma-2)S-1}{2S+1}},\ \forall\ n\geq 1,
\end{align}
where $C,C'>0$ are fixed constants independent of $n,g,b$.\\

\noindent\textbf{Step 3}($|C_{V,g}-C_V|$) This step can be obtained immediately by using Lemma \ref{lemma 6} and \ref{lemma 7}, since
\begin{align*}
    C_{V,g}-C_{V}=\underbrace{\Big(ng\sum_{i=1}^nW_{ig}^2(0)V(X_i)-C_V\Big)}_{\text{Lemma 6}}+\underbrace{ng\sum_{i=1}^nW_{ig}^2(0)V(X_i)(\varepsilon_i^2-1)}_{\text{Lemma 7}}.
\end{align*}
More specifically, for any $\beta>0$, $n\ge 1$ and $g=b\leq H_0$, we have 
\begin{align}
    \label{eq proof of lemma 9 E3}
    \P(|C_{V,g}-C_V|\leq \frac{1}{\log n})\geq 1-C''(\log n)^{\beta}n^{-\frac{2S(\varsigma-2)}{2S+1}}.
\end{align}
Therefore, by combining \eqref{eq proof of lemma 9 E1}--\eqref{eq proof of lemma 9 E3}, we finish the proof.\\

\noindent \textbf{Proof of Lemma \ref{lemma 11}} The proof is divided into two parts.

\noindent\textbf{Part 1} (Lower bound of $\P(\mathcal{G}_2)$ and $\P(\mathcal{G}_3)$) We first introduce the following notations for convenience,
\begin{align*}
    \tilde{b}_K(x)=Q_K^{-1/2}b_K(x),\ \hat G_K(x)=\frac{1}{n}\sum_{i=1}^n\tilde{b}_K(X_i)\tilde{b}^\top_K(X_i),\ Z_i:=\tilde{b}_K(X_i)\tilde{b}^\top_K(X_i)-I_K
\end{align*}
where $Q_K$ is introduced in Assumption \ref{as 9}. Then, Assumption \ref{as 8} immediately implies 
\begin{align*}
    \zeta_K:=\sup_{x\in [-1,1]}||\tilde b_K(x)||\leq C_{\zeta}K^{1/2},\ ||Z_i||\leq 2\zeta_K^2,\ ||\sum_{i=1}^n\E[Z_iZ_i^\top]||\leq 2n\zeta_K^2,
\end{align*}
 where $C_\zeta>0$ is independent of $K$ and $n$. Together with Assumptions \ref{as 1} and \ref{as 8}--\ref{as 10}, applying Theorem 1.6 of \citeA{tropp2015introduction} asserts
 \begin{align}
 \label{eq proof of lemma 11 eq1}
     \P(\mathcal{G}_n):=\P\Big(||\hat G_K-I_K||\leq \frac{1}{2}\Big)=\P\big(||\frac{1}{n}\sum_{i=1}^nZ_i||\leq \frac{1}{2}\Big)\geq 1-2K\exp\Big(-\frac{3n}{56\zeta_K^2}\Big)
 \end{align}
holds for all $n\geq 1$. Meanwhile, since some simple algebra shows that $\sum_{i=1}^nW_{iK}^2(0)=\frac{1}{n}\tilde b_K(0) \hat G_K^{-1}\tilde b_K^\top(0)$, when event $\mathcal{G}_n$ defined in \eqref{eq proof of lemma 11 eq1} happens, we have $\sum_{i=1}^nW_{iK}^2(0)\leq \frac{2\zeta_K^2}{n}$, which asserts
\begin{align}
\label{eq proof of lemma 11 eq2}
    \P\Big(\sum_{i=1}^nW_{iK}^2(0)\leq \frac{C_{\zeta}K}{n}\Big)\geq 1-2K\exp\Big(-\frac{3n}{56C_{\zeta}K}\Big).
\end{align}

As for the lower bound of $\P(\mathcal{G}_3)$, please also note that, when $\mathcal{G}_n$ happens, $||\hat G^{-1}_K||\leq 2$, which implies 
\begin{align*}
     |W_{iK}(0)|\leq \frac{1}{n}|\tilde b_K(0)\hat G^{-1}_{K} \tilde b_K^\top(0)|\leq \frac{2C_{\zeta}K}{n},\ \forall\ i\leq n.
\end{align*}
Thus,
\begin{align}
    \label{eq proof of lemma 11 eq3}
     \P\Big(\max_{i\le n}|W_{iK}(0)|\leq \frac{2C_{\zeta}K}{n}\Big)\geq 1-2K\exp\Big(-\frac{3n}{56C_{\zeta}K}\Big).
\end{align}

\noindent \textbf{Part 2} (Lower bound of $\P(\mathcal{G}_{1})$) The proof of Part 2 is further divided into three steps.

\textbf{Step 1} By defining $\Xi_i=b_K(X_i)b_K^\top(X_i)-Q_K$, Assumption \ref{as 8} yields 
\begin{align*}
    \hat Q_K-Q_K=\frac{1}{n}\sum_{i=1}^n\Xi_i,\ ||\Xi_i||_{\infty}\leq C_RK,\ \Big\|\sum_{i=1}^n\E[\Xi_i\Xi_i^\top]\Big\|\leq C_{\sigma}nK,
\end{align*}
where $C_R,C_{\sigma}>0$ are constants independent of $n$ and $K$. Then, applying Theorem 1.6 of \citeA{tropp2015introduction} yields
\begin{align}
  \label{eq proof of lemma 11 eq4}
  \P(\mathcal{E}_Q):&=  \P\Big(||\hat Q_K-Q_K||\leq \frac{q_-}{2}\Big)\notag\\
  &\geq1-  2K\exp\Big(-\frac{(nq_-)^2}{8(C_{\sigma}nK+C_RKnq_-/6)}\Big)\notag\\
  &=:1-2K\exp\Big(-c_{G}\frac{n}{K}\Big).
\end{align}
Then, on condition of event $\mathcal{E}_G$ in \eqref{eq proof of lemma 11 eq4}, using Wely inequality obtains
\begin{align*}
    &\lambda_{\min}(\hat Q_K)\geq \lambda_{\min}(Q_K)-||\hat Q_K-Q_K||\geq \frac{q_-}{2},\\
   & \lambda_{\max}(\hat Q_K)\leq \lambda_{\max}(Q_K)+||\hat Q_K-Q_K||\leq q_++\frac{q_-}{2},
\end{align*}
which asserts
\begin{align}
 \label{eq proof of lemma 11 eq5}
      \frac{\lambda_{\max}(\hat Q_K)}{\lambda_{\min}(\hat Q_K)}\leq \frac{q_-/2+q_+}{q_-/2}=:\kappa^*<\infty.
\end{align}
Thus, together with Assumption \ref{as 9}, event $\mathcal{E}_Q$ implies following two points,
\begin{itemize}
    \item $\hat Q_K$ is a positive definite symmetric matrix and $\kappa^*$ defined in \eqref{eq proof of lemma 11 eq5} is finite.
    \item By denoting $\hat Q_K(k,l)$ as the $(k,l)$-th entry of matrix $\hat Q_K$, we have $\hat Q_K(k,l)=0$ if $|k-l|>M_0$, where $M_0>0$ is a fixed constant independent of $n$ and $K$.
\end{itemize}
By applying Lemma 7.2 of \citeA{chen2015optimal}, we obtain that there exists a constant $C_I>0$ independent of $n$ and $K$ such that $ ||\hat Q^{-1}_K||_{l_\infty}\leq C_I$ holds when $\mathcal{E}_Q$ happens, where, for matrix $\A$, $||A||_{l_\infty}=\sup_{||x||_{\infty}\leq 1}||\A x||$. This implies
\begin{align}
     \label{eq proof of lemma 11 eq6}
     \P(||\hat Q^{-1}_K||_{l^\infty}\leq C_I)\geq 1-2K\exp(-c_G (n/K)),\ \forall\ n\geq 1.
\end{align}

\textbf{Step 2} Define event 
\begin{align*}
    \mathcal{E}_{l_1}:=\Big\{\max_{1\leq k\leq K}\frac{1}{n}\sum_{i=1}^n |b_{Kk}(X_i)|\leq \frac{2C_L}{\sqrt{K}}\Big\},
\end{align*}
where $C_L$ is introduced in Assumption \ref{as 8}. Then, using Bernstein inequality asserts 
\begin{align}
      \label{eq proof of lemma 11 eq7}
      \P\Big(\mathcal{E}_{l_1}\Big)\geq 1-K\exp\Big(-c_1\frac{n}{K}\Big),
\end{align}
where $c_1>0$ is independent of $n$ and $K$.

\textbf{Step 3} By letting $a_K=\hat Q_K^{-1}b_K(0)$, we have $W_{iK}(0)=\frac{1}{n}a_K^\top b_K(X_i)$, which implies
\begin{align*}
    \sum_{i=1}^n|W_{iK}(0)|&=\frac{1}{n}\sum_{i=1}^n\Big|\sum_{k=1}^K a_{Kk}b_{Kk}(X_i)\Big|\leq ||a_K||_{1} \max_{k\leq K}\frac{1}{n}\sum_{i=1}^n|b_{Kk}(X_i)|\\
    &\leq ||\hat Q_K^{-1}||_{l_1}||b_{K}(0)||_1\max_{k\leq K}\frac{1}{n}\sum_{i=1}^n|b_{Kk}(X_i)|\\
    &=||\hat Q_K^{-1}||_{l_\infty}||b_{K}(0)||_1\max_{k\leq K}\frac{1}{n}\sum_{i=1}^n|b_{Kk}(X_i)|,
\end{align*}
where the last equality is due to the fact that $\hat Q_K^{-1}$ is symmetric. Then, together with Assumption \ref{as 8}, \eqref{eq proof of lemma 11 eq6} and \eqref{eq proof of lemma 11 eq7}, we immediately obtain 
\begin{align*}
    \P\Big(\sum_{i=1}^n|W_{iK}(0)|\leq 2C_IC_{\infty}C_L\Big)\geq 1-\P(\mathcal{E}_Q^c)-\P(\mathcal{E}_{l_1}^c)\ge 1-2K\exp\Big(-(c_1\lor c_G)\frac{n}{K}\Big).
\end{align*}
Thus, we finish the proof of Lemma \ref{lemma 11}.\\

\noindent The proof of Corollaries \ref{corollary a1}-\ref{corollary a4} are basically a repetition of the proof of Lemmas \ref{lemma 3}-\ref{lemma 6} by replacing $\Gamma_1$ with $\Gamma_1'$, we thus omit them here.


\end{document}